\documentclass[12pt,a4paper]{amsart}

\usepackage{amsmath}
\usepackage{amsfonts}
\usepackage{amssymb}
\usepackage{amsthm}
\usepackage[margin=2cm]{geometry}
\usepackage{amsmath}
\usepackage{amsfonts}
\usepackage{amssymb}
\usepackage{amsthm}
\usepackage{geometry}

\usepackage{mathtools}
\usepackage[dvipsnames]{xcolor}
\usepackage{url}
\usepackage{enumitem}
\usepackage{float}
\usepackage{graphicx}
\usepackage[dvipsnames]{xcolor}
\usepackage{multirow}
\usepackage[hypertexnames=false]{hyperref}

\usepackage{subcaption}

\newcommand{\RR}{\mathbb{R}}

\newcommand{\NN}{\mathbb{N}}

\newcommand{\eps}{\varepsilon}

\newcommand{\dd}{\mathrm{d}}

\newcommand{\x}{\mathbf{x}}

\newcommand{\uj}{\mathbf{u}_j}
\newcommand{\ujh}{\mathbf{u}_{j,h}}
\newcommand{\uz}{\mathbf{u}_0}
\newcommand{\uzh}{\mathbf{u}_{0,h}}
\newcommand{\aaa}{{\boldsymbol{\alpha}}}

\newtheorem{theorem}{Theorem}[section]
\newtheorem{corollary}[theorem]{Corollary}
\newtheorem{proposition}[theorem]{Proposition}
\newtheorem{lemma}[theorem]{Lemma}

\theoremstyle{definition}

\newtheorem{remark}[theorem]{Remark}
\numberwithin{equation}{section}

\usepackage[backend=bibtex,giveninits=true,style=numeric,citestyle=numeric, doi=false,isbn=false,url=false,eprint=false,maxbibnames=99]{biblatex}
\bibliography{numerical-eigenvalues} 

\usepackage{amsaddr}

\title[Mesh-free method for Dirichlet eigenpairs of Laplacian with potential]{Mesh-free numerical method for Dirichlet eigenpairs of the Laplacian with potential}
\author{Dragoș Manea$^*$}
\address{\small ``Simion Stoilow"  Mathematical Institute of the Romanian Academy, 21~Calea Griviței, 010702~Bucharest, Romania}
\thanks{$^*$E-mail addresses: \texttt{dmanea28@gmail.com}, \texttt{dmanea@imar.ro}; \href{https://orcid.org/0000-0003-4085-226X}{ORCID ID: 0000-0003-4085-226X}}

\subjclass[2020]{65N25, 35P15, 65N35}
\keywords{Method of Particular Solutions, Laplace operator with potential, Numerical approximation of eigenpairs}
\begin{document}

\begin{abstract}
This paper is concerned with the numerical approximation of the $L^2$ Dirichlet eigenpairs of the operator $-\Delta + V$ on a simply connected $C^2$ bounded domain $\Omega \subset \mathbb{R}^2$ containing the origin, where $V$ is a radial potential.

We propose a mesh-free method inspired by the Method of Particular Solutions for the Laplacian (i.e. $V=0$). Extending this approach to general $C^1$ radial potentials is challenging due to the lack of explicit basis functions analogous to Bessel functions. To overcome this difficulty, we consider the equation $-\Delta u + V u = \lambda u$ on a ball containing $\Omega$, without imposing boundary conditions, for a collection of values $\lambda$ forming a fine discretisation of the interval in which eigenvalues are sought. By rewriting the problem in polar coordinates and applying a Fourier expansion with respect to the angular variable, we obtain a decoupled system of ordinary differential equations. These equations are solved numerically using a one-dimensional Finite Element Method, yielding a family of basis functions that are solutions of the equation $-\Delta u + V u = \lambda u$ on the ball and are independent of the domain $\Omega$.

Dirichlet eigenvalues of $-\Delta + V$ are then approximated by minimising the boundary values on $\partial \Omega$ among linear combinations of the basis functions and identifying those values of $\lambda$ for which the computed minimum is sufficiently small. The proposed method is highly memory-efficient compared to the standard Finite Element approach.
\end{abstract}
\maketitle
\section{Introduction and main results}
\label{sec:introduction}

\subsection{Presentation of the problem}
\label{sec:presentation-of-problem}
This paper is concerned with the numerical approximation of solutions to the Helmholtz equation with a radial potential on planar domains and homogeneous Dirichlet boundary conditions. More precisely, we focus on the numerical solution of the following eigenvalue problem:
\begin{equation}\label{EVP-radial}\tag{EVP}\begin{cases}
-\Delta u(\x) + V(\|\x\|)u(\x)=\lambda u(\x), & \x\in \Omega;\\
u(\x)=0, & \x\in \partial\Omega,
\end{cases}\end{equation}
where $\Omega \subset \mathbb{R}^2$ is a simply connected $C^2$ bounded domain containing the origin $O$. Since $\Omega$ is bounded, we fix $R>0$ such that $\Omega \subset B_O(R)$, where $B_O(R)$ denotes the ball of radius $R$ centred at the origin. Throughout the paper, the potential $V : [0,R] \to \mathbb{R}$ is assumed to be of class $C^1$ and $\|\cdot\|$ stands for the Euclidean norm of a vector. The above Dirichlet eigenvalue problem is understood in the standard $L^2(\Omega)$ framework; we refer, for instance, to~\cite{Henrot2006} for the precise functional-analytic setting.

Classical methods for approximating the eigenpairs of~\eqref{EVP-radial} are mainly based on triangulating the domain and then solving the resulting matrix eigenvalue problem obtained via a Galerkin discretisation; see, for example,~\cite{BabuskaOsborn1991,boffi2010} and the references therein. Although effective, this approach suffers from a major drawback: the resulting matrices are very large -- of order $h^{-4}$, where $h$ denotes the mesh size -- leading to significant memory and precision limitations in practice. Moreover, any change in the domain requires restarting the procedure, including remeshing and recomputing the associated matrices.

In this context, we propose a method that avoids meshing the domain $\Omega$ altogether. Instead, the computational effort is spread across multiple trials of the real number $\lambda$, verifying whether it provides a suitable approximation of an eigenvalue. The matrices involved in testing these candidate values are substantially smaller than those arising in the Finite Element approach; a detailed comparison between the two methods is provided in Section~\ref{sec:numerical-eigenvalues}.

Our numerical approach to solving~\eqref{EVP-radial} is based on the Method of Particular Solutions (MPS), originally introduced in the 1960s by Fox, Henrici, Moler, and Payne~\cite{FoxHenriciMoler1967,MolerPayne1968}, and further developed in the 2000s by Betcke and Trefethen~\cite{BetckeTrefethen2005,Betcke2008}. In brief, for each test value $\lambda$, the method constructs a family of functions satisfying the differential equation~\eqref{EVP-radial}$_1$ and then searches for linear combinations that approximately satisfy the Dirichlet boundary condition. The existence of such a combination indicates the presence of an eigenfunction $u^*$ associated with an eigenvalue $\lambda^*$, which is close to the test value $\lambda$.

Beyond the relatively small size of the matrices involved, a further advantage of MPS is that the same set of basis functions can be used for multiple domains, with only the boundary conditions requiring adaptation. The method has proved to be highly accurate for a variety of polygonal domains; see, for instance,~\cite{BetckeTrefethen2005,GuidottiLambers2008,JonesUltraPrecise}. However, a major limitation of the existing literature is that MPS has been developed almost exclusively for the Laplace operator. This is partly due to the convergence proof of Fox, Henrici, and Moler~\cite{FoxHenriciMoler1967}, which relies on complex analysis and therefore applies only to operators with analytic coefficients. Moreover, explicit constructions of the basis functions are only available in the case of the Laplacian, where they reduce to classical Bessel functions. As a consequence, applications of MPS have been largely restricted to the Laplacian or to operators that can be reduced to it via linear transformations, such as the heat operator with constant anisotropy~\cite{Kleefeld2018}, or to very specific analytic radial potentials, such as the inverse-square potential~\cite{LiZhang2017}.

In contrast to the existing literature, we propose a version of the Method of Particular Solutions that applies to the Laplace operator with an arbitrary $C^1$ radial potential $V$. Our approach relies exclusively on PDE and ODE techniques to construct and efficiently approximate the analogue of Bessel-type basis functions in this setting. These basis functions depend only on the enclosing ball $B_O(R)$ and not on the domain $\Omega$, allowing them to be reused when the domain is modified. Furthermore, since no complex analysis is involved, the method does not present structural obstacles to extensions in three dimensions. For clarity of exposition, however, we restrict our analysis in this paper to two-dimensional domains.

\subsection{Description of the method}
\label{sec:description-of-method}
Since the potential $V$ is bounded, a translation argument \cite[Remark 1.1.3]{Henrot2006} allows us to assume without loss of generality that $V \geq 1$. It then follows from \cite[Theorem 1.2.2]{Henrot2006} that all eigenvalues of \eqref{EVP-radial} are greater than one, form a discrete sequence diverging to infinity, and have finite multiplicity. Our goal is to approximate all eigenpairs $(\lambda_n^*,u_n^*)$ with $\lambda_n^* \in [1,K]$, where $K>1$ is a prescribed parameter.

More precisely, by an eigenpair of \eqref{EVP-radial} we mean a pair $(\lambda^*,u^*)\in \RR \times H_0^1(\Omega)$ such that, for every $\varphi\in H_0^1(\Omega)$,

\begin{equation}\label{EVP-radial-w}
\int_{\Omega} \nabla u^*\nabla\varphi\,\dd \x + \int_{\Omega} V(\|\x\|) u^*(\x)\varphi(\x)\, \dd \x = \lambda \int_{\Omega} u^*\,\varphi \,\dd \x.
\end{equation}
We note that, by elliptic regularity \cite[Theorem 8.12]{GilbargTrudinger2001}, eigenfunctions belong to $H^2(\Omega)$ and satisfy
\[
\|u^*\|_{H^2(\Omega)} \leq C(\Omega)\left(\|V\|_{L^\infty([0,R])}+|\lambda^*|\right)\|u^*\|_{L^2(\Omega)} .
\]

We now outline the proposed algorithm:
\begin{enumerate}[label={\textbf{Step \Roman*}}]
   \item \textbf{Discretisation of the spectral parameter.}
We consider a fine equidistant division of the interval $[1,K]$ with step size $\mu>0$. For each point $\lambda$ in the division, the subsequent steps determine whether $\lambda$ approximates an eigenvalue of \eqref{EVP-radial} up to a prescribed accuracy.
    \item \textbf{Construction of the basis functions.}
For a fixed $\lambda$, we search for approximate solutions of \eqref{EVP-radial} defined on the ball $B_O(R)$ whose Fourier expansion with respect to the angular variable $\theta$ contains only the first $2J+1$ terms. Specifically, we seek functions of the form
    \begin{equation}\label{eq:def-UjLambda-1-intro}
    u^{J,\lambda}(\x)\coloneqq u_0^c(r)+\sum_{j=1}^J \left[u_j^c(r) \cos(j\theta)+ u_j^s (r) \sin(j\theta)\right], \quad \forall \x=(r\cos(\theta),r\sin(\theta))\in B_O(R),
    \end{equation}
where the functions $u_j^c$ and $u_j^s$ are defined on $[0,R]$. Moreover, they are chosen such that the function $u^{J,\lambda}$ belongs to $H^2(B_O(R))$ and satisfies the following equation: 
\begin{equation}\label{eq:PDE-uJlambda}
    -\Delta u^{J,\lambda} + V(\|\x\|)u^{J,\lambda}=\lambda u^{J,\lambda},\quad \x \in B_O(R),\end{equation}
without boundary conditions on $\partial B_O(R)$. More precisely, if we formally project equation \eqref{eq:PDE-uJlambda} with respect to the spherical variable $\theta$ onto the span of $\cos(j\theta)$, we get that $u_j^c$ satisfies the following Bessel-type ordinary differential equation with potential:
\begin{equation} \label{eq:uj-ODE}
-(u_j^c)''(r) -\frac{1}{r} (u_j^c)'(r)+\frac{j^2}{r^2} u_j^c(r) + V(r) u_j^c(r) - \lambda u_j^c(r)=0,\quad r\in (0,R].
\end{equation}
Moreover, projecting onto the span of $\sin(j\theta)$, one obtains that $u_j^s$ satisfies the same equation. Although this second-order ODE admits a two-dimensional solution space, some information is lost while deducing \eqref{eq:uj-ODE} from \eqref{eq:PDE-uJlambda}, since the behaviour of $u^{J,\lambda}$ near the origin $O$ was not considered. This issue is solved in Section \ref{sec:function-spaces} by considering a weak formulation of \eqref{eq:uj-ODE} which stems from \eqref{eq:PDE-uJlambda} and whose solution space is only one-dimensional. 

Denoting by $\uj^{\lambda}$ one of the two generators of this one-dimensional solution space normalised such that $\|\uj^{\lambda}(r)\cos(j\theta)\|_{L^2(B_O(R))}=1$, we obtain that any function of the form \eqref{eq:def-UjLambda-1-intro} satisfies \eqref{eq:PDE-uJlambda} if and only if there exists a vector $\aaa=(a_0^c,a_1^c,\ldots a_J^c,a_1^s,a_2^s,\ldots a_J^s)\in \RR^{2J+1}$ such that:
\begin{equation}\label{eq:def-UJLambdaAlpha}
u^{J,\lambda}(\x)=u^{J,\lambda}_\aaa(\x)\coloneqq \alpha_0^c \uz^{\lambda}(r)+\sum_{j=1}^J \uj^{\lambda}(r)\left[\alpha_j^c \cos(j\theta)+ \alpha_j^s (r) \sin(j\theta)\right],
\end{equation}
for every  $\x=
(r\cos(\theta),r\sin(\theta))\in B_O(R)$.
The next step of the algorithm will determine, if possible, a suitable vector $\aaa\in \RR^{2J+1}$ such that $u^{J,\lambda}_\aaa$ is an approximate eigenfunction of \eqref{EVP-radial}. 
\item \textbf{(Find $\aaa$ by minimisation).} We determine whether a test value $\lambda$ in the division of $[1,K]$ approximates an eigenvalue of \eqref{EVP-radial} up to some precision, by minimising over $\aaa\neq 0_{\RR^{2J+1}}$ the following quotient:
\begin{equation}\label{eq:def-F}
\mathcal{F}^{\lambda}(\aaa)\coloneqq \frac{\|u^{J,\lambda}_\aaa\|_{L^2(\partial\Omega)}}{\|u^{J,\lambda}_\aaa\|_{L^2(\Omega)}}.
\end{equation}
If the minimum value is below a prescribed threshold $\eps$, then $(\lambda,u^{J,\lambda}_\aaa)$ approximates an eigenpair of \eqref{EVP-radial}; see Theorem~\ref{thm:main}. The procedure also detects eigenvalues of higher multiplicity, as discussed in Section~\ref{sec:numerical-eigenfunctions}.\\

\noindent\textbf{Details about the numerical approximation of the quotient $\mathcal{F}^\lambda(\aaa)$ and its minimisation:}
\begin{enumerate}[label={\textbf{III.\arabic*}}]
    \item \textbf{(Approximating the basis functions via FEM).} The basis functions $\uj^\lambda$ are approximated using a one-dimensional Finite Element Method (FEM) on the interval $[0,R]$ applied to the weak formulation of \eqref{eq:uj-ODE}; see Section \ref{sec:FEM}.
    \item \label{item:alg:integration} \textbf{(Numerical integration on $\partial\Omega$  and $\Omega)$.} The numerator of $\mathcal{F}^\lambda(\aaa)$ is approximated via a Riemann sum on $\partial\Omega$, whereas the denominator is computed using a Monte Carlo method; see Section \ref{sec:numerical-integration}.
    \item \textbf{}The two steps above lead to a discrete counterpart of the quotient $\mathcal{F}^\lambda(\aaa)$, which has the form:
    \[\mathcal{F}^\lambda_h(\aaa)\coloneqq \frac{\|\mathcal{M}_{\partial\Omega}\alpha\|}{\|\mathcal{M}_{\Omega}\alpha\|},\]
    where the matrices $\mathcal{M}_{\partial\Omega}$ and $\mathcal{M}_{\Omega}$ are constructed in Section \ref{sec:quotient-approximation}. Theorem \ref{thm:approx-ratio} provides bounds for the error between $\mathcal{F}^\lambda$ and $\mathcal{F}^\lambda_h$.
    \item \textbf{(Linear optimisation).} The minimisation of the quotient $\mathcal{F}^\lambda_h$ is performed as in \cite[Section 5]{BetckeTrefethen2005}, using a QR decomposition of the joint matrix $$\left[\begin{array}{c}\mathcal{M}_{\partial\Omega}\\ \mathcal{M}_{\Omega}\end{array}\right]=\left[\begin{array}{c}\mathcal{Q}_{\partial\Omega}\\ \mathcal{Q}_{\Omega}\end{array}\right]\mathcal{R}$$ and then a Singular Value Decomposition of the matrix $\mathcal{Q}_{\partial\Omega}$. See Section \ref{sec:numerical} for more details.
\end{enumerate}

\end{enumerate}

The main result of this paper provides quantitative error bounds between the approximate eigenpairs produced by the algorithm and the true eigenpairs of \eqref{EVP-radial}, and guarantees that no eigenpair in $[1,K]$ is missed.

\begin{theorem}\label{thm:main}
Let $\Omega\subset B_O(R)$ be a simply connected $C^2$ bounded domain in $\RR^2$ containing the origin $O$. Let $V\in C^1([0,R])$ be a potential taking values in $[1,\infty)$,  and let $K>1$. Then, the following statements hold:
\begin{enumerate}[label={\Roman*)}]
\item \label{item:main-I} There exists a threshold $\eps_0=\eps_0(\Omega,K)$ and a constant $C(\Omega,V,K)>0$ such that, if $\lambda\in [1,K]$, $\eps\in (0,\eps_0)$ and $\aaa\in \RR^{2J+1}$ satisfy:
\[\mathcal{F}^\lambda(\aaa)\leq \eps,\]
then there exists an eigenpair $(\lambda^*,u^*)$ of \eqref{EVP-radial} with $\|u^*\|_{L^2(\Omega)}=1$ such that:
\[|\lambda-\lambda^*|\leq \eps\, C(\Omega,V,K)\quad \text{and}\quad \left\|\frac{u^{J,\lambda}_\aaa}{\| u^{J,\lambda}_\aaa\|_{L^2(\Omega)}}-u^*\right\|_{L^2(\Omega)}< \eps\,C(\Omega,V,K). \]
\item \label{item:main-II} For every $\eps>0$, there exists a positive integer $J_0(\eps, \Omega, V,K)$ such that for any $J>J_0$, there exists a threshold $\mu(\eps,\Omega,V,K,J)>0$ with the following property:

If $(\lambda^*,u^*)$ is an eigenpair of \eqref{EVP-radial} with $\lambda^*\in [1,K]$ and $\|u^*\|_{L^2(\Omega)}=1$, for each test value $\lambda\in [1,K]$ that satisfies $|\lambda-\lambda^*|<\mu(\eps,\Omega,V,K,J)$, there exists a vector $\aaa\in \RR^{2J+1}$ such that:
\[\mathcal{F}^\lambda(\aaa)\leq \eps\]
and 
\[\left\|\frac{u^{J,\lambda}_\aaa}{\| u^{J,\lambda}_\aaa\|_{L^2(\Omega)}}-u^*\right\|_{L^2(\Omega)}<\eps.\]
\end{enumerate}
\emph{Interpretation of statement \ref{item:main-II}:} if the truncation parameter $J$ is chosen sufficiently large and the discretisation of the interval $[1,K]$ is fine enough -- namely, with step size smaller than $\mu(\eps,\Omega,V,K,J)$ -- then the test value $\lambda$ which is closest to the eigenvalue $\lambda^*$ satisfies that $\mathcal{F}^\lambda(\aaa)$ is small for a particular $\aaa$  whose corresponding function $u^{J,\lambda}_\aaa$ is an accurate approximation of the true eigenfunction $u^*$. Consequently, the algorithm is guaranteed to approximate all eigenpairs of \eqref{EVP-radial} in the interval $[1,K]$.
\end{theorem}
The proof of Theorem \ref{thm:main} follows directly from Theorems~$\ref{thm:direct-implication-approx-lambda}$ and~$\ref{thm:approximate-eigenvalues}$ and relies crucially on the interior observability inequality established in Proposition~$\ref{prop:H2-stability-J}$. This inequality allows one to derive various estimates for the functions $u^{J,\lambda}_\aaa$ defined on $B_O(R)$ using only the information that they are normalised in $L^2(\Omega)$.
In general, if a function is only known to satisfy the equation~$\eqref{eq:PDE-uJlambda}$ and to have $L^2(\Omega)$ norm equal to one, it is not possible to deduce bounds on its $L^2(B_O(R))$ norm for a ball $B_O(R)$ strictly containing $\Omega$, without additional a priori information, for instance on its behaviour at the boundary $\partial B_O(R)$ (see, for example, \cite{Ervedoza2023}). However, if we restrict attention to functions of the form~$\eqref{eq:def-UjLambda-1-intro}$, that is, functions whose Fourier expansion with respect to the angular variable $\theta$ contains only the first $2J+1$ terms, then an $L^2(\Omega) \to H^2(B_O(R))$ observability estimate can be established.
This estimate is obtained by applying ordinary differential equation techniques, in particular Grönwall’s inequality, to the coefficient functions $u_j^c$ and $u_j^s$, each of which satisfies the Bessel-type equation with potential~$\eqref{eq:uj-ODE}$.\\

The paper is organised as follows. Section~$\ref{sec:convergence}$ is devoted to the proof of convergence of the proposed method, corresponding to statement~$\ref{item:main-I}$ of Theorem~$\ref{thm:main}$. Section~$\ref{sec:no-eigenpair-missed}$ addresses the construction of the basis functions $\uj^\lambda$ and establishes that the method does not miss any eigenpair, corresponding to statement~$\ref{item:main-II}$ of Theorem~$\ref{thm:main}$. In Section~$\ref{sec:FEM}$, we derive error bounds for the Finite Element approximation of the basis functions $\uj^\lambda$. Section~$\ref{sec:numerical-integration}$ discusses the numerical approximation of the quotient $\mathcal{F}^\lambda(\aaa)$, using Riemann sums on $\partial\Omega$ and Monte Carlo sampling in the interior of $\Omega$. Finally, Section~$\ref{sec:numerical}$ describes the minimisation procedure for $\mathcal{F}^\lambda$ and presents two numerical examples illustrating the computation of eigenpairs. Further research directions are outlined in Section~$\ref{sec:conclusion}$.\\

\textbf{Note:} 
Throughout the paper, constants appearing in the estimates may change from one occurrence to another. Unless a constant is independent of all parameters (in which case it is said to be universal), its dependence on the relevant quantities will be indicated explicitly. For example, $C(R,V,K)$ denotes a constant depending on $R$, $V$ and $K$. If a constant depends on the domain $\Omega$, it may also depend on two radii $\tilde{R}=\tilde{R}(\Omega)$ and $R=R(\Omega)$ such that
\[
 \overline{B_O(\tilde R)} \subset \Omega \subset \overline{\Omega} \subset B_O(R),
 \]
and such that the interval $[0,R(\Omega)]$ is contained in the domain of the potential $V$.

\section{Our method converges: if the quotient $\mathcal{F}^\lambda(\aaa)$ is small, then we are near an eigenpair}
\label{sec:convergence}

The purpose of this section is to prove the first statement of Theorem~\ref{thm:main}. We temporarily defer the construction of the basis functions $\uj^\lambda$ to Section~\ref{sec:weaks-solutions-basis-functions} and assume that, for every positive integer $J$ and every vector $\aaa \in \RR^{2J+1}$, the function $u^{J,\lambda}_{\aaa}$ satisfies the partial differential equation~\eqref{eq:PDE-uJlambda}. Under this assumption, if for a particular vector $ \aaa_{\min}^\lambda$ the quotient $\mathcal{F}^\lambda(\aaa_{\min}^\lambda)$ is sufficiently small, then the following theorem, adapted from the work of Moler and Payne~\cite[Theorems~1 and~2]{MolerPayne1968} and applied to the function  $u \coloneqq u^{J,\lambda}_{\aaa_{min}^\lambda} $, yields assertion~\ref{item:main-I} of Theorem~\ref{thm:main}.

\begin{theorem}\label{thm:approximate-eigenvalues}
    Let $\Omega\subset B_O(R)$ be a bounded $C^2$ domain in $\RR^2$ containing the origin $O$, let $V\in C^1([0,R])$ taking values in $[1,\infty)$ and $K>1$. There exists $C(\Omega)>0$ and $\eps_0>0$ depending only on the domain $\Omega$ such that if $\eps\in (0,\eps_0)$, $\lambda\in [1,K]$ and $u\in H^2(\Omega)$ satisfy the following properties:
    \begin{enumerate}[label={(\alph*)}]
        \item $-\Delta  u (\x) + V(\|\x\|)  u(\x)=\lambda u(\x)$ for a.e. $\x\in\Omega$;
        \item \label{hyp:u-on-boundary-leq-eps}$\|u\vert_{\partial\Omega}\|_{L^2(\partial\Omega)}\leq \eps$;
        \item \label{hyp:u-normalised} $\|u\|_{L^2(\Omega)}=1$,
    \end{enumerate}
    then there exists an eigenpair $(\lambda^*,u^*)$ of \eqref{EVP-radial} such that:
    \begin{enumerate}[label={\roman*)}]
        \item \label{item:approx-eigenvalues} $\displaystyle \left|\lambda-\lambda^*\right|\leq \eps \, K\, C(\Omega)$.
    \end{enumerate}
    Moreover, there exists $C(\Omega,V,K)>0$ and $\tilde \eps_0$ depending only on $\Omega$, $V$, $K$ such that, if $\eps\in (0,\tilde\eps_0)$,
    \begin{enumerate}[label={\roman*)}]
    \setcounter{enumi}{1}
        \item $\displaystyle \|u- u^*\|_{L^2(\Omega)}\leq \eps \, C(\Omega,V,K)$.
    \end{enumerate}
\end{theorem}
\begin{proof}
Let $A_V:H_0^1(\Omega)\to H^{-1}(\Omega)$ be the elliptic operator corresponding to the eigenvalue problem \eqref{EVP-radial}. More precisely, for any $\tilde u\in H_0^1(\Omega)$,
\[\langle A \tilde u, \varphi\rangle_{H_0^1(\Omega), H^{-1}(\Omega)}=\int_{\Omega} \nabla \tilde u\nabla\varphi\dd \x + \int_{\Omega} V(\|\x\|) \tilde u(\x)\varphi(\x)\, \dd \x, \quad \forall \varphi\in H_0^1(\Omega).\]
Since $V\geq 1$, the operator $A$ is coercive and thus, by \cite[Theorem 1.2.2]{Henrot2006}, there exists an orthonormal Hilbert basis $(\xi_n^*)_{n\geq 1}$ of $L^2(\Omega)$ which consists of eigenfunctions of $A$ with corresponding eigenvalues $1\leq\lambda_1^*\leq \lambda_2^* \leq \lambda_3^*\leq  \ldots$; in other words, for every positive integer $n$, $\xi_n^*\in H_0^1(\Omega)$ and $A\xi_n^*=\lambda_n \xi_n^*$.

Next, in order to apply the reasoning in \cite[Theorems 1 and 2]{MolerPayne1968}, let us define $\tilde u\in H^2(\Omega)$ such that it satisfies the following boundary-value problem:
\begin{equation}\label{eq:BVP-same-boundary-as-u}
\begin{cases}
-\Delta \tilde u(\x)+ V(\|\x\|)\tilde u(x)=0, & \x\in \Omega;\\
\tilde u(\x)= u (\x), & \x\in \partial \Omega.
\end{cases}
\end{equation}
The solution $\tilde u$ exists and it is unique by \cite[Theorem 8.3]{GilbargTrudinger2001} and has $H^2(\Omega)$ regularity by \cite[Theorem 8.12]{GilbargTrudinger2001}. 

Next, we aim to estimate the $L^2(\Omega)$ norm of the function $\tilde u$. For this purpose, inspired by the reasoning in \cite{Bucataru2022,BucataruGrajd}, we consider an arbitrary test function $\varphi\in L^2(\Omega)$ and take  $\psi\in H^2(\Omega)\cap H_0^1(\Omega)$ the unique solution of the Dirichlet problem:
\begin{equation}\begin{cases}
    -\Delta \psi(\x) + V(\|\x\|) \psi(\x) =\varphi(\x), & \x\in \Omega;\\
    \psi(\x) = 0, & \x \in \partial\Omega.
\end{cases}\end{equation}
First, we remark that, since $V\geq 1$, 
\begin{equation}\label{eq:psi-H1-leq-vaprhi-L2}
\|\psi\|_{H^1(\Omega)}\leq \|\varphi\|_{L^2(\Omega)}.
\end{equation}
We also note that the $H^2(\Omega)$ regularity of $\psi$ is a consequence of \cite[Theorem 8.12]{GilbargTrudinger2001} which, together with \eqref{eq:psi-H1-leq-vaprhi-L2} and \cite[Result (iv), p.~316]{brezis}, further implies that there exists a constant $C(\Omega)>0$ such that the normal derivative of $\psi$ in the sense of traces satisfies:
\begin{equation}\label{eq:normal-derivative-psi-leq-varphi-L2}
\|\partial_\nu \psi\|_{L^2(\partial\Omega)}\leq \|\partial_\nu \psi\|_{H^{\frac{1}{2}}(\partial\Omega)}\leq C(\Omega) \| \psi\|_{H^2(\Omega)}\  \leq C(\Omega) \|\varphi\|_{L^2(\Omega)}.
\end{equation}

 Next, an integration by parts argument leads to:
\begin{equation*}\begin{aligned}
    \int_{\Omega} \tilde u\,\varphi\,\dd\x &= -\int_{\Omega } \tilde u\, \Delta \psi\,\dd\x + \int_{\Omega} V(\|\x\|) \tilde u(\x)\psi(\x)\, \dd\x\\
    &=- \int_{\Omega} \Delta \tilde u\, \psi\, \dd\x + \int_{\Omega} V(\|\x\|) \tilde u(\x) \psi(\x)\, \dd\x - \int_{\partial \Omega} \tilde u\,\partial_\nu \psi \,\dd\x\\
    &=- \int_{\partial \Omega} u\,\partial_\nu \psi \,\dd\x,
\end{aligned}
\end{equation*}
from which we deduce by the Cauchy-Schwarz inequality and \eqref{eq:normal-derivative-psi-leq-varphi-L2} that:
\begin{equation}
\left|\int_{\Omega} \tilde u\, \varphi\, \dd\x\right|\leq C(\Omega) \| u\vert_{\partial{\Omega}}\|_{L^2(\partial\Omega)}\|\varphi\|_{L^2(\Omega)}.
\end{equation}
Since the previous inequality is valid for an arbitrary $\varphi\in L^2(\Omega)$, we deduce by the hypothesis \ref{hyp:u-on-boundary-leq-eps} that:
\begin{equation}\label{eq:v-leq-eps}
\|\tilde u\|_{L^2(\Omega)}\leq \eps \, C(\Omega).
\end{equation}
The rest of the proof is essentially contained in \cite[Theorems 1 and 2]{MolerPayne1968}, however we include it here for the sake of completion.

Since $(u - \tilde u)\in H_0^1$ satisfies $A(u -\tilde u) = \lambda u$ and the eigenvectors $(\xi_n^*)_{n\geq 1}$ form an orthonormal Hilbert basis of $L^2(\Omega)$, we are able to write:
\begin{equation}\label{eq:eigenvalues-estimates-1}
\begin{aligned}
\|\tilde u\|_{L^2(\Omega)}^2&=\sum_{n=1}^\infty\left(\tilde u,\xi_n^*\right)^2_{L^2(\Omega)}=\sum_{n=1}^\infty\left[\left(\tilde u- u,\xi_n^*\right)_{L^2(\Omega)}+\left(u,\xi_n^*\right)_{L^2(\Omega)}\right]^2 \\
&=\sum_{n=1}^\infty\left[\frac{1}{\lambda_n^*}\left(\tilde u- u,A_V\xi_n^*\right)_{L^2(\Omega)}+\left( u,\xi_n^*\right)_{L^2(\Omega)}\right]^2\\
&=\sum_{n=1}^\infty\left[\frac{1}{\lambda_n}\left(A_V(\tilde u- u),\xi_n^*\right)_{L^2(\Omega)}+\left(u,\xi_n^*\right)_{L^2(\Omega)}\right]^2\\
&=\sum_{n=1}^\infty\left[-\frac{1}{\lambda_n^*}\left(\lambda u,\xi_n^*\right)_{L^2(\Omega)}+\left(u,\xi_n^*\right)_{L^2(\Omega)}\right]^2\\
&=\sum_{n=1}^\infty\left|\frac{\lambda-\lambda_n^*}{\lambda_n^*}\right|^2\left(u,\xi_n^*\right)_{L^2(\Omega)}^2\geq \inf_{n\geq 1} \left|\frac{\lambda-\lambda_n^*}{\lambda_n^*}\right|^2 \|u\|_{L^2(\Omega)}^2.
\end{aligned}
\end{equation}

Using \eqref{eq:v-leq-eps} and hypothesis \ref{hyp:u-normalised}, we deduce that:
\[\inf_{n\geq 1} \left|\frac{\lambda-\lambda_n^*}{\lambda_n^*}\right|\leq \eps\,  C(\Omega).\]
Therefore, choosing $\eps_0< \frac{1}{2\,C(\Omega)}$, it follows that there exists a positive integer $n_\lambda$ such that:
\begin{equation}\label{eq:eigenvalues-relative-error}
\left|\frac{\lambda-\lambda_{n_\lambda}^*}{\lambda_{n_\lambda}^*}\right|\leq \eps\, C(\Omega)<\frac 1 2.
\end{equation}
The first part of the conclusion of Theorem \ref{thm:approximate-eigenvalues} follows via algebraic manipulation of the inequality above, taking into account that $\lambda\leq K$.

In order to prove the second statement of Theorem \ref{thm:approximate-eigenvalues}, we assume without loss of generality that $\xi_{n_\lambda}^*$ is the normalised  $L^2(\Omega)$ projection of $u$ on the eigenspace corresponding to $\lambda_{n_\lambda}^*$. Note that if the projection of $u$ on the eigenspace happens to be null, then the reasoning below works without any assumption.

To the error $\|u-\xi_{n_\lambda}^*\|_{L^2(\Omega)}$, we use that $(\xi_n^*)_{n\geq 1}$ is an orthonormal Hilbert basis of $L^2(\Omega)$, to obtain:
\begin{equation}\label{eq:eigenfunctions-estimate-1}
\begin{aligned}
    \left\|u-\xi_{n_\lambda}^*\right\|_{L^2(\Omega)}^2 &= \sum_{n\geq 1} \left(u-\xi_{n_\lambda}^*,\xi_n^*\right)_{L^2(\Omega)}^2 =  \left[\left(u,\xi_{n_\lambda}^*\right)_{L^2(\Omega)}-1\right]^2 + \sum_{n\neq n_\lambda} \left(u,\xi_n^*\right)_{L^2(\Omega)}^2 \\
    &= 1- 2\left(u,\xi_{n_\lambda}^*\right)_{L^2(\Omega)} + \| u\|_{L^2(\Omega}^2 = 2\left[1- \left( u,\xi_{n_\lambda}^*\right)_{L^2(\Omega)}\right].
\end{aligned}
\end{equation}
Since $\xi_{n_\lambda}^*$ is the normalized projection of $u$, we get, by \cite[Corollary 5.4]{brezis}, that the scalar product $\left( u,\xi_{n_\lambda}^*\right)_{L^2(\Omega)}$ is non-negative.  Consequently, we can estimate $1- \left(u,\xi_{n_\lambda}^*\right)_{L^2(\Omega)}$ as follows:
\[\begin{aligned} 1- \left(u,\xi_{n_\lambda}^*\right)_{L^2(\Omega)}=\frac{1- \left( u,\xi_{n_\lambda}^*\right)_{L^2(\Omega)}^2}{1+ \left( u,\xi_{n_\lambda}^*\right)_{L^2(\Omega)}}\leq 1- \left( u,\xi_{n_\lambda}^*\right)_{L^2(\Omega)}^2 &= \|u\|_{L^2(\Omega)}^2 - \left(u,\xi_{n_\lambda}^*\right)_{L^2(\Omega)}^2 = \sum_{n\neq n_\lambda} \left(u,\xi_n^*\right)_{L^2(\Omega)}^2.
\end{aligned}
\]
Moreover, since $\xi_{n_\lambda}^*$ is the normalized projection of $u$, this implies that $u$ is orthogonal to any other eigenfunction $\xi_n^*$, $n\neq n_\lambda$ with the same eigenvalue $\lambda_n^*=\lambda_{n_\lambda}^*$. Therefore, the sum in the last term above becomes:
\[\sum_{n\neq n_\lambda} \left(u,\xi_n^*\right)_{L^2(\Omega)}^2 = \sum_{\lambda_n^*\neq \lambda_{n_\lambda}^*} \left(u,\xi_n^*\right)_{L^2(\Omega)}^2.\]
Next, let $\delta(\Omega,V,K)$ be the minimum distance between two consecutive distinct eigenvalues of \eqref{EVP-radial}, which are smaller than or equal to $K+1$. Then, we use \eqref{eq:v-leq-eps} \eqref{eq:eigenvalues-estimates-1}, and the statement \ref{item:approx-eigenvalues} of Theorem \ref{thm:approximate-eigenvalues} to write:
\[\sum_{\lambda_n^*\neq \lambda_{n_\lambda}^*} \left(u,\xi_n^*\right)_{L^2(\Omega)}^2 \leq \sup_{\lambda_n^*\neq \lambda_{n_\lambda}^*} \left|\frac{\lambda_n^*}{\lambda-\lambda_n^*}\right|^2 \|\tilde u\|_{L^2(\Omega)}^2\leq \frac{\eps \,K\, C(\Omega)}{\delta(\Omega,V,K)-|\lambda-\lambda_{n_\lambda}^*|}\leq \frac{\eps \,K\, C(\Omega)}{\delta(\Omega,V,K)-\frac{\eps \, K\, C(\Omega)}{1-\eps C(\Omega)}},\]
provided that $\eps_0$ is small enough. The conclusion follows.
\end{proof}

\section{No eigenpair is missed by the algorithm: the quotient $\mathcal{F}^\lambda(\aaa)$ is small near every eigenpair}
\label{sec:no-eigenpair-missed}

The aim of this section is to show that, if the test value $\lambda$ is sufficiently close to an eigenvalue $\lambda^*$ of \eqref{EVP-radial} and if the positive integer $J$ is sufficiently large, then each corresponding eigenfunction $u^*$ can be approximated on $\Omega$ by the restriction of a function $u^{J,\lambda}_{\aaa}$ satisfying~\eqref{eq:PDE-uJlambda} on $B_O(R)$, whose angular Fourier expansion involves only the first $2J+1$ terms.

\begin{theorem}
    \label{thm:direct-implication-approx-lambda} Let $\Omega\subset B_O(R)$ be a simply connected $C^2$ bounded domain in $\RR^2$ containing the origin $O$,  let $V\in C^1([0,R])$ taking values in $[1,\infty)$, and $K>1$. For every $\eps>0$, there exists a positive integer $J_0=J_0(\eps,\Omega,V,K)$ such that, for each $J\geq J_0$, there is a threshold $\mu(\eps,\Omega,V,K,J)>0$ with the following property: 
    
    If $(\lambda^*,u^*)$ is an eigenpair of \eqref{EVP-radial} with $\lambda^*\in [1,K]$ and $\|u^*\|_{L^2(\Omega)}=1$, then for each value $\lambda\in [1,K]$ with $|\lambda-\lambda^*|\leq \mu(\eps,\Omega,V,K,J)$,  there exists a vector $\aaa\in \RR^{2J+1}$ such that the function $u^{J,\lambda}_\aaa$ satisfies:
    \begin{enumerate}[label={\roman*)}]
    \item \label{item:direct-implication-approx-boundary} $\|u^{J,\lambda}_\aaa\|_{L^2(\partial\Omega)}\leq \|u^{J,\lambda}_\aaa\|_{H^{\frac{1}{2}}(\partial\Omega)}\leq \eps$;
    \item $\|u^{J,\lambda}_\aaa\|_{L^2(\Omega)}=1$;
    \item \label{item:direct-implication-approx-eigenfunction} $\left\|u^{J,\lambda}_\aaa-u^*\right\|_{H^1(\Omega)}\leq \eps$.
    \end{enumerate}
\end{theorem}

The rest of this section is dedicated to the construction of the basis functions $(\uj^{\lambda})_{j\geq 0}$ used to define the function $u^{J,\lambda}_\aaa$, and to the proof of Theorem \ref{thm:direct-implication-approx-lambda}; it has the following outline: 
\begin{itemize}
    \item We prove that every function $u\in H^2(\Omega)$ that solves the PDE
    \begin{equation}\label{eq:PDE-lambda}
    -\Delta u +V(\|\x\|)u =\lambda u        
    \end{equation}
    on the domain $\Omega$ (regardless of the boundary conditions), can be approximated with the restriction to $\Omega$ of a function that solves \eqref{eq:PDE-lambda} on the entire $B_O(R)$ (Section \ref{sec:extension-Runge}).
    \item We state and prove error bounds between functions in $H^2(B_O(R))$ and the partial sums of their Fourier expansions with respect to $\theta$ (Section \ref{sec:Fourier-estimates}).
    \item We show that the Fourier coefficients of functions which satisfy \eqref{eq:PDE-lambda} must satisfy a suitable weak formulation of the Bessel-type ODE with potential \eqref{eq:uj-ODE} (Section \ref{sec:function-spaces}).
    \item We prove existence and uniqueness up to a scaling factor for the solutions of the weak formulation of \eqref{eq:uj-ODE}. This will allow us to prove the conclusion of Theorem \ref{thm:direct-implication-approx-lambda} in the case when the test value $\lambda$ is an eigenvalue of \eqref{EVP-radial} (Section \ref{sec:weaks-solutions-basis-functions}).
    \item We prove a stability result for the weak solutions of \eqref{eq:uj-ODE} with respect to $\lambda$, which will lead to the conclusion of Theorem \ref{thm:direct-implication-approx-lambda} for test values $\lambda$ that are sufficiently close to an eigenvalue $\lambda^*$ (Section \ref{sec:stability-lambda}).
\end{itemize}

\subsection{Extension to a ball while preserving the PDE}
\label{sec:extension-Runge}
Since the polar decomposition only makes sense on circular domains, the first step towards the construction of an approximation such as $u^{J,\lambda}_\aaa$ of a function $u$ that satisfies \eqref{eq:PDE-lambda} on the domain $\Omega$ consists in extending -- at least approximately -- this function to the ball $B_O(R)$ enclosing $\Omega$, while preserving the equation \eqref{eq:PDE-lambda}. This extension can be done via the so-called \emph{Runge property}, which requires $\Omega$ to be simply connected; the result is contained in the following proposition which is essentially taken from \cite{KravchenkoVicente2022} (see also \cite[Theorem 36]{KravchenkoVicente2020}):

\begin{proposition}[Runge property in the $H^1$ norm]\label{prop:runge}
Let $\lambda\in \RR$ and $u\in H^1(\Omega)$ a weak solution to \eqref{eq:PDE-lambda} in the sense that:
\begin{equation}\label{eq:weak-no-boundary}
    \int_{\Omega} \nabla  u \nabla \varphi\, \dd \x +\int_{\Omega} V(\|\x\|) u(\x)\varphi(\x)\, \dd \x =\lambda \int_{\Omega} u\, \varphi \, \dd \x, \quad \forall \varphi\in H_0^1(\Omega).
\end{equation}
Then, for every $\eps>0$, there exists a function $\bar{u}_\eps\in H^2(B_O(R))$ satisfying:
\begin{enumerate}[label={\arabic*)}]
    \item For almost every $\x\in B_O(R)$, $\bar u_\eps$ satisfies \eqref{eq:PDE-lambda}.
    \item $\displaystyle \left\|\bar u_\eps\vert_\Omega-u\right\|_{H^1(\Omega)}\leq \eps$.
\end{enumerate}
\end{proposition}
\begin{proof}
First, we take an arbitrary $\eta>0$ and extend $V$ to a $C^1$ function on $B_O(R+\delta)$. By the Runge property in \cite[Theorem 13]{KravchenkoVicente2022}, there exists $\bar u_\eps\in H^1(B_O(R+\delta))$ satisfying the weak formulation \eqref{eq:weak-no-boundary} of \eqref{eq:PDE-lambda} on the entire $B_O(R+\delta)$ and with the property that $\left\|\bar u_\eps\vert_{\Omega}-u\right\|_{H^1(\Omega)}\leq \eps$. 

Finally, the interior regularity theory for elliptic PDE \cite[Theorem 8.8]{GilbargTrudinger2001} implies that $\bar u_\eps\in H^2(B_O(R))$ and it satisfies \eqref{eq:PDE-lambda} for almost every $\x\in B_O(R)$.
\end{proof}

\subsection{Fourier expansion with respect to the angular variable and truncation error estimates}
\label{sec:Fourier-estimates}
After extending the solutions of~\eqref{eq:PDE-lambda} from $\Omega$ to $B_O(R)$, we temporarily set aside the domain $\Omega$ and focus on functions $u \in H^2(B_O(R))$ that satisfy~\eqref{eq:PDE-lambda}. For such functions, the angular Fourier expansion is constructed by switching to polar coordinates and expanding $u$ in a Fourier series with respect to the angle $\theta$, in the spirit of~\cite{GuidottiLambers2008,KravchenkoVicente2020}. In what follows, we establish error estimates between these functions and the partial sums of their corresponding angular Fourier expansions.

Since $u\in L^2(B_O(R))$, we can define for a.e. $r\in [0,R]$ and for all positive integers $j$, the following quantities:
\begin{equation}\label{eq:def-u0-ur-vr}
u_0^c(r)\coloneqq \frac{1}{2\pi}\int_0^{2\pi} u(r,\theta) \dd \theta\hspace{0.2cm} \Bigg\vert\hspace{0.2cm} u_j^c(r)\coloneqq \frac{1}{\pi}\int_0^{2\pi}  u(r,\theta)\cos(j\theta) \dd \theta \hspace{0.2cm} \Bigg\vert\hspace{0.2cm} u_j^s(r)\coloneqq \frac{1}{\pi}\int_0^{2\pi} u(r,\theta)\sin(j\theta)\dd \theta,
\end{equation}
where we abuse the notation and write $u(r,\theta)$ instead of $u(r\cos(\theta),r\sin(\theta))$.

Then, for every positive integer $J$, the partial sum of the angular Fourier expansion of $u$ is given by:
\begin{equation}\label{eq:def-uJ}
u^{J}(r,\theta)\coloneqq u_0^c(r)+\sum_{j=1}^J \left[u_j^c(r) \cos(j\theta)+ u_j^s (r) \sin(j\theta)\right], \quad (r\cos(\theta),r\sin(\theta))\in B_O(R).\end{equation}
The truncation error between $u^J$ and $u$ is given by the following proposition:
\begin{proposition}\label{prop:uJ approximates-u}
    Let $u\in H^2(B_O(R))$ and $u^J$ defined in \eqref{eq:def-uJ}. Then, there exists a constant $C(R)>0$ such that:
    \[\|u^J-u\|_{L^2(B_O(R))}\leq \frac{C(R)}{J\sqrt{J}} \|u\|_{H^2(B_O(R))} \hspace{0.5cm}\text{and}\hspace{0.5cm} \|u^J-u\|_{H^1(B_O(R))}\leq\frac{C(R)}{\sqrt{J}}  \|u\|_{H^2(B_O(R))}.\]
\end{proposition}
The proof of this proposition is relegated to Appendix \ref{acaret:proofs}.

\subsection{Function spaces for the Fourier coefficients}
\label{sec:function-spaces}

The purpose of this section is to analyse in greater detail the coefficient functions $u_j^c$ and $u_j^s$ arising in the angular Fourier expansion \eqref{eq:def-uJ}, and to characterise them using the equation~\eqref{eq:PDE-lambda} satisfied by $u$. Formally, the representation~\eqref{eq:def-u0-ur-vr} suggests that the coefficient functions $u_j^c$ and $u_j^s$ satisfy the Bessel-type ordinary differential equation with potential~\eqref{eq:uj-ODE}. This equation is a second-order linear ODE and therefore admits a two-dimensional space of solutions.

However, equation~\eqref{eq:uj-ODE} alone does not capture the behaviour of $u$ near the origin, which is constrained by the fact that $u \in H^2(B_O(R))$ solves~\eqref{eq:PDE-lambda}. Consequently, not every solution of~\eqref{eq:uj-ODE} corresponds to a Fourier coefficient of such a function $u$. To identify the admissible solutions, we exploit the weak formulation~\eqref{eq:weak-no-boundary}. For this purpose, we introduce the following function spaces, which are designed to accommodate the Fourier coefficients of solutions to~\eqref{eq:weak-no-boundary}:
\[\begin{aligned}
L^2_r([0,R]) &\coloneqq \left\{v:[0,R]\to \RR \text{ measurable }:\displaystyle\|v\|_r^2\coloneqq\int_0^R |v(r)|^2 r \, \dd r\right\};\\
H_{j,r}([0,R]) &\coloneqq \left\{ v\in L^2_r([0,R]) : v \text{ is weakly differentiable and } \|v\|_{j,r}^2 \coloneqq a_j(v,v)+\|v\|_r^2\right\},
\end{aligned}\]
where the bilinear form $a_j$ is defined as:
\[a_j(v,w)\coloneqq\int_{0}^R v'(r)w'(r)\,r\,\dd r + j^2 \int_0^R v(r)w(r)\frac{1}{r}\dd r.\]

Furthermore, we define the space $H_{j,r}^0([0,R])$ as the completion in the $\|\cdot\|_{j,r}$ norm of the smooth functions with compact support in $(0,R)$. We remark that, similar to \cite[Theorems 4.8 and 8.1]{brezis}, one can easily prove that $L_r$, $H_{j,r}$ and $H_{j,r}^0$ are Hilbert spaces.

Next, we call $v\in H_{j,r}([0,R])$ a weak solution of \eqref{eq:uj-ODE} if:
 \begin{equation}\label{EVP-j-W}
    a_j(v,w) + \int_0^R V(r) v(r)w(r)\, r\,\dd r=\lambda (v,w)_r, \quad \forall w\in H_{j,r}^0([0,R]),\end{equation}
where $(\cdot,\cdot)_{r}$ is the scalar product on $L_r([0,R])$.

The following proposition establishes the connection between the $L^2$, $H^1$ and $H_0^1$ spaces on $B_O(R)$ and the function spaces on $[0,R]$ defined above, as well as the connection between weak solutions of the PDE \eqref{eq:PDE-lambda} and the ODE \eqref{eq:uj-ODE}:

\begin{proposition}\label{prop:spaces-equivalence}
Let $u$ be a measurable function on $B_O(R)$ and $v$ a measurable function on $[0,R]$.
    With the notation in \eqref{eq:def-u0-ur-vr}, the following statements hold true:
    \begin{enumerate}[label={(\roman*)}]
    \setlength{\itemsep}{0.3em}
        \item\label{item:l2-first} if $u\in L^2(B_O(R))$, then $u_j^c,u_j^s\in L_r([0,R])$, $\forall j\geq 1$ and $\|u_j^c\|_{r},\|u_j^s\|_{r}\leq \frac{1}{\sqrt{\pi}} \|u\|_{L^2(B_O(R))}$;
        \item  \label{item:h1-first} if $u\in H^1(B_O(R))$, then $u_j^c,u_j^s\in H_{j,r}([0,R])$, $\forall j\geq 1$  and $\|u_j^c\|_{j,r},\|u_j^s\|_{j,r}\leq \frac{2\sqrt{2}}{\sqrt{\pi}} \|u\|_{H^1(B_O(R))}$;
        \item \label{item:h01-first} if $u\in H^1_0(B_O(R))$, then $u_j^c,u_j^s\in H_{j,r}^0([0,R])$, $\forall j\geq 1$;
        \item \label{item:l2-second} for every $j\geq 1$, $v(r)\cos(j\theta)\in L^2(B_O(R))$ $\Leftrightarrow$ $v\in L_r([0,R])$;\\[5pt]
        \mbox{}\hspace{0.2cm} if this is the case, $\|v(r)\cos(j\theta)\|_{L^2(B_O(R))}=\sqrt{\pi} \|v\|_{r};$
        \item \label{item:h1-second} for every $j\geq 1$, $v(r)\cos(j\theta)\in H^1(B_O(R))$ $\Leftrightarrow$ $v\in H_{j,r}([0,R])$; if this is the case, \\[5pt]$\|\nabla[v(r)\cos(j\theta)]\|_{L^2(B_O(R))}^2=\pi \,a_j(v,v)$ \hspace{0.2cm} and \hspace{0.2cm} $\|v(r)\cos(j\theta)\|_{H^1(B_O(R))}=\sqrt{\pi} \|v\|_{j,r};$
        \item \label{item:h01-second} for every $j\geq 0$, $v(r)\cos(j\theta)\in H_0^1(B_O(R))$ $\Leftrightarrow$ $v\in H_{j,r}^0([0,R])$;
        \item \label{item:weak-first} if $u$ is a weak solution of \eqref{eq:PDE-lambda}, in the sense of \eqref{eq:weak-no-boundary}, then the functions $u_j^c(r)\cos(j\theta)$, $j\geq 0$ and $u_j^s(r)\sin(j\theta)$, $j\geq 1$ are also weak solutions of \eqref{eq:PDE-lambda};
        \item \label{item:weak-second} for every $j\geq 0$, $v(r)\cos(j\theta)$ is a weak solution of \eqref{eq:PDE-lambda}, in the sense of \eqref{eq:weak-no-boundary} if and only if $v$ is a weak solution of \eqref{eq:uj-ODE}, in the sense of \eqref{EVP-j-W}.
    \end{enumerate}
    The statements \ref{item:l2-second},\ref{item:h1-second},\ref{item:h01-second} and \ref{item:weak-second} also hold true if we replace  $\cos(j\theta)$ with $\sin(j\theta)$, with $j\geq 1$.
    
    If $j=0$, all the relevant statements are true, with the only modification that $\sqrt{\pi}$ is replaced by  $\sqrt{2\pi}$ in the expression of the constants involved.
    \end{proposition}
    A sketch of proof for this proposition can be found in Appendix \ref{acaret:proofs}.\\
    
The statement \ref{item:weak-second} of Proposition \ref{prop:spaces-equivalence} implies via the regularity theory for elliptic equations \cite[Theorem 8.12]{GilbargTrudinger2001} that, if $\lambda\in [1,K]$, then for every $j\geq 0$ and for every weak solution $v$ of \eqref{eq:uj-ODE}, in the sense of \eqref{EVP-j-W}, $v(r)\cos(j\theta)$ and $v(r)\sin(j\theta)$ belong to $H^2(B_O(R))$ and 
\begin{equation}\label{eq:regularity-wcosjtheta}
\|v(r)\cos(j\theta)\|_{H^2(B_O(R))}\leq C(R,V,K) \|v(r)\cos(j\theta)\|_{L^2(B_O(R))}=C(R,V,K)\|v\|_r,\end{equation}
 This second-order regularity can be transferred to the function $v$ via the following proposition, which is proved in Appendix \ref{acaret:proofs}:
 \begin{proposition}\label{prop:wcos-inH2}
 Let $v$ be a measurable function on $[0,R]$ such that $v(r)\cos(j\theta)\in H^2(B_O(R))$ for some positive integer $j$. Then, for every compact subinterval $\mathcal{I}$ of $(0,R]$, $v$ belongs to $H^2(\mathcal{I})$. Moreover, the weak second derivative of $v$ has the following form:
 \begin{equation}\label{eq:w-second-derivative}
     v''(r)=\frac{1}{\pi}\int_0^{2\pi} \Delta [v(r)\cos(j\theta)] \cos(j\theta)\,\dd \theta - \frac{1}{r}v'(r) +  \frac{j^2}{r^2} v(r).
 \end{equation}
 If $j=0$, the Proposition holds true, but the factor $\frac{1}{\pi}$ in front of the integral changes to $\frac{1}{2\pi}$.
 \end{proposition}

The following corollary allows us to use Ordinary Differential Equations methods (e.g. Gr\"onwall's Lemma) for the weak solutions of \eqref{eq:uj-ODE}, in the sense of \eqref{EVP-j-W}. This will be particularly useful in establishing the interior observability estimate for functions with at most $2J+1$ leading terms in the Fourier expansion with respect to $\theta$ (Proposition \ref{prop:H2-stability-J}). It will also play a role in the proof of Theorem \ref{thm:solution-space-1dim} concerning the dimension of the solution space of \eqref{EVP-j-W}.
\begin{corollary}\label{cor:weak-to-ODE}
    If $v$ is a weak solution of \eqref{eq:uj-ODE}, in the sense of \eqref{EVP-j-W}, then $v$ satisfies \eqref{eq:uj-ODE} on $(0,R]$ in the sense of weak derivatives. Equivalently, the functions $v$ and $v'$ admit continuous representatives on (0,R] which satisfy:
    \[v'(r_1)-v'(r_2) = \int_{r_2}^{r_1} \left[-\frac{1}{r} v'(r)+\frac{j^2}{r^2} v(r)+V(r) v(r)-\lambda\, v(r)\right]\dd r,\]
    for every $r_1,r_2\in (0,R]$. 
    
    Furthermore, the pair $U\coloneqq \left(\begin{array}{c} v\\v'\end{array}\right)$ satisfies:
    \begin{equation}\label{eq:U-integral-equation}
    U(r_2)-U(r_1)=\int_{r_1}^{r_2} \mathcal{A}_{j,r} U(r)\dd r,
    \end{equation}
    for every $r_1,r_2\in (0,R]$, where the matrix $\mathcal{A}_{j,r}$ has the following form:
    \[\mathcal{A}_{j,r}\coloneqq\left(\begin{array}{c|c}
    0 & 1 \\
     \hline\\[-0.4cm]
    \frac{j^2}{r^2}+V(r)-\lambda & -\frac{1}{r}
    \end{array}\right).\]
\end{corollary}
\begin{proof}
    Proposition \ref{prop:spaces-equivalence} \ref{item:weak-second} implies that $v(r)\cos(j\theta)$ is a weak solution of \eqref{eq:PDE-lambda} in the sense of \eqref{eq:weak-no-boundary}. The elliptic regularity theory \cite[Theorem 8.12]{GilbargTrudinger2001} implies that $v(r)\cos(j\theta)\in H^2(B_O(R))$ and it satisfies equation \eqref{eq:PDE-lambda} for almost every $\x \in B_O(R)$. More precisely, the weak Laplacian of $v(r)\cos(j\theta)$ satisfies:
    \[\Delta [v(r)\cos(j\theta)] = V(r)v(r)\cos(j\theta)-\lambda v(r)\cos(j\theta),\]  for almost every $\x=(r\cos(\theta),r\sin(\theta))\in B_O(R)$. Then, the conclusion follows by the equality \eqref{eq:w-second-derivative} from Proposition \ref{prop:wcos-inH2}.
\end{proof}

\subsection{Solving the weak problem to determine the Fourier coefficients}
\label{sec:weaks-solutions-basis-functions}
This section is dedicated to the construction of the basis functions $(\uj^\lambda)_{j\geq 0}$ as the unique (up to a change of sign) $L_r([0,R])$ normalised solutions of the weak equation \eqref{EVP-j-W}. This uniqueness property is the subject of the following theorem, which, in other words, asserts that only a one-dimensional subspace of the two-dimensional solution space of \eqref{eq:uj-ODE} consists of solutions that come from a solution of \eqref{eq:PDE-lambda}:

\begin{theorem} \label{thm:solution-space-1dim} Let the potential $V\in C^1([0,R])$ satisfy $V\geq 1$. For every $\lambda\geq 1$ and non-negative integer $j$, the solution space of \eqref{EVP-j-W} is one-dimensional.
\end{theorem}
The key elements of the proof of this theorem are the Fredholm Alternative and the Cauchy-Lipschitz uniqueness theorem for ODE. In order to apply the Fredholm Alternative, we need to provide an appropriate functional framework for posing the problem \eqref{EVP-j-W}. Thanks to the equivalence results in Propositions \ref{prop:spaces-equivalence}, this framework is constructed similar to the usual study of weak solutions for elliptic equations.

First, we prove two lemmata: the first one characterises the compactness of the embedding $H_{j,r}([0,R])\subset L_r([0,R])$ and the second one provides a Poincar\' e-type inequality for functions in $H_{j,r}^0([0,R])$. They will be used to construct in \eqref{eq:def-AjV} the Dirichlet operator $A_{j,V}$ corresponding to the bilinear form
\begin{equation}\label{eq:def-ajV}
a_{j,V}(\cdot,\cdot)\coloneqq a_j(\cdot,\cdot)+(V\cdot,\cdot)_r.
\end{equation}
\begin{lemma}[{Compactness of the embedding $H_{j,r}([0,R])\subset L_r([0,R])$}]\label{lem:Hjr0-comapct-in-Lr}
For every $R>0$ and $j\geq 0$, the space $H_{j,r}([0,R])$ is compactly embedded in $L_r([0,R])$.
\end{lemma}
\begin{proof}
If $(v^n)_{n\geq 1}$ is a bounded sequence in $H_{j,r}([0,R])$, it means that the sequence $(v^n(r)\cos(j\theta))_{n\geq 1}$ is bounded in $H^1(B_O(R))$. Since this space is compactly embedded in $L^2(B_O(R))$, the Rellich-Kondrakov Theorem \cite[Theorem 9.16]{brezis} implies the existence of a subsequence $(v^{n_k}(r)\cos(j\theta))_{k\geq 1}$ converging strongly in $L^2(B_O(R))$. By Proposition \ref{prop:spaces-equivalence} \ref{item:l2-second}, $(v^{n_k})_{k\geq 1}$ is a Cauchy sequence in $L_r([0,R])$, so it is strongly convergent.
\end{proof}

\begin{lemma}[{Poincar\'e-type inequality for $H_{j,r}^0([0,R])$ functions}]\label{lem:poincare-H_jr0}
For every $R>0$, there exists a constant $C(R)>0$ such that, for every $j\geq 0$ and every $v\in H_{j,r}^0([0,R])$,
\begin{equation}\label{eq:poincare-H_jr0}
a_j(v,v)\geq C(R)\|v\|_{r}^2.
\end{equation}
\end{lemma}
\begin{proof}
The lemma is a direct consequence of Proposition \ref{prop:spaces-equivalence} \ref{item:h1-second} and the Poincar\' e inequality on $B_O(R)$, \cite[Corollary 9.19]{brezis}.
\end{proof}

Next, we consider the operator $A_{j,V}$ which maps $H_{j,r}^0([0,R])$  into its dual space $H_{j,r}^0([0,R])^*$, corresponding to the bilinear form $a_{j,V}$. More precisely,
\begin{equation}\label{eq:def-AjV}
\langle A_{j,V} v,w\rangle_{H_{j,r}^0([0,R])^*,H_{j,r}^0([0,R])}\coloneqq a_{j,V}(v,w)\coloneqq a_j(v,w)+(V\,v,w)_r, \quad \forall v,w\in H_{j,r}^0.
\end{equation}

The two lemmata above, together with the Lax-Milgram theorem and the fact that the potential $V(r)$ is greater than $1$ for every $r\in [0,R]$, allow us to construct the solution operator $A_{j,V}^{-1}:L_r([0,R])\to L_r([0,R])$, which is a compact operator and its range is contained in $H_{j,r}^0([0,R])$. We also note that the operator $A_{j,V}^{-1}$ is self-adjoint.

With this notation, we remark that a function $v\in H_{j,r}^0([0,R])$ is a solution of \eqref{EVP-j-W} if and only if $v-\lambda A_{j,V}^{-1}v=0$, which is equivalent to $v\in \mathrm{ker}(I-\lambda A_{j,V}^{-1})$, where by $I$ we understand the identity operator on $L_r([0,R])$.\\

With this preparation in mind, we are able to provide a proof of Theorem \ref{thm:solution-space-1dim}.
\begin{proof}[Proof of Theorem \ref{thm:solution-space-1dim}]
    Let $\psi\in C_c^\infty((0,R])$ satisfy $\psi(R)=1$. Then, for every real number $\beta$, the function $v\in H_{j,r}([0,R])$ satisfies \eqref{EVP-j-W} with the boundary condition $v(R)=\beta\in \RR$ if and only if the function $(v-\beta\psi)$ belongs to $H_{j,r}^0([0,R])$ and, for every $w\in H_{j,r}^0([0,R])$,
    \begin{equation}\label{eq:w-apsi-weak}
    a_j(v-\beta\psi,w)-\lambda(v-\beta\psi,w)=-a\left(-\psi''(r)-\frac{1}{r}\psi'(r)+\frac{j^2}{r^2}\psi(r)+V(r)\psi(r)-\lambda\psi(r),w(r)\right)_r\,.\end{equation}
To simplify the expression above, we introduce the notation:
\begin{equation}\label{eq-def-AjTilde}
(\tilde{A}_{j,V}\,w)(r)\coloneqq -w''(r)-\frac{1}{r}w(r)+\frac{j^2}{r^2}w(r)+V(r)w(r),
\end{equation}
for $r\in (0,R]$, whenever the function $w$ is regular enough such that the right-hand side above makes sense. We note that $\tilde{A}_{j,V} w= A_{j,V}w$ for every $w\in C_c^\infty((0,R))$, but one might apply $\tilde A_{j,V}$ to functions that do not vanish in the point $R$. In this way, we can rewrite \eqref{eq:w-apsi-weak} as follows:
\[A_{j,V}(v-\beta\psi)-\lambda (v-\beta\psi)=-\beta \left(\tilde A_{j,V} \psi-\lambda \psi\right),\]
which, in turn, can be rewritten as:
\begin{equation}\label{eq:inhomogeneous-EVP-j-operator}
\left(I-\lambda A_{j,V}^{-1}\right)\left[v-\beta\lambda^{-1}\tilde A_{j,V}\,\psi\right]=\beta(I-\lambda^{-1}\tilde A_{j,V}) \psi.
\end{equation}

In the next step, we apply the Fredholm alternative \cite[Theorem 6.6]{brezis} to the compact operator $A_{j,V}$ and distinguish two cases:\\[1pt]

\emph{Case I:} $\lambda$ is not an eigenvalue of $A_{j,V}$, which means that $\mathrm{ker}(I-\lambda A_{j,V}^{-1})$ is trivial. In this case, for each $\beta\in \RR$, the equation \eqref{eq:inhomogeneous-EVP-j-operator} has exactly one solution $v_\beta\in H_{j,r}([0,R])$. Therefore, the $H_{j,r}([0,R])$ solutions of \eqref{EVP-j-W} are fully characterised by the value they attain in the point $R$ (in the sense of trace), thus they form a one-dimnesional space.\\[1pt]

\emph{Case II:} $\lambda$ is an eigenvalue of $A_{j,V}$, which means that $\mathrm{ker}
(I-\lambda A_{j,V}^{-1})$ is non-trivial. In this case, we prove that every $H_{j,r}([0,R])$ solution of \eqref{EVP-j-W} belongs in fact to $H_{j,r}^0([0,R])$. Assuming by contradiction the opposite, we obtain that there exists a number $\beta\neq 0$ such that \eqref{eq:inhomogeneous-EVP-j-operator} has a solution. Since $A_{j,V}^{-1}$ is self-adjoint, the Fredholm alternative implies that the right-hand side of \eqref{eq:inhomogeneous-EVP-j-operator} is orthogonal to $\mathrm{ker}(I-\lambda A_{j,V}^{-1})$. Taking a non-zero element $w\in \ker(I-\lambda A_{j,V}^{-1})$, we use the orthogonality and integrate by parts to obtain:
\begin{equation}\label{eq:fredholm-orthogonality}
\begin{aligned}0=\int_0^R \left(\psi(r)-\lambda^{-1} (\tilde A_{j,V}\, \psi)(r)\right)w(r)r\,\dd r &= (\psi, w)_r -\lambda^{-1} \left[a_j(\psi,w)+(V\, \psi,w)_r\right]\\
& =\lambda^{-1}\psi(R)w'(R)R +\left(\psi, (I-\lambda^{-1}\tilde A_{j,V}) w\right)_r,
\end{aligned}\end{equation}
where the last line makes sense by Corollary \ref{cor:weak-to-ODE}. Since $w\in \ker(I-\lambda A_{j,V}^{-1})$, the same corollary implies that $w$ satisfies \eqref{eq:uj-ODE}, i.e. $\tilde A_{j,V}\,w= \lambda\, w$ almost everywhere. Therefore, \eqref{eq:fredholm-orthogonality} implies that $\psi(R)w'(R)R=0$, but, since $\psi(R)=1$, this implies $w'(R)=0$. Since $w$ satisfies \eqref{eq:uj-ODE} and $w(R)=w'(R)=0$, Corollary \ref{cor:weak-to-ODE} and Gr\"onwall's inequality imply that $w$ is constant zero, which represents a contradiction.

Therefore, in this case, all solutions of \eqref{EVP-j-W} vanish in the point $R$. A consequence of this fact, together with the ODE formulation in Corollary \ref{cor:weak-to-ODE} and Gr\"onwall's inequality, is that the solutions of \eqref{EVP-j-W} are fully characterised by their derivative in the point $R$ (in the sense of trace), thus they form a one-dimensional space. This concludes the proof.
\end{proof}

\begin{corollary}\label{cor:Fourier-Bessel-Series-ujLambda}
For every $j\in \NN$ and $\lambda\geq 1$, Theorem \ref{thm:solution-space-1dim} allows us to denote by $\uj^\lambda$ one of the two solutions of \eqref{EVP-j-W} with $\|\uj^\lambda\|_{r}=1$. Therefore, Proposition \ref{prop:spaces-equivalence} \ref{item:weak-first}-\ref{item:weak-second} implies that for every $H^2(B_O(R))$ solution $u$ of \eqref{eq:PDE-lambda} on $B_O(R)$, there exist some coefficients $(\alpha_j^c)_{j\geq 0}$ and $(\alpha_j^s)_{j\geq 1}$ such that the following equality holds in $H^1(B_O(R))$:
\[u (\x)=\alpha_0 \uz^\lambda(r)+\sum_{j=1}^\infty \uj^\lambda(r) \left[\alpha_j^c \cos(j\theta)+\alpha_j^s\sin(j\theta)\right], \quad \x=(r\cos(\theta),r\sin(\theta)).\]

Moreover, Proposition \ref{prop:spaces-equivalence} \ref{item:weak-second} and the regularity theory for elliptic equation \cite[Theorem 8.12]{GilbargTrudinger2001} imply that, for every vector $\aaa\in \RR^{2J+1}$, the function $u^{J,\lambda}_\aaa$ defined in \eqref{eq:def-UJLambdaAlpha} belongs to $H^2(B_O(R))$ and satisfies \eqref{eq:PDE-lambda} on $B_O(R)$.
\end{corollary}

An immediate consequence of Propositions \ref{prop:runge}, \ref{prop:uJ approximates-u} and Corollary \ref{cor:Fourier-Bessel-Series-ujLambda} is the following weaker version of Theorem \ref{thm:direct-implication-approx-lambda}, valid when $\lambda$ is an eigenvalue of \eqref{EVP-radial}:

\begin{proposition}\label{prop:direct-implication-exact-lambda}
Let $\Omega\subset B_O(R)$ be a simply connected $C^2$ bounded domain in $\RR^2$ containing the origin $O$, let $V\in C^1([0,R])$ taking values in $[1,\infty)$, and $K>1$. For every $\eps>0$, there exists a positive integer $J_0=J_0(\eps,\Omega,V,K)$ such that, for each $J\geq J_0$ and every eigenpair  $(\lambda^*,u^*)$ of \eqref{EVP-radial} with $\lambda^*\in [1,K]$ and $\|u^*\|_{L^2(\Omega)}=1$, one can find a vector $\aaa\in \RR^{2J+1}$for which the function $u^{J,\lambda^*}_\aaa$ satisfies:
    \begin{enumerate}[label={\roman*)}]
    \item \label{item:direct-implication-boundary} $\|u^{J,\lambda^*}_\aaa\|_{H^{\frac{1}{2}}(\partial\Omega)}\leq \eps$;
    \item \label{item:direct-implication-normalised} $\|u^{J,\lambda^*}_\aaa\|_{L^2(\Omega)}=1$;
    \item \label{item:direct-implication-exact-eigenfunction} $\left\|u^{J,\lambda^*}_\aaa-u^*\right\|_{H^1(\Omega)}\leq \eps$.
    \end{enumerate}

\end{proposition}
\begin{proof}
Let $(\lambda_n^*)_{n=1}^{Q_K}$ the eigenvalues of \eqref{EVP-radial} that are at most equal to $K$ (repeated according to their multiplicity) and $(\xi_n^*)_{n=1}^{Q_K}$ an $L^2(\Omega)$-orthonormal set of corresponding eigenfunctions. By the Runge property (Proposition \ref{prop:runge}), for each of those eigenfunctions, there exists a solution $\bar\xi_n^*$ of \eqref{eq:PDE-lambda} on $B_O(R)$ with $\lambda=\lambda_n^*$ such that:
\begin{equation}\label{eq:runge-for-orthonormal-eigenvectors}\left\|\bar \xi_n^*\vert_\Omega-\xi_n^*\right\|_{H^1(\Omega)}\leq \frac{\eps}{2\sqrt{Q_K}}.\end{equation}
Next, since the eigenfunctions $(\xi_n^*)_{n=1}^{Q_K}$ belong to $H^2(B_O(R))$, Proposition \ref{prop:uJ approximates-u} implies the existence of a positive integer $J_0(\eps,\Omega,V,K)$ such that, for every $n\in \overline{1,Q_K}$ and $J\geq J_0$,
\begin{equation}\label{eq:truncation-orthonormal-eigenvectors}
\left\|\bar \xi_n^{*,J}-\bar\xi_n^*\right\|_{H^1(B_O(R))}\leq \frac{\eps}{2\sqrt{Q_K}},    
\end{equation}
where by $\bar \xi_n^{*,J}$ we denote the truncation of the angular Fourier expansion of $\xi_n^*$, as in \eqref{eq:def-uJ}. 

Since the eigenfunction $u^*$ is a linear combination of the functions $\xi_n^*$ in the eigenspace of $\lambda^*$, we can define the function $\bar u^*$ as the linear combination of the extended functions $\bar \xi_n^*$ with the same coefficients. Therefore, the estimates \eqref{eq:runge-for-orthonormal-eigenvectors}-\eqref{eq:truncation-orthonormal-eigenvectors}, together with the inequality between the arithmetic mean and the quadratic mean, imply that the angular Fourier truncation $\bar u^{*,J}$ of $\bar u^*$ satisfies:
\begin{equation}\label{eq:truncation-near-eigenfunction}
\|\bar u^{*,J}\vert_\Omega -u^*\|_{H^1(\Omega)}\leq \eps \|u^*\|_{L^2(\Omega)}=\eps.
\end{equation}
Moreover, $\bar u^{*,J}$ satisfies \eqref{eq:PDE-lambda} on $B_O(R)$ with $\lambda=\lambda^*$, therefore, by Corollary \ref{cor:Fourier-Bessel-Series-ujLambda}, there exists a vector $\aaa\in \RR^{2J+1}$ such that 
\[u^{J,\lambda^*}_\aaa=\frac{\bar u^{*,J}}{\|\bar u^{*,J}\|_{L^2(\Omega)}}.\]
Eventually, the estimate \eqref{eq:truncation-near-eigenfunction} leads to the statements \ref{item:direct-implication-normalised} and \ref{item:direct-implication-exact-eigenfunction} of the conclusion, which in turn imply, via the theory of traces \cite[p.~315]{brezis}, the estimate on the boundary of $\Omega$ in statement \ref{item:direct-implication-boundary}. 
\end{proof}

\subsection{Continuity in $\lambda$ for the Fourier coefficients}
\label{sec:stability-lambda}
Proposition~$\ref{prop:direct-implication-exact-lambda}$ shows that, if in the first step of the proposed algorithm (Section~\ref{sec:description-of-method}) the exact eigenvalue $\lambda^*$ coincides with one of the grid points, then the minimum value of the quotient $\mathcal{F}^{\lambda^*}$ is indeed small. However, in order to establish that the algorithm is capable of detecting all eigenpairs of~$\eqref{EVP-radial}$, it is necessary to prove Theorem~$\ref{thm:direct-implication-approx-lambda}$, which asserts that the conclusion of Proposition~$\ref{prop:direct-implication-exact-lambda}$ remains valid when the basis functions $(\uj^{\lambda})_{j=0}^J$ are computed using a parameter $\lambda$ sufficiently close to an eigenvalue $\lambda^*$ of~$\eqref{EVP-radial}$.

The derivation of Theorem~$\ref{thm:direct-implication-approx-lambda}$ from Proposition~$\ref{prop:direct-implication-exact-lambda}$ relies on two stability results. The first one, an interior observability estimate stated in Proposition~$\ref{prop:H2-stability-J}$, provides an upper bound depending on $J$ and $\Omega$ for the $H^2(B_O(R))$ norm of functions of the form~$\eqref{eq:def-UJLambdaAlpha}$ that are normalized in $L^2(\Omega)$. The second, Proposition~$\ref{prop:uj-stable-in-lambda}$, establishes the stability of the basis functions $\uj^{\lambda}$ in the space $H_{j,r}([0,R])$ under small perturbations of the parameter $\lambda$.

\begin{proposition}\label{prop:H2-stability-J}
Let $R>0$,  $\tilde R\in (0,R]$, $V\in C^1([0,R])$ taking values in $[1,\infty)$; also let $K>1$ and $J$ a positive integer. There exists a constant $C(R,\tilde{R},V,K,J)>0$ such that for every $\lambda\in [1,K]$ and any vector $\aaa\in\RR^{2J+1}$, the function $u^{J,\lambda}_\aaa$ defined in \eqref{eq:def-UJLambdaAlpha} satisfies
    \[\|u^{J,\lambda}_\aaa\|_{H^2(B_O(R))}\leq C(R,\tilde R, V,K,J)\|u^{J,\lambda}_\aaa\|_{L^2(B_O(\tilde{R}))}.\]
\end{proposition}
\begin{proof}
Proposition \ref{prop:spaces-equivalence} implies that, for every $j\geq 0$, $\uj^{\lambda}(r)\cos(j\theta)$ is a weak solution of the equation \eqref{eq:PDE-lambda} on $B_O(R)$.
The interior regularity theory for elliptic equations \cite[Theorem 8.8 and Problem 8.2]{GilbargTrudinger2001} implies that:
 \[\|\uj^{\lambda}(r)\cos(j\theta)\|_{H^2(B_O(\tilde R/2))}\leq C(\tilde R,V,K)\|\uj^{\lambda}(r)\cos(j\theta)\|_{L^2(B_O(\tilde{R}))}=C(\tilde R,V,K)\|\uj^{\lambda}\|_{L_r([0,\tilde R])}.\]
 The theory of traces for $H^2$ functions \cite[p.~316]{brezis} implies that:
 \[\|\uj^{\lambda}(r)\cos(j\theta)\|_{H^\frac 3 2 (\partial B_O(\tilde{R}/2))}+\|\partial_\nu [\uj^{\lambda}(r)\cos(j\theta)]\|_{H^\frac 1 2 (\partial B_O(\tilde{R}/2))}\leq C(\tilde R,V,K)\|\uj^{\lambda}\|_{L_r([0,\tilde R])},\]
 where $\partial_\nu$ stands for the normal derivative at the boundary of the ball in question. Then, Lemma \ref{lem:fractional-norms} implies that:
 \[\left|\uj^{\lambda}\left(\textstyle \frac{\tilde R}{2}\right)\right|+\left|\left(\uj^{\lambda}\right)'\left(\textstyle\frac{\tilde R}{2}\right)\right|\leq C(\tilde R,V,K,J)\|\uj^{\lambda}\|_{L_r([0,\tilde R])}.\]
 Next, Corollary \ref{cor:weak-to-ODE} implies that the pair $U\coloneqq \left(\begin{array}{c}\uj^{\lambda}\\ \left(\uj^{\lambda}\right)'\end{array}\right)$ satisfies \eqref{eq:U-integral-equation}. Since the supremum norm of the matrix $\mathcal {A}_{j,r}$ in \eqref{eq:U-integral-equation} is bounded by a constant $C(\tilde{R},V,K,J)$, Gr\"onwall's inequality implies that:
  \begin{equation}\label{eq:estimates-from-gronwall}
  \left|\uj^{\lambda}(r)\right|+\left|\left(\uj^{\lambda}\right)'(r)\right|\leq C(R,\tilde R,V,K,J)\|\uj^{\lambda}\|_{L_r([0,\tilde R])},\quad \text{for every } r\in \left[\textstyle\frac{\tilde{R}}{2},R\right].\end{equation}
  This inequality has two consequences: first, \[\|\uj^{\lambda}\|_{L_r([\tilde R,R])} \leq C(R,\tilde R,V,K,J)\|\uj^{\lambda}\|_{L_r([0,\tilde R])},\]
  from which we deduce that:
  \begin{equation}\label{eq:L2-observability-J}
  \|\uj^{\lambda}(r)\cos(j\theta)\|_{L^2(B_O(R))}=\sqrt{\pi} \|\uj^{\lambda}\|_{L_r([0,R])}\leq C(R,\tilde R,V,K,J)\|\uj^{\lambda}\|_{L_r([0,\tilde R])}.
  \end{equation}
The second consequence of \eqref{eq:estimates-from-gronwall} is deduced via Lemma \ref{lem:fractional-norms}:
 \[\|\uj^{\lambda}(r)\cos(j\theta)\|_{H^\frac 3 2 (\partial B_O(R))}+\|\partial_\nu [\uj^{\lambda}(r)\cos(j\theta)]\|_{H^\frac 1 2 (\partial B_O(R))}\leq C(R,\tilde R,V,K,J)\|\uj^{\lambda}\|_{L_r([0,\tilde R])}.\]
 The theory of traces \cite[p.~316]{brezis} implies the surjectivity of the pair of trace and normal derivative operators: 
 \[(\rm{Tr},\partial_\nu):H^2(B_O(R))\to H^\frac 3 2 (\partial B_O(R))\times H^\frac 1 2 (\partial B_O(R)).\]
 Therefore the Open Mapping Theorem \cite[Theorem 2.6]{brezis} implies that there exists $\psi\in H^2(B_O(R))$ such that:
 \[\|\psi\|_{H^2(B_O(R))}\leq C(R,\tilde R,V,K,J)\|\uj^{\lambda}\|_{L_r([0,\tilde R])} \]
 and, in the sense of traces,
\begin{equation}\label{eq:existence-psi-with-suitable-boundary-values}
\left[\psi-\uj^{\lambda}(r)\cos(j\theta)\right]\Big\vert_{\partial B_O(R)}\equiv 0 \quad \text{ and }\quad \partial_\nu\left[\psi-\uj^{\lambda}(r)\cos(j\theta)\right]\Big\vert_{\partial B_O(R)}\equiv 0.
\end{equation}

Using \eqref{eq:L2-observability-J} and \eqref{eq:existence-psi-with-suitable-boundary-values}, the regularity theory up to the boundary for elliptic equations \cite[Theorem 8.12]{GilbargTrudinger2001} implies that:
 \[\|\uj^{\lambda}(r)\cos(j\theta)\|_{H^2(B_O(R))}\leq C(R,\tilde R,V,K,J)\|\uj^{\lambda}\|_{L_r([0,\tilde R])}.\]

 In order to deduce the conclusion of the proposition, we note that the inequality above remains valid if we replace $\cos(j\theta)$ with $\sin(j\theta)$, with $j>0$. Therefore, the triangle inequality and the inequality between the arithmetic mean and the quadratic mean imply:
\[\begin{aligned}
\|u^{J,\lambda}_\aaa\|_{H^2(B_O(R))}^2&\leq (2J+1)\left[(\alpha_0^c)^2 \|\uz^{\lambda}(r)\|_{H^2(B_O(R))}^2+\sum_{j=1}^J \left((\alpha_j^c)^2+(\alpha_j^s)^2\right)  \|\uj^{\lambda}(r)\cos(j\theta)\|_{H^2(B_O(R))}^2\right]\\
&\leq C(R,\tilde{R},V,K,J)\left[(\alpha_0^c)^2 \|\uz^{\lambda}\|_{L_r([0,\tilde R])}^2+\sum_{j=1}^J \left((\alpha_j^c)^2+(\alpha_j^s)^2\right)  \|\uj^{\lambda}\|_{L_r([0,\tilde{R}])}^2\right]\\
&= C(R,\tilde{R},V,K,J) \| u^{J,\lambda}_\aaa\|_{L^2(B_O(\tilde{R}))}^2.\hspace{7.5cm} \qedhere
\end{aligned}\]
\end{proof}

Next, we state the $H_{j,r}([0,R])$ stability result for the basis functions $\uj^{\lambda}$ with respect to the value $\lambda\in [1,K]$. 

\begin{proposition}\label{prop:uj-stable-in-lambda} In the hypotheses of Proposition \ref{prop:H2-stability-J}, let $\lambda,\tilde \lambda\in [1,K]$ and the basis functions $\uj^{\lambda}$ and $\uj^{\tilde \lambda}$ given by Theorem \ref{thm:solution-space-1dim}. Then, there exists a threshold $\mu(R,V,K,J)>0$, a constant $C(R,V,K,K)>0$, and a choice of sign $\pm$ such that:
\[\left\|\uj^{\tilde \lambda}\pm\uj^{\lambda}\right\|_{j,r}\leq C(R,V,K,J)|\tilde \lambda-\lambda|,\]
provided that $|\tilde \lambda-\lambda|\leq \mu(R,V,K,J)$ and $j\leq J$.
\end{proposition}
In order to prove this proposition, we need to construct the Neumann counterpart of the Dirichlet operator $A_{j,V}$ defined in Section \ref{sec:weaks-solutions-basis-functions}. Namely, let $B_{j,V}:H_{j,r}([0,R])\to H_{j,r}([0,R])^*$ be the operator corresponding to the bilinear form $a_{j,V}$ in \eqref{eq:def-ajV}, that is: 
\[\langle B_{j,V} v,w\rangle_{H_{j,r}([0,R])^*,H_{j,r}([0,R])}\coloneqq a_{j,V}(v,w)=a_j(v,w)+(V\,v,w)_r, \quad \forall v,w\in H_{j,r}.\]
Since $V\geq 1$, the operator $B_{j,V}$ is coercive and by Lemma \ref{lem:Hjr0-comapct-in-Lr} we can construct the compact self-adjoint solution operator $B_{j,V}^{-1}$ on $L_r([0,R])$.
The following Proposition characterises $B_{j,V}$ as the Neumann operator associated to the weak form \eqref{EVP-j-W} of the ODE \eqref{eq:uj-ODE}:
\begin{proposition}\label{prop:neumann-as-ode}
    Let $R>0$ and the potential $V\in C^1([0,R])$ taking values in $[1,\infty)$. For every $\lambda\geq 1$, non-negative integer $j$, and $v\in H_{j,r}([0,R])$, the following statements are equivalent:
    \begin{enumerate}[label={(\roman*)}]
        \item \label{item:neumann-op-1} $B_{j,V} \,v= \lambda v$;
        \item \label{item:neumann-op-2} $v\in H^2(\mathcal{I})$ for every compact interval $\mathcal{I}\subset (0,R]$, satisfies \eqref{eq:uj-ODE} in the sense of weak derivatives and $v'(R)=0$ in the sense of traces. 
    \end{enumerate}
\end{proposition}
\begin{proof}
If the function $v$ satisfies $B_{j,V}\, v=\lambda v$, then by Corollary \ref{cor:weak-to-ODE}, $v$ belongs to $H^2(\mathcal{I})$ and satisfies the ODE \eqref{eq:uj-ODE}, which in turn is the same as $\tilde{A}_{j,V} v=\lambda v$ in the sense of weak derivatives, where $\tilde A_{j,V}$ is defined in \eqref{eq-def-AjTilde}. Therefore, it is enough to prove the equivalence \ref{item:neumann-op-1}-\ref{item:neumann-op-2} for functions $v\in H^2(\mathcal{I})$ that satisfy $\tilde{A}_{j,V} v=\lambda v$.

Fur such a function $v$, Lemma \ref{lem:density-in Hjr} and an integration by parts argument implies that the equality $B_{j,V} \, v=\lambda v$ is equivalent to:
\begin{equation}\label{eq:integration-parts-neumann}
\int_0^R (\tilde{A}_{j,V}\, v)(r) w(r) \, r\,\dd r + v'(R) w(R)R =\lambda \int_0^R v(r)w(r)r\,\dd r,\quad  \forall w\in C_c^\infty((0,R]),\end{equation} and, since $\tilde{A}_{j,V} v=\lambda v$, it follows that $B_{j,V} \, v=\lambda v$ is equivalent to $v'(R)=0$.
\end{proof}
So far, in Theorem~$\ref{thm:solution-space-1dim}$ and Proposition~$\ref{prop:neumann-as-ode}$, we have characterised the homogeneous Dirichlet and Neumann problems associated with the bilinear form $a_{j,V}(\cdot,\cdot)-\lambda(\cdot,\cdot)_r$, for an arbitrary $\lambda \geq 1$, as solutions of the ordinary differential equation~$\eqref{eq:uj-ODE}$ subject to appropriate boundary conditions at $r=R$. Consequently, the Cauchy–Lipschitz theorem for ordinary differential equations, together with the uniqueness up to a multiplicative constant of weak solutions to~$\eqref{eq:uj-ODE}$ in the sense of~$\eqref{EVP-j-W}$, implies that the Dirichlet and Neumann spectra associated with~$\eqref{eq:uj-ODE}$ are disjoint. Since both spectra are discrete, it follows that their intersections with the interval $[1,K]$ are separated by a positive constant, which is quantified in the following proposition:

\begin{proposition}\label{prop:gap-dirichlet-neumann} Let $R>0$ and $V\in C^1([0,R])$ taking values in $[1,\infty)$. For every $K>1$ and each positive integer $J$, there exists a value $\delta_{R,V,K,J}\in (0,1)$ such that if $\lambda^j_n$ and $\zeta^j_m$ belong to $[1,K]$ and are, respectively, Dirichlet and Neumann eigenvalues associated to the bilinear form $a_{j,V}(\cdot,\cdot)$ on $L_r([0,R])$ for some $j\in \overline{0,J}$, then
\[|\lambda^j_n-\zeta^j_m|>\delta_{R,V,K,J}.\]
\end{proposition}
\begin{proof}
    Since the spectra of both operators $A_{j,V}$ and $B_{j,V}$ are discrete, then it is sufficient to prove that there is no value $\lambda\in [1,K]$ that is both a Dirichlet and Neumann eigenvalue of \eqref{eq:uj-ODE}. Indeed, if this would be the case, than the normalised solution $\uj^\lambda$ of \eqref{EVP-j-W} given by Theorem \eqref{thm:solution-space-1dim} would satisfy $\uj^\lambda(R)=(\uj^\lambda)'(R)=0$. However, Corollary \ref{cor:weak-to-ODE} and Gr\"onwall's inequality would imply that $\uj^\lambda$ is identically zero, which is a contradiction.
\end{proof}

The last tool needed for the proof of Proposition \ref{prop:uj-stable-in-lambda} is the following result that loosely speaking states that if $\lambda$ is far from an eigenvalue of a bilinear form $b(\cdot,\cdot)$ defined on a subspace of $L_r([0,R])$, then the form $b(\cdot,\cdot)-\lambda(\cdot,\cdot)_r$ admits a bounded solution operator.

\begin{proposition}\label{prop:operator-outside-of-spectrum}
Let $H$ be a Hilbert space that is compactly embedded in $L_r([0,R])$ such that the embedding operator is contractive. Let $b:H\times H\to \RR$ be a symmetric, continuous and coercive bilinear form, the latter meaning that there exists a constant $M>0$ such that:
\begin{equation}\label{eq:b-coercive}
b(w,w)\geq \frac{1}{M}\|w\|^2_H, \quad \forall v\in H.
\end{equation}
If $v\in H$ and $f\in L_r([0,R])$ are such that:
\begin{equation}
\label{eq:generic-elliptic-lambda*}
b(v,w)-\lambda (v,w)_r=(f,w)_r, \quad\forall w\in H\end{equation}
and $\lambda\geq 1 $ is not an eigenvalue of $b$, 
then:
\[\|v\|_{r}\leq \frac{\|f\|_r}{\displaystyle \min_{\lambda^* \text{\emph{ eigenval. of }}b} |\lambda-\lambda^*| }\]
and 
\[\|v\|_{H}\leq M \|f\|_r\left[1+\frac{|\lambda|}{\displaystyle \min_{\lambda^*\text{\emph{ eigenval. of }}b} |\lambda-\lambda^*| }\right].\]

\emph{Note:} We say that $\lambda^*$ is an eigenvalue of $b$ if there exists a function $v\in H\setminus\{0\}$ such that $b(v,w)-\lambda (v,w)_r=0, \,\forall w\in H$.
\end{proposition}
A proof of this proposition can be found in Appendix \ref{acaret:proofs}.\\

We are now able to prove Proposition \ref{prop:uj-stable-in-lambda} by applying Proposition \ref{prop:operator-outside-of-spectrum} to either the operator $A_{j,V}$ or $B_{j,V}$ -- more precisely, to the bilinear form $a_{j,V}$ on either $H_{j,r}^0([0,R])$ or $H_{j,r}([0,R])$ -- depending on whether the values $\lambda$ and $\tilde \lambda$ are close to either the Dirichlet or the Neumann spectrum of \eqref{eq:uj-ODE}. 
\begin{proof}[Proof of Proposition \ref{prop:uj-stable-in-lambda}]
With the notation in Proposition \ref{prop:gap-dirichlet-neumann}, we consider the threshold $\mu(R,V,K,J)\coloneqq \frac{1}{3}\delta_{R,V,K+1,J}$. Then, since $\lambda,\tilde \lambda\in [1,K]$ satisfy $|\tilde \lambda-\lambda|\leq\mu(R,V,K,J)$ and the distance between the Dirichlet and Neumann spectra of \eqref{eq:uj-ODE} intersected with the interval $[1,K+1]$ is smaller than  $3\mu(R,J,V,K)$, at least one of the following statements holds true:
\begin{enumerate}[label={\Roman*.}]
\setlength{\itemsep}{0.3cm}
    \item $\displaystyle \inf_{\lambda^j_n \text { eigenval. of } A_{j,V}} |\lambda-\lambda^j_n|\geq\mu(R,J,V,K)\quad \text{and}\quad \inf_{\lambda^j_n \text { eigenval. of } A_{j,V}} |\tilde\lambda-\lambda^j_n|\geq\mu(R,J,V,K)$;
    \item $\displaystyle \inf_{\zeta^j_m \text { eigenval. of } B_{j,V}} |\lambda-\zeta^j_m|\geq\mu(R,J,V,K)\quad \text{and}\quad \inf_{\zeta^j_m \text { eigenval. of } B_{j,V}} |\tilde \lambda-\zeta^j_m|\geq\mu(R,J,V,K)$.
\end{enumerate}
We will study these cases separately. If Statement I holds true, then, in particular, $\uj^{\lambda}$ is not a Dirichlet eigenfunction of $a_{j,V}$. Then, it follows that, in the sense of traces $\uj^{\lambda}(R)\neq 0$, so we can consider the following function $v\in H_{j,r}^0([0,R])$,
\[v(r)\coloneqq \uj^{\tilde \lambda}(r)-\frac{\uj^{\tilde \lambda}(R)}{\uj^{\lambda}(R)} \uj^{\lambda}(r).\]
Since $\uj^{\lambda}$ and $\uj^{\tilde \lambda}$ satisfy \eqref{EVP-j-W} for $\lambda$ and $\tilde \lambda$, respectively, this means that:
\[A_{j,V} v - \lambda v = (\tilde \lambda-\lambda) \uj^{\tilde \lambda}. \]

Since, by definition, $\|\uj^{\tilde \lambda}\|_r=1$, we apply Proposition \ref{prop:operator-outside-of-spectrum} and obtain that:
\begin{equation}\label{eq:dirichlet-error-u*-u**-not-normalised}
\|v\|_r\leq \frac{|\tilde \lambda-\lambda|}{\mu(R,V,K,J)} \quad \text{ and } \quad \|v\|_{j,r}\leq |\tilde \lambda-\lambda|\left[1+\frac{K}{\mu(R,V,K,J)}\right].
\end{equation}
First, this means that, since $\|\uj^{\lambda}\|_r=\|\uj^{\tilde \lambda}\|_r=1$, the triangle inequality implies:
\[\left|1-\left|\frac{\uj^{\tilde \lambda}(R)}{\uj^{\lambda}(R)}\right|\right|\leq \frac{|\tilde \lambda-\lambda|}{\mu(R,V,K,J)},\]
which implies that there exists a choice of sign $\pm$ such that:
\begin{equation}\label{eq:dirichlet-ratio-near-1}
\left|1\pm\frac{\uj^{\tilde \lambda}(R)}{\uj^{\lambda}(R)}\right|\leq \frac{|\tilde \lambda-\lambda|}{\mu(R,V,K,J)}.\end{equation}
With the same choice of the sign, we write the triangle inequality:
\begin{equation}\label{eq:lambda-stability-final-inequality}
\begin{aligned}
\left\|\uj^{\tilde\lambda}\pm \uj^{\lambda}\right\|_{j,r}&\leq \left\|\uj^{\tilde \lambda}-  \frac{\uj^{\tilde \lambda}(R)}{\uj^{\lambda}(R)}\uj^{\lambda}\right\|_{j,r}+ \left\|\left(1\pm\frac{\uj^{\tilde \lambda}(R)}{\uj^{\lambda}(R)}\right)\uj^{\lambda}\right\|_{j,r}\\
&\leq |\tilde \lambda-\lambda|\left[1+\frac{K+\|\uj^{\lambda}\|_{j,r}}{\mu(R,V,K,J)}\right],
\end{aligned}\end{equation}
where the last inequality follows by \eqref{eq:dirichlet-error-u*-u**-not-normalised} and \eqref{eq:dirichlet-ratio-near-1}. Since $\|\uj^{\lambda}\|_r=1$, Propositions \ref{prop:spaces-equivalence} and \ref{prop:H2-stability-J} Imply that:
\[\|\uj^{\lambda}\|_{j,r}= \frac{1}{\sqrt{\pi}}\|\uj^{\lambda}\|_{H^1(B_O(R))}\leq C(R,V,K,J) \|\uj^{\lambda}\|_{r}= C(R,V,K,J).\]
This finishes the proof of Case I.

In the second case, we proceed similarly by noticing that $\uj^{\lambda}$ is not a Neumann eigenfunction of $a_{j,V}$, so $(\uj^{\lambda})'(R)\neq 0$. Therefore, we use Proposition \ref{prop:operator-outside-of-spectrum} for the operator $B_{j,V}$ and prove that the difference 
\[v(r)\coloneqq \uj^{\tilde \lambda}(r)-\frac{(\uj^{\tilde \lambda})'(R)}{(\uj^{\lambda})'(R)} \uj^{\lambda}(r)\]
satisfies \eqref{eq:dirichlet-error-u*-u**-not-normalised}. The conclusion of Case II follows as above.
\end{proof}

Eventually, we combine Propositions \ref{prop:direct-implication-exact-lambda}, \ref{prop:uj-stable-in-lambda} and \ref{prop:H2-stability-J} to write the:
\begin{proof}[Proof of Theorem \ref{thm:direct-implication-approx-lambda}]
Let the positive integer $J_0$ and the vector $\aaa\in \RR^{2J+1}$ provided by Proposition \ref{prop:direct-implication-exact-lambda} such that $\|u^{J,\lambda^*}_\aaa\|_{L^2(\Omega)}=1$ and $\|u^{J,\lambda^*}_\aaa-u^*\|_{H^1(\Omega)}\leq \eps$.

Roughly speaking, we will show that, for $|\lambda-\lambda^*|$ small enough, the quantity $\| u^{J,\lambda}_\aaa-u^{J,\lambda^*}_\aaa\|_{H^1(\Omega)}$ can be made arbitrarily small. Then, the conclusion is obtained by normalising the vector $\aaa$ such that the norm $\|u^{J,\lambda}_\aaa\|_{L^2(\Omega)}$ becomes equal to one.

Indeed, let $\tilde{R}=\tilde{R}(\Omega)>0$ such that $B_O(\tilde{R})\subseteq \Omega$. Proposition \ref{prop:H2-stability-J} and the fact that $\|u^{J,\lambda^*}_\aaa\|_{L^2(\Omega)}=1$ imply that:
\begin{equation}\label{eq:UJ-bounded-L2-1}
\|u^{J,\lambda^*}_\aaa\|_{L^2(B_O(R))}\leq C(R,\tilde{R},V,K,J) \|u^{J,\lambda^*}_\aaa\|_{L^2(B_O(\tilde R))}\leq C(R,\tilde{R},V,K,J).\end{equation}
Since $\|\uj^{\lambda^*}\|_r=1$ for every $j\geq 0$, this implies that:
\begin{equation}\label{eq:UJ-bounded-L2-2}
\|\aaa\|^2=(\alpha_0^c)^2+\sum_{j=1}^J \left((\alpha_j^c)^2+(\alpha_j^s)^2\right)\leq  C(R,\tilde{R},V,K,J).\end{equation}

Next, we construct a new vector $\tilde \aaa$ such that, for every $j\in \overline{0,J}$, $\tilde \alpha_j^c\coloneqq\pm \alpha_j^c$ and, if $j\geq 1$, $\tilde \alpha_j^s\coloneqq \pm \alpha_j^s$, with the choices of signs given by Proposition \ref{prop:uj-stable-in-lambda}. As a result, the triangle inequality, together with Proposition \ref{prop:spaces-equivalence}, and the inequality between the arithmetic mean and quadratic mean, imply that:
\begin{equation}\label{eq:UJ-mean-inequality}
\begin{aligned}
\|u^{J,\lambda}_{\tilde \aaa}-u^{J,\lambda^*}_\aaa\|_{H^1(B_O(R))}^2&\leq (2J+1)\left[(\alpha_0^c)^2 \|\uz^{\lambda}\pm\uz^{\lambda^*}\|_{0,r}^2+\sum_{j=1}^J \left((\alpha_j^c)^2+(\alpha_j^s)^2\right)  \|\uj^{\lambda}\pm\uj^{\lambda^*}\|_{j,r}^2\right]\\
&\leq C(R,\tilde R,V,K,J)\, \max_{j=\overline{0,J}}\|\uj^{\lambda}\pm\uj^{\lambda^*}\|_{j,r}^2.
\end{aligned}\end{equation}
We recall that the signs were chosen such that Proposition \ref{prop:uj-stable-in-lambda} holds true. This allows us to consider a threshold $\mu(R,\tilde{R},V,K,J)>0$ such that, if  $|\lambda-\lambda^*|< \mu(R,\tilde{R},V,K,J)$, then
\[\| u^{J,\lambda}_{\tilde \aaa}-u^{J,\lambda^*}_\aaa\|_{H^1(\Omega)}\leq |\lambda-\lambda^*| C(R,\tilde R,V,K,J).\]
The conclusion follows by also taking into account that the $H^\frac{1}{2}(\partial \Omega)$-norm is controlled by the $H^1(\Omega)$-norm.
\end{proof}

\section{Numerical approximation of the basis functions $\uj^\lambda$ via one-dimensional Finite Element Method}
\label{sec:FEM}
The aim of this section is to provide a method for approximating the basis functions $\uj^{\lambda}$. We recall that they are the unique (up to a change of sign) $H_{j,r}([0,R])$ solutions to the problem \eqref{EVP-j-W} with $\|\uj^{\lambda}\|_r=1$. Namely, they satisfy:
\begin{equation}\label{EVP-j-W-short}
a_j(\uj^{\lambda},w)+(V \uj^{\lambda},w)_r=\lambda(\uj^{\lambda},w)_r, \quad \forall w \in H_{j,r}^0([0,R]).
\end{equation}
Inspired by the fact that this is a variational problem, we propose a Finite Element approximation of its solution $\uj^\lambda$, based on a Galerkin discretisation of the space $H_{j,r}([0,R])$.

\subsection{One-dimensional Finite Element spaces. Error estimates for the linear interpolation}
\label{sec:interpolation-estimates}

In what follows, we will introduce a family of finite-dimensional subspaces of $H_{j,r}([0,R])$ that will allow us to approximate the solutions of \eqref{EVP-j-W-short}. Let us consider an equidistant grid of width $h>0$ on the interval $[0,R]$:\[0\eqqcolon r_0^h<r_1^h\coloneqq h<r_2^h\coloneqq 2h<\ldots<r_{N_h}^h\coloneqq R,\]
where $N_h\coloneqq \frac{R}{h}$ is assumed to be an integer.

Associated to this division, we define the piecewise affine functions $\phi_i:[0,R]\to \RR$ such that $\phi_i(r^h_i)=1$ and $\phi_i(r^h_m)=0$, $\forall m\neq i$. More precisely, if $i\in\overline{1,N_h-1}$ then
\[\phi_i(r)\coloneqq\begin{cases}
    \frac{r-r_{i-1}^h}{h}, & r\in \left[r_{i-1}^h,r_i^h\right];\\
    \frac{r_{i+1}^h-r}{h}, & r\in \left[r_i^h,r_{i+1}^h\right];\\
    0, & \text{otherwise}.
\end{cases}\]
If $i=0$, then
\[\phi_0(r)\coloneqq\begin{cases}
\frac{r_1^h-r}{h}, & r\in \left[0,r_1^h\right];\\
    0, & \text{otherwise},
\end{cases}\]
and, if $i=N_h$, then
\[\phi_N(r)\coloneqq\begin{cases}
\frac{r-r_{N_h}^h}{h}, & r\in \left[r_{N_h-1}^h,r_{N_h}^h\right];\\
    0, & \text{otherwise}.
\end{cases}\]

We aim to use the functions $(\phi_i)_{i=1}^{N_h}$ as Finite Element basis functions for solving the problem \eqref{EVP-j-W-short}. For this purpose, we note that Proposition \ref{prop:Hjr-is-zero-in-zero} implies that, for $j\geq 1$, every element $v\in H_{j,r}([0,R])$ satisfies $v(0)=0$. We also take into account Proposition \ref{prop:spaces-equivalence} \ref{item:h01-second} in order to define the following finite dimensional subspaces of $H_{j,r}([0,R])$ and $H_{j,r}^0([0,R])$:
\begin{itemize}
    \item For a positive integer $j$,
\[\begin{aligned}
W_{j,h}& \coloneqq {\rm span}\{\phi_1,\phi_2,\ldots, \phi_{N_h}\}\subset H_{j,r}([0,R]);\\
W_{j,h}^0 &\coloneqq {\rm span}\{\phi_1,\phi_2,\ldots, \phi_{N_h-1}\}\subset H_{j,r}^0([0,R]).
\end{aligned}\]
\item For $j=0$, since Proposition \ref{prop:Hjr-is-zero-in-zero} does not apply, we define:
\[\begin{aligned}
W_{0,h}& \coloneqq{\rm span}\{\phi_0,\phi_1,\phi_2,\ldots, \phi_{N_h}\}\subset H_{0,r}([0,R]);\\
W_{0,h}^0 &\coloneqq {\rm span}\{\phi_0,\phi_1,\phi_2,\ldots, \phi_{N_h-1}\}\subset H_{0,r}^0([0,R]).
\end{aligned}\]
\end{itemize}

The first step in constructing a Galerkin numerical scheme for the problem~$\eqref{EVP-j-W-short}$ is to show that the spaces $W_{j,h}$ and $W_{j,h}^0$ indeed approximate the spaces $H_{j,r}([0,R])$ and $H_{j,r}^0([0,R])$, respectively, as $h$ approaches zero. This result is stated in the following proposition and proved in Appendix \ref{acaret:proofs-2}.

\begin{proposition}\label{prop:fem-interpolator}
Let $j$ be a non-negative integer and $v\in H_{j,r}([0,R])$ that satisfies $v(r)\cos(j\theta)\in H^2(B_O(R))$. Then there exists a function $\tilde v_h\in W_{j,h}$ such that:
\begin{enumerate}[label={\roman*)}]
\setlength{\itemsep}{0.4em}
    \item $\|v-\tilde v _h\|_r\leq \frac{2\sqrt{7}}{\sqrt{\pi}}\, h^2\, \|v(r)\cos(j\theta)\|_{ H^2(B_O(R))};$
    \item $\|v-\tilde v _h\|_{j,r}\leq \frac{3\sqrt{14}}{\sqrt{\pi}}\, h\, \|v(r)\cos(j\theta)\|_{ H^2(B_O(R))}.$
\end{enumerate}
Moreover, if $v\in H_{j,r}^0([0,R])$, then we can find such a function $\tilde v_h$ in $W_{j,h}^0$.\\ We note that $\tilde v_h$ is constructed as the linear interpolator of $v$ in the nodes $(r^h_i)_{i=1}^{N_h}$.
\end{proposition}

\subsection{Quadrature for the term involving the potential $V$}
\label{sec:gauss-quadrature}
Given the spaces $W_{j,h}$ and $W_{j,h}^0$, to achieve a complete discretisation of \eqref{EVP-j-W-short} for a general potential $V \in C^1([0,R])$, it remains to introduce an appropriate quadrature scheme for the discrete counterpart of the second term in \eqref{EVP-j-W-short}, which has the form:
\[(V\, v^h,w^h)_r =\int_0^R V(r) v^h(r)w^h(r)\dd r, \quad v^h,w^h\in W_{j,h}. \]
We choose to apply a three-point Gauss quadrature on each of the intervals $[r^h_i,r^h_{i+1}]$, $i\in \overline{1,N_h-1}$:
\[\begin{aligned}(V\,v^h,w^h)_r &=\int_0^R V(r) v^h(r)w^h(r)r\,\dd r =\sum_{i=0}^{N_h-1} \int_{r^h_i}^{r^h_{i+1}} V(r) v^h(r)w^h(r)r\,\dd r \\
&\simeq \sum_{i=0}^{N_h-1}\underbrace{\textstyle\left[\frac{5}{9} V(r)v^h(r)w^h(r)r\vert_{r=r^h_{i+\omega_1}}+\frac{8}{9} V(r)v^h(r)w^h(r)r\vert_{r=r^h_{i+\omega_2}}+\frac{5}{9} V(r)v^h(r)w^h(r)r\vert_{r=r^h_{i+\omega_3}}\right]}_{\text{denoted by }\,\mathcal{G}(Vv^hw^h,r^h_i,r^h_{i+1})},\end{aligned}
\]
where $\omega_1\coloneqq \frac{1}{2}-\frac{\sqrt{3}}{\sqrt{5}}$, $\omega_2\coloneqq \frac{1}{2}$, $\omega_3\coloneqq \frac{1}{2}+\frac{\sqrt{3}}{\sqrt{5}}$ and $r^h_{i+\omega}\coloneqq (1-\omega) r^h_i+\omega\, r^h_{i+1}$.

This choice is justified by the following reasoning: consider $V_h\in L^\infty([0,R])$ such that, for every  $i\in \overline{1,N_h-1}$, $V_h\vert_{[r^h_i,r^h_{i+1}]}$ is the Lagrange interpolation polynomial of $V\vert_{[r^h_i,r^h_{i+1}]}$ in the points of the Gauss quadrature. More precisely, for $r\in (r^h_i,r^h_{i+1})$,
\[\begin{aligned}V_h(r)\coloneqq V(r^h_{i+\omega_1}) \frac{(r^h_{i+\omega_2}-r)(r^h_{i+\omega_3}-r)}{(r^h_{i+\omega_2}-r^h_{i+\omega_1})(r^h_{i+\omega_3}-r^h_{i+\omega_1})}&+V(r^h_{i+\omega_2}) \frac{(r^h_{i+\omega_1}-r)(r^h_{i+\omega_3}-r)}{(r^h_{i+\omega_1}-r^h_{i+\omega_2})(r^h_{i+\omega_3}-r^h_{i+\omega_2})}\\
&+V(r^h_{i+\omega_3}) \frac{(r^h_{i+\omega_1}-r)(r^h_{i+\omega_2}-r)}{(r^h_{i+\omega_1}-r^h_{i+\omega_3})(r^h_{i+\omega_2}-r^h_{i+\omega_3})}.
\end{aligned}\]
Since the three-point Gauss quadrature is exact for polynomials of degree at most five and $V_h$ is quadratic on every interval $[r^h_i,r^h_{i+1}]$, whereas $v^h$ and $w^h$ are affine, it holds that:
\[\mathcal{G}(V_hv^hw^h,r^h_i,r^h_{i+1})=\int_{r^h_i}^{r^h_{i+1}}V_h(r)v^h(r)w^h(r)r\dd r.\]
Moreover, since $V$ and $V_h$ coincide at the points $r^h_{i+\omega_k}$, $k=\overline{1,3}$, the Gauss approximation for $(V\, v^h,w^h)_r$ can be written as:
\[(V v^h,w^h)_r\simeq \sum_{i=0}^{N_h-1}\mathcal{G}(Vv^hw^h,r^h_i,r^h_{i+1}) = (V_h v^h,w^h)_r.
\]
Next, elementary properties of the Lagrange interpolation polynomial imply that, if $V\in C^1([0,R])$, then there exists a universal constant $C>0$ such that:
\begin{equation}\label{eq:error-V-Vh-C1}
\|V_h-V\|_{L^\infty([0,R])}\leq h\, C \|V'\|_{L^\infty([0,R])}.
\end{equation}
Moreover, if the potential $V$ belongs to $C^2([0,R])$, we can improve the error rate to: 
\begin{equation}\label{eq:error-V-Vh-C2}
\|V_h-V\|_{L^\infty([0,R])}\leq h^2\, C \|V''\|_{L^\infty([0,R])}.\end{equation}

Therefore, for every $v^h,w^h\in W_{j,h}$, the error between the actual value of the integral $(V v^h,w^h)_r$ and its three-point Gauss approximation on all intervals $[r^h_i,r^h_{i+1}]$, $i\in \overline{1,N_h-1}$ is bounded as follows:
\begin{equation}\label{eq:quadrature-Lr-error}
    \left|(V_h\, v^h,w^h)_r-(V\, v^h,w^h)_r\right|\leq\begin{cases} h\, C(V)\|v^h\|_r\|w^h\|_r, & \text{if }V\in C^1([0,R]); \\
    h^2 C(V)\|v^h\|_r\|w^h\|_r, & \text{if } V\in C^2([0,R]).
    \end{cases}
\end{equation}

\subsection{Approximating the Fourier coefficients using one-dimensional Finite Element Method}
\label{sec:FEM-converges}
We can now write the equation that will be solved on the computer in order to approximate $\uj^{\lambda}$:

Find $\ujh^{\lambda}\in W_{j,h}$ satisfying $\|\ujh^{\lambda}\|_r=1$ and, for every $w^h \in W_{j,h}^0$,
\begin{equation}\label{EVP-j-W-h}
a_j(\ujh^{\lambda},w^h)+(V_h\, \ujh^{\lambda},w^h)_r=\lambda(\ujh^{\lambda},w^h)_r.
\end{equation}
We remark that, for $j\geq 1$, this is a homogeneous linear system with $N-1$ equations and $N$ unknowns, whereas, for $j=0$, the equation \eqref{EVP-j-W-h} is a homogeneous $N\times (N+1)$ linear system.
The following theorem provides the existence and uniqueness (up to a change of sign) of the solution $\ujh^{\lambda}$, together with error estimates with respect to $\uj^{\lambda}$.
\begin{theorem}\label{thm:ujh-uj-convergence} Let $R>0$ and the potential $V\in C^1([0,R])$ taking values in $[1,\infty)$. Then, for every $K>1$ and positive integer $J$  there exists a threshold $h_0=h_0(R,V,K,J)>0$ and a positive constant $C(R,V,K,J)$ such that, for each $\lambda\in [1,K]$, $j\in \overline{0,J}$ and $h\in (0,h_0)$, the solution space of the linear system \eqref{EVP-j-W-h} is one-dimensional.

Moreover, if $\ujh^{\lambda}$ and $\uj^{\lambda}$ are $L_r([0,R])$-norm normalised solutions of \eqref{EVP-j-W-h} and \eqref{EVP-j-W-short}, respectively, then there exists a choice of sign $\pm$ such that:
\[\|\ujh^{\lambda}\pm\uj^{\lambda}\|_{j,r}\leq h\, C(R,V,K,J).\]
Furthermore, if $V\in C^2([0,R])$, the following improved estimate holds true:
\[\|\ujh^{\lambda}\pm\uj^{\lambda}\|_{r}\leq h^2\, C(R,V,K,J).\]
\end{theorem}

In order to provide a proof of Theorem \ref{thm:ujh-uj-convergence}, we consider the bilinear form on $H_{j,r}([0,R])$ associated to the discrete equation \eqref{EVP-j-W-h}:
\begin{equation}\label{eq:def-ajh}
 a_{j,V_h}(v,w)=a_j(v,w)+(V_h\, v,w)_r, \quad \forall v,w \in H_{j,r}([0,R]).
 \end{equation}
 The fact that $V\geq 1$, together with the inequality \eqref{eq:quadrature-Lr-error}, implies that if $h<\frac{1}{2C\|V'\|_{L^\infty([0,R])}}$, the bilinear form $a_{j,V_h}$ is coercive on $H_{j,r}([0,R])$ with:
 \begin{equation}\label{eq:ajVh-coercive}
 a_{j,V_h}(w,w)\geq \frac{1}{2} \|w\|^2_{j,r}, \quad \forall w\in H_{j,r}([0,R]).
 \end{equation}
 Since this bilinear form is also symmetric and continuous, we can define two associated projection operators: the Dirichlet projection $\Pi_h^0:H_{j,r}^0([0,R])\to W_{j,h}^0$ and the Neumann projection $\Pi_h:H_{j,r}([0,R])\to W_{j,h}$ characterised by the following equalities:
\[\begin{aligned}
    a_{j,V_h}(v-\Pi_h^0 v,w^h)&=0, \quad \forall v\in H_{j,r}^0([0,R]), \forall w^h \in W_{j,h}^0;\\
    a_{j,V_h}(v-\Pi_h v,w^h)&=0, \quad \forall v\in H_{j,r}([0,R]), \forall w^h \in W_{j,h}.
\end{aligned}\]
The following lemma establishes error estimates between functions $v$ belonging to $H^0_{j,r}([0,R])$ or $H_{j,r}([0,R])$ and their Dirichlet and Neumann projections, respectively, provided that the function $v(r)\cos(j\theta)$ belongs to $H^2(B_O(R))$. Furthermore, the lemma contains error estimates for the Dirichlet projection of $v-v(R)\phi_{N_h}$, where $\phi_{N_h}$ is the last Finite Element basis function introduced in Section \ref{sec:interpolation-estimates}. These error estimates are crucial for the proof of Theorem \ref{thm:ujh-uj-convergence}.
\begin{lemma}\label{lem:projection-error} Let $R>0$ and $V\in C^1([0,R])$ taking values in $[1,\infty)$. Let also $0<h_0\leq \frac{1}{2\|V'\|_{L^\infty([0,R])}}$. Then there exists a constant $C(R,V)>0$ such that for every $h\in (0,h_0)$ and every non-negative integer $j$, the following estimates hold true:
\begin{enumerate}[label={\roman*)}]
\item \label{item:projection-D} For every $v\in H_{j,r}^0([0,R])$ such that $v(r)\cos(j\theta)\in H^2(B_O(R))$, 
\[\left\|v-\Pi_h^0 v\right\|_{r}\leq h^2 C(R,V)\,\|v(r)\cos(j\theta)\|_{H^2(B_O(R))}.\]
\begin{center}
    and
\end{center} 
\[\left\|v-\Pi_h^0 v\right\|_{j,r}\leq h\, C(R,V)\,\|v(r)\cos(j\theta)\|_{H^2(B_O(R))}\]
\item \label{item:projection-N} For every $v\in H_{j,r}([0,R])$ such that $v(r)\cos(j\theta)\in H^2(B_O(R))$, 
\[\left\|v-\Pi_h v\right\|_{r}\leq h^2 C(R,V)\,\|v(r)\cos(j\theta)\|_{H^2(B_O(R))}\]
\begin{center}
    and
\end{center} 
\[\left\|v-\Pi_h v\right\|_{j,r}\leq h\, C(R,V)\,\|v(r)\cos(j\theta)\|_{H^2(B_O(R))}.\]
\item \label{item:projection-D-modified} For every $v\in H_{j,r}([0,R])$ such that $v(r)\cos(j\theta)\in H^2(B_O(R))$, 
\[\left\|[v-v(R)\phi_{N_h}]-\Pi_h^0 [v-v(R)\phi_{N_h}]\right\|_{r}\leq h^2 C(R,V)\,\|v(r)\cos(j\theta)\|_{H^2(B_O(R))}\]
\begin{center}
    and
\end{center} 
\[\left\|[v-v(R)\phi_{N_h}]-\Pi_h^0 [v-v(R)\phi_{N_h}]\right\|_{j,r}\leq h\, C(R,V)\,\|v(r)\cos(j\theta)\|_{H^2(B_O(R))}.\]
\end{enumerate}
\end{lemma}
A proof of this lemma can be found in Appendix \ref{acaret:proofs-2}.\\

The next step in the proof of Theorem \ref{thm:ujh-uj-convergence} is to establish that the discrete Dirichlet eigenvalues below a fixed threshold $K>1$, associated with $a_{j,V_h}$ on $W_{j,r}^0$, are close to the continuous Dirichlet eigenvalues of $a_{j,V}$ on $H_{j,r}^0([0,R])$ as the Finite Element grid parameter $h$ approaches zero.

An analogous result holds for the Neumann case: the discrete Neumann eigenvalues associated with $a_{j,V_h}$ on $W_{j,r}$ are close to the continuous Neumann eigenvalues of $a_{j,V}$ on $H_{j,r}([0,R])$.

This fundamental observation enables the application of Propositions \ref{prop:gap-dirichlet-neumann} and \ref{prop:operator-outside-of-spectrum}, thereby allowing the proof of Theorem \ref{thm:ujh-uj-convergence} to be decomposed according to whether $\lambda$ lies away from the Dirichlet spectrum or from the Neumann spectrum of $a_{j,V}$.

\begin{proposition}\label{prop:approximation-of-spectra} Let $R>1$, $K>1$ and the potential $V\in C^1([0,R])$ taking values in $[1,\infty)$. Also consider, for every non-negative integer $j$, the Dirichlet eigenvalues of the bilinear form $a_{j,V}$ on $H_{j,r}^0([0,R])$ which are less than or equal to K (repeated according to their multiplicities):
\[1\leq \lambda^j_1\leq \lambda^j_2\leq \ldots \leq \lambda^j_{Q_{j,K}}\]

Then, there exists a threshold $h_0(R,V,K)>0$  such that, for every $h\in (0,h_0)$, the bilinear form $a_{j,V_h}$ on $W_{j,h}^0$ has at least $Q_{j,K}$ eigenvalues (counted with multiplicities). Moreover, there exists a constant $C(R,V,K)>0$ such that, for each $n\in \overline{1,Q_{j,K}}$,
\[|\lambda^j_{n,h}-\lambda^j_n|\leq h\,C(R,V,K),\]
where $\lambda^j_{n,h}$ is the $n$-th eigenvalue of $a_{j,V_h}$ on $W_{j,h}^0$.

If $V\in C^2([0,R])$, the above inequality can be improved to: 
\[|\lambda^j_{n,h}-\lambda^j_n|\leq h^2\,C(R,V,K).\]

Furthermore, the statement of this proposition is also valid for the Neumann framework outlined above.
\end{proposition}

The proof of this proposition is deferred to Appendix \ref{acaret:proofs-2}. We are now in a position to prove the error estimates between the basis functions $\uj^{\lambda}$ and their finite element approximations $\ujh^{\lambda}$.
\begin{proof}[Proof of Theorem \ref{thm:ujh-uj-convergence}]
First, Proposition \ref{prop:gap-dirichlet-neumann} implies that there exists a gap $\delta_{R,V,K,J}\in\left (0,\frac{1}{2}\right)$ between the Dirichlet and Neumann eigenvalues associated to $a_{j,V}$ that are smaller than $K+1$. 

Next, we choose $h_0(R,V,K,J)$ small enough such that, by Proposition \ref{prop:approximation-of-spectra}, the aforementioned Dirichlet and Neumann eigenvalues are approximated by their discrete counterparts (associated with $a_{j,V_h}$ on $W_{j,h}^0$ and $W_{j,h}$, respectively) with an error of at most $\frac{\delta_{R,V,K,J}}{4}$. 

Consequently, there is a gap of at least $\frac{\delta_{K,V,K,J}}{2}$ between the discrete Dirichlet and Neumann eigenvalues associated with $a_{j,V_h}$ that are smaller than $K+\frac{1}{2}$. As a result, any $\lambda\in [1,K]$ falls into at least one of the following cases:
\begin{enumerate}[label={\Roman*.}]
\setlength{\itemsep}{0.3cm}
\item \label{case:discrete-not-D} $\min \left\{|\lambda-\lambda^j_{n,h}| : \lambda^j_{n,h} \text{ is a Dirichlet eigenvalue of }\\ a_{j,V_h}\text { on } W_{j,h}^0 \right\}>\frac{\delta_{R,V,K,J}}{6}$;
\item \label{case:discrete-not-N} $\min \left\{|\lambda-\zeta^j_{m,h}| : \zeta^j_{m,h} \text{ is a Neumann eigenvalue of }\\ a_{j,V_h}\text { on } W_{j,h} \right\}>\frac{\delta_{R,V,K,J}}{6}$.
\end{enumerate}
In Case I, we note that $\ujh^{\lambda}(R)$ is different from zero. As a result, we are able to consider the error $e^{\lambda}_h$  between $\uj^{\lambda}$ and the multiple of $\ujh^{\lambda}$ which has the same value as $\uj^{\lambda}$ in the point $R$: 
\[e^{\lambda}_h\coloneqq \uj^{\lambda}-\frac{\uj^{\lambda}(R)}{\ujh^{\lambda}(R)}\ujh^{\lambda}\in H_{j,r}^0([0,R]).\]
We aim to prove that  $\|e_h^\lambda\|_{j,r}$ is small as $h$ approaches zero. Indeed, since $\uj^{\lambda}$ satisfies \eqref{EVP-j-W-short} and $\ujh^{\lambda}$ satisfies \eqref{EVP-j-W-h}, we obtain that:
\[a_{j,V_h}(e_h^{\lambda}, w^h) - \lambda (e_h^{\lambda}, w^h)_r = \left([V_h-V]\, \uj^{\lambda},w^h\right)_r, \quad \forall w^h\in W_{j,r}^0.\]
Next, the definition of the projection operator $\Pi_h^0$ implies that:
\[a_{j,V_h}(\Pi_h^0e_h^{\lambda}, w^h) - \lambda (\Pi_h^0e_h^{\lambda}, w^h)_r = \left([V_h-V]\, \uj^{\lambda}+\lambda\left[e_h^{\lambda}-\Pi_h^0e_h^{\lambda}\right],w^h\right)_r, \quad \forall w^h\in W_{j,r}^0.\]
Then, since $\|\uj^{\lambda}\|_r=1$, Proposition \ref{prop:operator-outside-of-spectrum} implies that:
\begin{equation}\label{eq:discrete-not-D-projection-bound}
\|\Pi_h^0 e_h^{\lambda}\|_{j,r}\leq \left(1+\frac{6K}{\delta_{R,V,K,J}}\right) \left(\|V_h-V\|_{L^\infty([0,R])}+K\left\|e_h^{\lambda}-\Pi_h^0e_h^{\lambda}\right\|_r\right).\end{equation}
Furthermore, the definition of the projection operator $\Pi_h^0$ implies that:
\[\begin{aligned}e_h^{\lambda}-\Pi_h^0e_h^{\lambda} &= [\uj^{\lambda} -\uj^{\lambda}(R)\phi_{N_h}]-\left[\frac{\uj^{\lambda}(R)}{\ujh^{\lambda}(R)}\ujh^{\lambda}-\uj^{\lambda}(R)\phi_{N_h}\right]\\
&\quad-\Pi_h^0[\uj^{\lambda}-\uj^{\lambda}(R)\phi_{N_h}] + \Pi_h^0\left[\frac{\uj^{\lambda}(R)}{\ujh^{\lambda}(R)}\ujh^{\lambda}-\uj^{\lambda}(R)\phi_{N_h}\right]
\end{aligned}\]
and, since $\frac{\uj^{\lambda}(R)}{\ujh^{\lambda}(R)}\ujh^{\lambda}-\uj^{\lambda}(R)\phi_{N_h}\in W_{j,h}^0$, we arrive at:
\[e_h^{\lambda}-\Pi_h^0e_h^{\lambda} = \left[\uj^{\lambda} -\uj^{\lambda}(R)\phi_{N_h}\right]-\Pi_h^0\left[\uj^{\lambda} -\uj^{\lambda}(R)\phi_{N_h}\right].\]
Therefore, Propositions \ref{lem:projection-error} and \ref{prop:H2-stability-J} imply that:
\[\left\|e_h^{\lambda}-\Pi_h^0e_h^{\lambda}\right\|_r\leq h^2\, C(R,V) \|\uj^{\lambda}(r)\cos(j\theta)\|_{H^2(B_O(R))}\leq h^2\, C(R,V,K,J)\|\uj^{\lambda}\|_r=h^2\, C(R,V,K,J)\]
and, similarly, 
\[\left\|e_h^{\lambda}-\Pi_h^0e_h^{\lambda}\right\|_{j,r}\leq h\, C(R,V,K,J).\]
Inserting the last two inequalities into \eqref{eq:discrete-not-D-projection-bound}, we obtain the following estimates:
\[\begin{aligned}\|e_h^{\lambda}\|_r &\leq \left(h^2+\|V_h-V\|_{L^\infty([0,R])}\right) C(R,V,K,J);\\
\|e_h^{\lambda}\|_{j,r}&\leq \left(h\,\,+\|V_h-V\|_{L^\infty([0,R])}\right)C(R,V,K,J).
\end{aligned}
\]

The conclusion of Case I follows by \eqref{eq:error-V-Vh-C1}-\eqref{eq:error-V-Vh-C2}, using a reasoning similar to the final part of the proof of Proposition \ref{prop:uj-stable-in-lambda} (i.e., inequalities \eqref{eq:dirichlet-ratio-near-1}-\eqref{eq:lambda-stability-final-inequality}).

In Case II, we claim that there exists a constant $\beta\in \RR\setminus\{0\}$, such that, for every $w^h\in W_{j,h}$,
\begin{equation}
\label{eq:discrete-normal-derivative}
a_{j,V_h}(\ujh^{\lambda},w^h)-\lambda(\ujh^{\lambda},w^h)_r=\beta\,w^h(R)R.\end{equation}
Indeed, for $w^h\in W_{j,h}^0$, the equality above is true because $\ujh^{\lambda}$ satisfies \eqref{EVP-j-W-h}. Moreover, if we denote:
\[\beta\coloneqq \frac{1}{R}\left( a_{j,V_h}(\ujh^{\lambda},\phi_{N_h})-\lambda(\ujh^{\lambda},\phi_{N_h})_r\right),\]
the equality \eqref{eq:discrete-normal-derivative} follows since $W_{j,h}={\rm span}\left\{\phi_{N_h},W_{j,h}^0\right\}$. Furthermore, the constant $\beta$ is not null, since, in this case, $\lambda$ is not a Neumann eigenvalue of $a_{j,V_h}$ on $W_{j,h}$.

On the other hand, since $\uj^{\lambda}$ satisfies \eqref{EVP-j-W-short}, an integration by parts in the spirit of Proposition \ref{prop:neumann-as-ode} (equation \eqref{eq:integration-parts-neumann}) implies that:
\[a_{j,V}(\uj^{\lambda},w)-\lambda(\uj^{\lambda},w)_r=(\uj^{\lambda})'(R)w(R)R, \quad \forall w\in H_{j,r}([0,R]).\]
Therefore, if we denote the error term in this case by:
\[e_h^{\lambda}\coloneqq \uj^{\lambda}-\frac{(\uj^{\lambda})'(R)}{\beta}\ujh^{\lambda
},\]
we arrive at:
\[
a_{j,V_h}(e_h^{\lambda},w^h)-\lambda(e_h^{\lambda},w^h)_r=0, \quad\forall w^h\in W_{j,h}.
\]
Using Proposition \ref{prop:operator-outside-of-spectrum} and Lemma \ref{lem:projection-error} \ref{item:projection-N}, we obtain the conclusion similarly to the proof of Case I.
\end{proof}

 Theorem \ref{thm:ujh-uj-convergence} implies error bounds between the function $u^{J,\lambda}_\aaa$ defined in \eqref{eq:def-UJLambdaAlpha} and its discrete counterpart $u^{J,\lambda}_{\aaa,h}$, which can be actually obtained using the computer. This discrete counterpart is defined as:
 \begin{equation}\label{eq:defUJh}
 u^{J,\lambda}_{\boldsymbol{\alpha},h} (\x)\coloneqq\alpha_0^c \uzh^{\lambda}(r)+\sum_{j=1}^J \ujh^{\lambda}(r) \left[\alpha_j^c \cos(j\theta)+\alpha_j^s\sin(j\theta)\right], \quad \x=(r\cos(\theta),r\sin(\theta))\in B_O(R),
    \end{equation}
    where $\aaa\coloneqq(\alpha_0^c,\alpha_1^c,\ldots, \alpha_J^c,\alpha_1^s,\ldots,\alpha_J^s)\in \RR^{2J+1}$
\begin{corollary}\label{cor:error-UJ-UJh}
In the context of Theorem \ref{thm:ujh-uj-convergence}, for every $\tilde R\in (0,R]$ there exists a constant $C(R,\tilde R,V,K,J)>0$ such that for every $\aaa\in \RR^{2J+1}$,
\[\|u^{J,\lambda}_{\boldsymbol{\alpha},h}-u^{J,\lambda}_{\boldsymbol{\alpha}}\|_{H^1(B_O(R))}\leq h\, C(R,\tilde R,V,K,J)\|u^{J,\lambda}_{\boldsymbol{\alpha}}\|_{L^2(B_O(\tilde R))}\]
and there exist continuous representatives of $u^{J,\lambda}_{\boldsymbol{\alpha},h}$ and $u^{J,\lambda}_{\boldsymbol{\alpha}}$ on $B_O(R)\setminus B_O(\tilde R)$ which satisfy:
\[\begin{aligned}\|u^{J,\lambda}_{\boldsymbol{\alpha},h}-u^{J,\lambda}_{\boldsymbol{\alpha}}\|_{C\left(B_O(R)\setminus B_O(\tilde R)\right)}&\leq h\, C(R,\tilde R,V,K,J)\|u^{J,\lambda}_{\boldsymbol{\alpha}}\|_{L^2(B_O(\tilde R))};\\
\|u^{J,\lambda}_{\boldsymbol{\alpha},h}\|_{C\left(B_O(R)\setminus B_O(\tilde R)\right)}, \|u^{J,\lambda}_{\boldsymbol{\alpha}}\|_{C\left(B_O(R)\setminus B_O(\tilde R)\right)} &\leq C(R,\tilde R,V,K,J)\|u^{J,\lambda}_{\boldsymbol{\alpha}}\|_{L^2(B_O(\tilde R))}.
\end{aligned}\]
Moreover, if $V\in C^2([0,R])$, then we get a second-order $L^2(B_O(R))$ error estimate as $h$ approaches zero:
\[\|u^{J,\lambda}_{\boldsymbol{\alpha},h}-u^{J,\lambda}_{\boldsymbol{\alpha}}\|_{L^2(B_O(R))}\leq h^2\, C(R,\tilde R,V,K,J)\|u^{J,\lambda}_{\boldsymbol{\alpha}}\|_{L^2(B_O(\tilde R))}.\]
\end{corollary}
\begin{proof} First, we recall that every function $w\in H_{j,r}([0,R])$ has a continuous representative in $C([\tilde R,R])$. Moreover, there exists a constant $C(R,\tilde R)>0$ such that
\[\|w\|_{C([\tilde R,R])}\leq C(R,\tilde R) \|w\|_{j,r}.\]
Therefore, Theorem \ref{thm:ujh-uj-convergence} implies that:
\[\|\ujh^{\lambda}-\uj^{\lambda}\|_{C([\tilde R,R])}\leq h \, C(R,\tilde R,V,K,J).\]
Then, taking into account that $\|\uj^{\lambda}\|_r=1$, the conclusion of Corollary \ref{cor:error-UJ-UJh} follows from Theorem \ref{thm:ujh-uj-convergence} and Proposition \ref{prop:H2-stability-J}, via a reasoning involving the inequality between the arithmetic mean and the quadratic mean, similar to \eqref{eq:UJ-bounded-L2-1}-\eqref{eq:UJ-mean-inequality}.
\end{proof}

\section{Numerical approximation of the quotient $\mathcal{F}^\lambda(\aaa)$}
\label{sec:numerical-integration}
In the previous section, we have successfully estimated the function $ u^{J,\lambda}_{\boldsymbol{\alpha}}$ using a one-dimensional Finite Element scheme for every basis function $\uj^{\lambda}$. Therefore, in order to approximately compute the quotient $\mathcal{F}^{\lambda}(\boldsymbol{\alpha})$ in \eqref{eq:def-F}, we need to integrate numerically both its numerator and denominator and prove the order of accuracy of this process.

\subsection{Numerical approximations for $L^2$ norms on the boundary of the domain $\Omega$}
\label{sec:numerical-integration-boundary}
In order to estimate the numerator $\|u^{J,\lambda}_{\boldsymbol{\alpha}}\|_{L^2(\partial\Omega)}$ of $\mathcal{F}^\lambda(\aaa)$, we consider $L>0$ and $\gamma:[0,L]\to \partial \Omega$ a $C^2$ non-degenerate (i.e. $\gamma'\neq (0,0)$ in every point) periodic parametrisation curve of $\partial \Omega$. We also consider a large integer $N_{\partial\Omega}$ and a division of $[0,L]$:
\[0\leq t_1<t_2< \ldots < t_{N_{\partial\Omega}}< L,\]
that satisfies:
\begin{equation}\label{boundary-division-width}\frac{L}{2 N_{\partial \Omega}}<t_i-t_{i-1}<\frac{3L}{2 N_{\Omega}},\quad  \forall i\in \overline{1,N_{\partial\Omega}},\end{equation}
with the convention that $t_{0}=t_{N_{\partial\Omega}}-L$.
The points $(t_i)_{i=1}^{N_{\partial\Omega}}$ can be constructed by randomly displacing every node of the uniform $N_{\partial\Omega}$-sized grid on $[0,L]$ by a tiny amount. 

\emph{Note:} One situation in which we might prefer such a randomly displaced grid over the uniform one is when the parametrisation $\gamma$ of $\partial\Omega$ depends on the angular value $\theta$. In this case, if the nodes $t_i$ are equally spaced in the interval $[0,2\pi]$, the points $\gamma(t_i)$ may interfere with the zeros of the angular basis functions $\cos(j\theta)$ and $\sin(j\theta)$, which  are periodic in $\theta$.\\

With these considerations, we define the approximation of the $L^2(\partial\Omega)$ norm of $u^{J,\lambda}_{\boldsymbol{\alpha}}$:
\[S_{N_{\partial\Omega}}(u^{J,\lambda}_{\boldsymbol{\alpha},h})\coloneqq \sqrt{\sum_{i=1}^{N_{\partial\Omega}} |u^{J,\lambda}_{\boldsymbol{\alpha},h}(\gamma(t_i))|^2\, \|\gamma'(t_i)\|(t_i-t_{i-1})}.\]
 The convergence of this quadrature rule is outlined in the following proposition, whose proof can be found in Appendix \ref{acaret:proofs-3}.

\begin{proposition}\label{prop:boundary-int}
Let $\Omega$ be a simply connected $C^2$ bounded domain and $R,\tilde R>0$ such that $B_O(\tilde R)\subseteq\Omega\subseteq B_O(R)$. Let also the potential $V\in C^1([0,R])$ taking values in $[1,\infty)$. For every $K>1$ and positive integer $J$, there exists a threshold $h_0=h_0(R,\tilde R,V,K,J)>0$ and a constant $C(\Omega,V,K,J)>0$ such that, for every $\lambda\in [1,K]$, $h\in(0,h_0)$ and $\boldsymbol{\alpha}\in \RR^{2J+1}$, 
\[\left|S_{N_{\partial\Omega}}(u^{J,\lambda}_{\boldsymbol{\alpha},h})-\| u^{J,\lambda}_{\boldsymbol{\alpha}}\|_{L^2(\partial\Omega)}\right|\leq C(\Omega,V,K,J)\left[\frac{1}{N_{\partial\Omega}}+h\right]\|u^{J,\lambda}_{\boldsymbol{\alpha}}\|_{L^2(\Omega)}.\]
\end{proposition}

\subsection{Monte Carlo integration for $L^2$ norms on $\Omega$}
\label{sec:Monte-Carlo}
Now we approximate the denominator $\|u^{J,\lambda}_{\boldsymbol{\alpha}}\|_{L^2(\Omega)}$ of the quotient $\mathcal{F}^{\lambda}(\boldsymbol{\alpha})$. Inspired by \cite[Section 5]{BetckeTrefethen2005}, we use a Monte-Carlo method and consider the random points $\x_1,\x_2,\ldots, \x_{N_\Omega}$ which are independent and uniformly distributed in $\Omega$. The Monte-Carlo approximation of $\|u^{J,\lambda}_{\boldsymbol{\alpha}}\|_{L^2(\Omega)}$ reads:
\[\mathcal{S}_{N_\Omega}(u^{J,\lambda}_{\boldsymbol{\alpha},h})\coloneqq \sqrt{\frac{|\Omega|}{N_{\Omega}} \sum_{i=1}^{N_\Omega} |u^{J,\lambda}_{\boldsymbol{\alpha},h}(\x_i)|^2},\]
where by $|\Omega|$ we denote the area of the domain $\Omega$.
The following proposition (see Appendix \ref{acaret:proofs-3} for the proof) quantifies the error of this probabilistic approximation, uniformly with respect to $\aaa$. 
\begin{proposition}\label{prop:Monte-Carlo}
Let the parameters $\Omega, R,\tilde R, K$ be as in Proposition \ref{prop:boundary-int}.
There exists a threshold $h_0=h_0(R,\tilde R,V,K,J)>0$ and a constant $C(\Omega,V,K,J)>0$ such that, for every fixed $\lambda \in [1,K]$, $h\in (0,h_0)$ and $\eta>0$, the probability of the event:
\[\mathcal{E}_{N_\Omega,\eta}\coloneqq \text{" For every }\boldsymbol{\alpha}\in \RR^{2J+1}, \left|\mathcal{S}_{N_\Omega} ( u^{J,\lambda}_{\boldsymbol{\alpha},h}) - \|u^{J,\lambda}_{\boldsymbol{\alpha}}\|_{L^2(\Omega)}\right|<(\eta+h) C(\Omega,V,K,J)\, \| u^{J,\lambda}_{\boldsymbol{\alpha}}\|_{L^2(\Omega)}\text{"}\]
is bounded from below as follows:
\[\mathbb{P}(\mathcal{E}_{N_\Omega,\eta})\geq 1-\frac{C(\Omega,V,K,J)}{\eta^2 \,N_\Omega}.\]
Note: if the potential $V$ belongs to $C^2([0,R])$, then the same conclusion can be also obtained for the event:
\[\mathcal{E}_{N_\Omega,\eta}\coloneqq \text{" For every }\boldsymbol{\alpha}\in \RR^{2J+1}, \left|\mathcal{S}_{N_\Omega} (u^{J,\lambda}_{\boldsymbol{\alpha},h}) - \|u^{J,\lambda}_{\boldsymbol{\alpha}}\|_{L^2(\Omega)}\right|<(\eta+h^2) C(\Omega,V,K,J)\, \| u^{J,\lambda}_{\boldsymbol{\alpha}}\|_{L^2(\Omega)}\text{"}.\]
\end{proposition}

\subsection{The quotient $\mathcal{F}^\lambda$ approximated by the ratio of matrix norms}
\label{sec:quotient-approximation}

The numerical integration procedures developed in Sections \ref{sec:numerical-integration-boundary} and \ref{sec:Monte-Carlo} allow us to write the approximation of the quotient $\mathcal{F}^\lambda(\aaa)$ as follows:
\begin{equation}\label{def:Fh}
\mathcal{F}^\lambda_h(\aaa)\coloneqq\frac{S_{N_{\partial \Omega}}(u^{J,\lambda}_{\aaa,h})}{\mathcal{S}_{N_\Omega}(u^{J,\lambda}_{\aaa,h})}=\frac{\|\mathcal{M}_{\partial\Omega}\aaa\|}{\|\mathcal{M}_{\Omega}\aaa\|},\end{equation}
where the $N_{\partial\Omega}\times (2J+1)$ matrix $\mathcal{M}_{\partial\Omega}$ is defined by:
\[\begin{aligned}(
\mathcal{M}_{\partial\Omega})_{i,1}&\coloneqq \uzh^{\lambda}(\|\gamma(t_i)\|)\sqrt{\|\gamma'(t_i)\|(t_i-t_{i-1})}, \quad i\in \overline{1,N_{\partial\Omega}};\\
(\mathcal{M}_{\partial\Omega})_{i,2j}&\coloneqq \ujh^{\lambda}(\|\gamma(t_i)\|)\cos[j\,{\rm arg}(\gamma(t_i))]\sqrt{\|\gamma'(t_i)\|(t_i-t_{i-1})}, \quad i\in \overline{1,N_{\partial\Omega}},\,  j\in \overline{1,J};\\
(\mathcal{M}_{\partial\Omega})_{i,2j+1}&\coloneqq \ujh^{\lambda}(\|\gamma(t_i)\|)\sin[j\,{\rm arg}(\gamma(t_i))]\sqrt{\|\gamma'(t_i)\|(t_i-t_{i-1})}, \quad i\in \overline{1,N_{\partial\Omega}},\,  j\in \overline{1,J},
\end{aligned}\]
where by ${\rm arg}(\x)$ we denote the argument $\theta$ of the non-null vector $\x=(r\cos(\theta),r\sin(\theta))\in \RR^2$.

Similarly, the the $N_{\Omega}\times (2J+1)$ matrix $\mathcal{M}_{\Omega}$ is defined by:
\[\begin{aligned}(
\mathcal{M}_{\Omega})_{i,1}&\coloneqq \sqrt{\frac{|\Omega|}{N_{\Omega}}}\uzh^{\lambda}(\boldsymbol{r_i}), \quad i\in \overline{1,N_{\Omega}};\\
(\mathcal{M}_{\Omega})_{i,2j}&\coloneqq \sqrt{\frac{|\Omega|}{N_{\Omega}}} \ujh^{\lambda}(\boldsymbol{r}_i)\cos(j\boldsymbol\theta_i), \quad i\in \overline{1,N_{\Omega}},\,  j\in \overline{1,J};\\
(\mathcal{M}_{\Omega})_{i,2j+1}&\coloneqq \sqrt{\frac{|\Omega|}{N_{\Omega}}}\ujh^{\lambda}(\boldsymbol{r}_i)\sin(j\boldsymbol\theta_i), \quad i\in \overline{1,N_{\Omega}},\,  j\in \overline{1,J}.
\end{aligned}\]
For simplicity, we will consider that the number of collocation points $N_{\partial\Omega}$ and $N_\Omega$ depends on the discretisation parameter $h$ in order to obtain the following theorem that quantifies the error between the quotients $\mathcal{F}^\lambda_h(\aaa)$ and $\mathcal{F}^\lambda(\aaa)$. 
\begin{theorem}\label{thm:approx-ratio}
Let $\Omega$ be a simply connected $C^2$ bounded domain containing the origin, $K>1$, the potential $V\in C^1([0,R])$ taking values in $[1,\infty)$, and a fixed probability $p\in (0,1)$. Then, for each positive integer $J$, there exist a threshold $h_0=h_0(\Omega,V,K,J)>0$ and a constant $C(\Omega,V,K,J)>0$ such that, for every $h\in (0,h_0)$ there exist some positive integers $N_{\partial\Omega,0}=N_{\partial\Omega,0}(h)$ and $N_{\Omega,0}=N_{\Omega,0}(p,\Omega,V,K,J,h)$ such that, for every $\lambda\in [1,K]$ and every $N_{\partial\Omega}\geq N_{\partial\Omega,0}$, $N_\Omega\geq N_{\Omega,0}$, the probabilistic event
\[\begin{aligned}\mathcal{E}_h\coloneqq \text{" For every }\boldsymbol{\alpha}\in \RR^{2J+1}, &\,\, \mathcal{F}^\lambda_h(\aaa)\leq \frac{\mathcal{F}^\lambda(\aaa)}{1-h\,C(\Omega,V,K,J)} + h\, C(\Omega,V,K,J)\\
\text {and }  &\,\, \mathcal{F}^\lambda(\alpha)\leq [1+h\, C(\Omega,V,K,J)]\mathcal{F}^\lambda_h(\aaa) + h\, C(\Omega,V,K,J)\text{ "}.
\end{aligned}\]
has probability greater than $1-p$. Furthermore, if the potential $V$ belongs to $C^2([0,R])$, then we can replace the event $\mathcal{E}_h$ with:
\[\begin{aligned}\mathcal{E}_h\coloneqq \text{" For every }\boldsymbol{\alpha}\in \RR^{2J+1}, &\,\, \mathcal{F}^\lambda_h(\aaa)\leq \frac{\mathcal{F}^\lambda(\aaa)}{1-h^2\,C(\Omega,V,K,J)} + h\, C(\Omega,V,K,J)\\
\text {and }  &\,\, \mathcal{F}^\lambda(\alpha)\leq [1+h^2\, C(\Omega,V,K,J)]\mathcal{F}^\lambda_h(\aaa) + h\, C(\Omega,V,K,J)\text{ "}.
\end{aligned}\]
\end{theorem}
\begin{proof}
In Proposition \ref{prop:boundary-int} we take $N_{\partial\Omega}>\frac{1}{h}$ and in Proposition \ref{prop:Monte-Carlo} we take $\eta\coloneqq h^2$ or $\eta\coloneqq h$, depending on whether $V$ belongs to $C^2([0,R])$ or only to $C^1([0,R])$. Then, we choose $N_\Omega$ such that the term   $\frac{C(\Omega,R,\tilde R,V,K,J)}{\eta^2 \,N_\Omega}$ is smaller than $p$. The conclusion follows from Propositions \ref{prop:boundary-int} and \ref{prop:Monte-Carlo} via simple arithmetic manipulations.
\end{proof}

\section{Minimisation of the quotient $\mathcal{F}^\lambda_h$ and numerical simulations}
\label{sec:numerical}
This section is devoted to the practical computation of approximations to the eigenpairs $(\lambda^*,u^*)$ of \eqref{EVP-radial}, based on the algorithm described in Section \ref{sec:description-of-method}. The approach, inspired by \cite{BetckeTrefethen2005}, consists in identifying values of $\lambda$ for which there exists a vector $\aaa \in \RR^{2J+1}$ such that the quotient $\mathcal{F}^\lambda_h(\aaa)$, defined in \eqref{def:Fh}, is sufficiently small.
Whenever such a vector is found, Theorems \ref{thm:main} and \ref{thm:approx-ratio} ensure that $\lambda$ is close to an eigenvalue $\lambda^*$, and that the corresponding function $u^{J,\lambda}_{\aaa,h}$ provides an approximation of an associated eigenfunction $u^*$.

\subsection{Approximating eigenpairs via minimisation}\label{sec:minimisation-QR}
We briefly describe the technique employed in \cite{BetckeTrefethen2005} to minimise the quotient $\mathcal{F}^\lambda_h(\aaa)=\frac{\|\mathcal{M}_{\partial\Omega}\aaa\|}{\|\mathcal{M}_{\Omega}\aaa\|}$ with respect to $\aaa\in \RR^{2J+1}\setminus\{0\}$. First, a QR decomposition is performed for the matrix:
$$\left[\begin{array}{c}\mathcal{M}_{\partial\Omega}\\ \mathcal{M}_{\Omega}\end{array}\right]=\left[\begin{array}{c}\mathcal{Q}_{\partial\Omega}\\ \mathcal{Q}_{\Omega}\end{array}\right]\mathcal{R},$$
where the  dimensions of the matrices $\mathcal{Q}_{\partial\Omega}$, $\mathcal{Q}_{\Omega}$ and $\mathcal{R}$ are $N_{\partial\Omega}\times (2J+1)$, $N_{\Omega}\times (2J+1)$ and $(2J+1)\times (2J+1)$, respectively. We note that $\mathcal{R}$ is a quadratic matrix because we always choose the numerical parameters such that $N_{\partial\Omega},N_{\Omega}>2J+1$. Moreover, Proposition \ref{prop:Monte-Carlo} implies that $\mathcal{R}$ is invertible, because the matrix $\mathcal{M}_\Omega$ has full rank.

The next step consists of a Singular Value Decomposition (SVD) performed on the matrix $\mathcal{Q}_{\partial\Omega}$, which provides the smallest singular value $\sigma^\lambda_1$, together with the associated singular vector $\boldsymbol \beta_1^\lambda$. Following the reasoning in \cite[Section 8]{BetckeTrefethen2005}, $\mathcal{F}^\lambda_h$ attains its minimum for the vector: \[\aaa_{min}^\lambda\coloneqq\mathcal{R}^{-1}\boldsymbol{\beta}_1^\lambda\]
and the minimal value is equal to:
\[\mathcal{F}^\lambda_h(\aaa_{min}^\lambda)=\frac{\sigma_1^\lambda}{\sqrt{1-(\sigma_1^\lambda)^2}}.\]

If this quantity is sufficiently small (equivalently, if $\sigma_1^\lambda$ is below a prescribed tolerance $\eps$) then $\lambda$ is retained as an approximate eigenvalue, and the function $u^{J,\lambda}_{\aaa_{min}^\lambda,h}$ is taken as an approximation of a corresponding eigenfunction. We refer to Section \ref{sec:numerical-eigenvalues} for a numerical illustration of this procedure on representative examples.

\subsection{Approximating eigenspaces of dimension higher that one}\label{sec:multiple-eigenfunctions}
The algorithm is also capable of identifying all linearly independent eigenfunctions associated with a given eigenvalue. More precisely, if the first $n$ (smallest) singular values $\sigma_1^\lambda, \sigma_2^\lambda, \ldots, \sigma_n^\lambda$ of the matrix $\mathcal{Q}_{\partial\Omega}$ are below a prescribed threshold $\varepsilon$, then denoting by $\boldsymbol{\beta}_k^\lambda$ a right-singular vector of $\mathcal{Q}_{\partial\Omega}$ corresponding to $\sigma_k^\lambda$ and 
\[\aaa_{k}^\lambda\coloneqq\mathcal{R}^{-1}\boldsymbol{\beta}_k^\lambda, \] 
the functions $u^{J,\lambda}_{\aaa_k^\lambda,h}$
provide an approximation of a basis of the eigenspace of~$\eqref{EVP-radial}$ corresponding to an eigenvalue $\lambda^*$ close to $\lambda$. Before presenting numerical examples illustrating the approximation of multiple eigenfunctions in Section~\ref{sec:numerical-eigenfunctions}, we explain below why this procedure is justified.

For simplicity, we consider only two singular values $\sigma_1^\lambda$ and $\sigma_2^\lambda$ which are smaller than $\eps$. Theorems \ref{thm:main} and \ref{thm:approx-ratio} imply that both functions $u^{J,\lambda}_{\aaa_k^\lambda,h}$, $k=\overline{1,2}$ are close to some eigenfunctions $u^*_1$ and $u^*_2$ of \eqref{EVP-j-W}, which will prove that are linearly independent. Indeed, since the right singular vectors $\boldsymbol\beta_1^\lambda$ and $\boldsymbol\beta_2^\lambda$ are orthonormal, then for any real numbers $\omega_1$ and $\omega_2$, 
\begin{equation}\label{eq:beta-orthogonals}\|\omega_1\boldsymbol{\beta}_1^\lambda+\omega_2 \boldsymbol{\beta}_2^\lambda\|^2=\omega_1^2+\omega_2^2.\end{equation}
As a result, Proposition \ref{prop:Monte-Carlo} implies that:
\begin{equation}\label{eq:approx-linear-combination}
\|\omega_1\,u^*_1+\omega_2\, u^*_2\|^2\simeq
\left\|\omega_1\, u^{J,\lambda}_{\aaa_1^\lambda,h}+\omega_2\, u^{J,\lambda}_{\aaa_2^\lambda,h}\right\|_{L^2(\Omega)}^2\simeq \left\|\mathcal{M}_\Omega[\omega_1\aaa_1^\lambda+\omega_2\aaa_2^\lambda]\right\|^2=\left\|\mathcal{Q}_\Omega[\omega_1\boldsymbol\beta_1^\lambda+\omega_2\boldsymbol\beta_2^\lambda]\right\|^2.
\end{equation}
The matrix $\left[\begin{array}{c}\mathcal{Q}_{\partial\Omega}\\ \mathcal{Q}_{\Omega}\end{array}\right]$ is orthonormal, so \eqref{eq:beta-orthogonals} and \eqref{eq:approx-linear-combination} lead to:
\[\|\omega_1\,u^*_1+\omega_2\, u^*_2\|^2\simeq  \omega_1^2+\omega_2^2-\left\|\mathcal{Q}_{\partial\Omega}[\omega_1\boldsymbol\beta_1^\lambda+\omega_2\boldsymbol\beta_2^\lambda]\right\|^2\geq (1-2\eps^2)(\omega_1^2+\omega_2^2),\]
where the last inequality follows because $\boldsymbol{\beta}_1^\lambda$ and $\boldsymbol{\beta}_2^\lambda$ are singular vectors of $\mathcal{Q}_{\partial\Omega}$ corresponding to singular values less than $\eps$. Therefore, the eigenfunctions $u^*_1$ and $u^*_2$ are linearly independent; in fact, \eqref{eq:approx-linear-combination} implies that they are almost orthonormal.

\subsection{Numerical approximation of eigenvalues}\label{sec:numerical-eigenvalues}
In this section, we consider two examples of domain $\Omega$ and potential $V$. The first one is an ellipsis with an analytic potential:
\begin{equation}\label{ex:ellipsis}
    \Omega_E\coloneqq\left\{(x,y)\in \RR^2: \frac{x^2}{4}+y^2<1 \right\}\subset B_O(2.1);\quad\quad V_E(r)\coloneqq\frac{2}{r^2+1}+1.
\end{equation}
and the second one is a star-shaped domain with a potential which is only $C^1$:
\begin{equation}\label{ex:plasture}
\begin{aligned}
\Omega_S\coloneqq&\left\{(r\cos\theta, r\sin\theta):\theta\in [0,2\pi), 0<r<3+\frac{\cos(4\theta)}{2}\right\}\subset B_O(3.6); \\ V_S(r)\coloneqq&\begin{cases}
1+r, & r\in [0,1];\\
1+r+(r-1)^2,& r>1.
\end{cases}
\end{aligned}
\end{equation}

Figure \ref{fig:bigGraph} contains, for each of the two examples, the graph of the minimum quotient value $\mathcal{F}^\lambda_h(\aaa_{min}^\lambda)$ obtained in Section \ref{sec:minimisation-QR}, as a function of $\lambda$, choosing the numerical parameters as follows:
\[J=450;\quad N_h=1000;\quad N_{\partial\Omega}=2000;\quad  N_{\Omega}=1000; \quad \mu=2\times10^{-2} \text{ (discretisation step for }\lambda\text{)}.\]
In order to reduce the instability far away from the eigenvalues, we employed a regularisation technique inspired from \cite[Section 4]{Betcke2008}. More precisely, we performed the minimisation of $\mathcal{F}^\lambda_h$ only for vectors $\aaa$ restricted to the orthogonal complement of the subspace generated by the right-singular vectors of the matrix $\mathcal{R}$ corresponding to singular values smaller than $10^{-8}$. We note that, as in \cite{Betcke2008}, the positions of the local minima in Figure \ref{fig:bigGraph} do not change significantly compared to the situation when the graph is generated without regularisation.

\begin{figure}[htbp]
  \centering
  \begin{subfigure}[t]{0.48\textwidth}
    \centering
    \includegraphics[width=\linewidth]{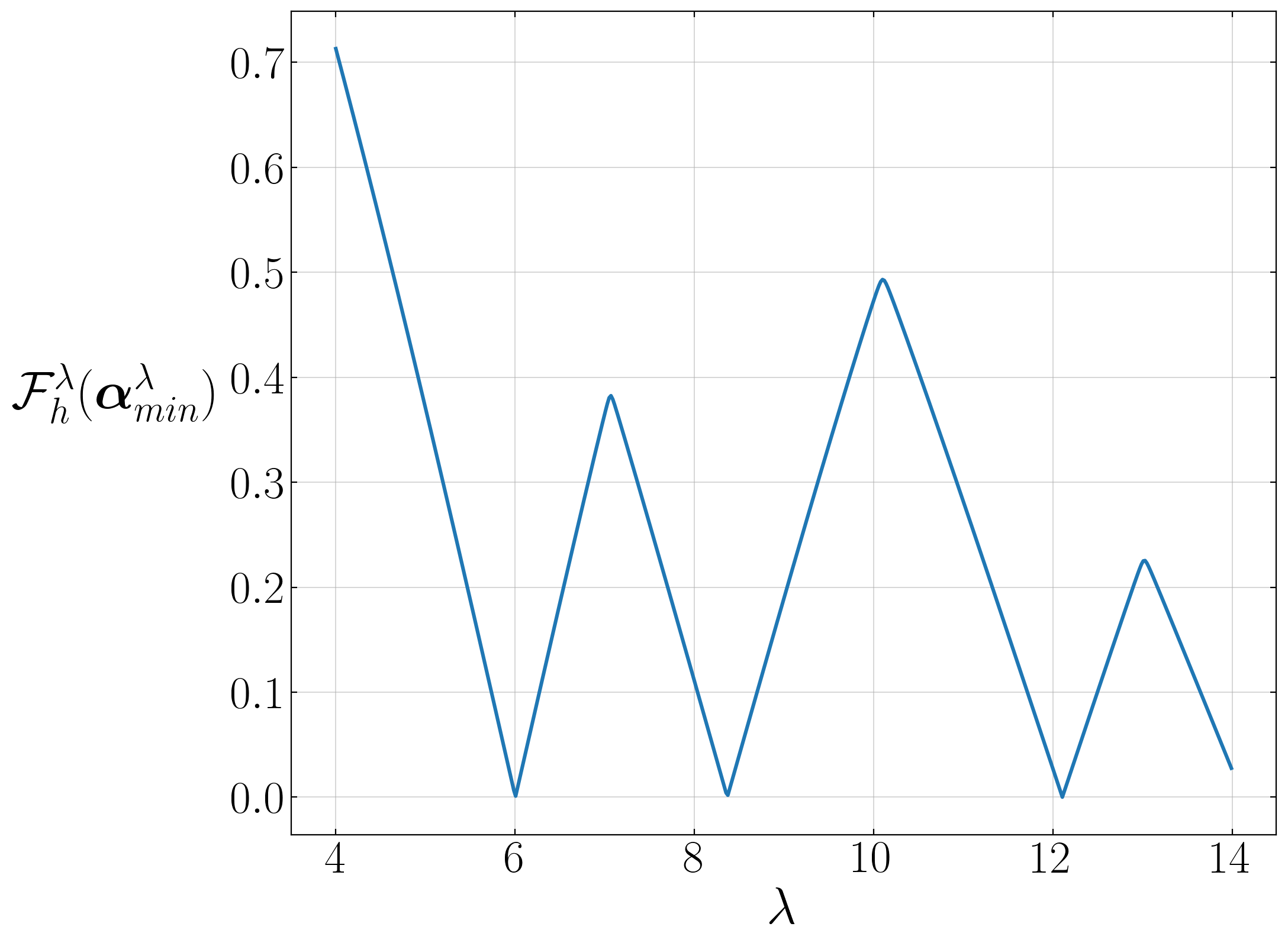}
    \caption{}
    \label{fig:EbigGraph}
  \end{subfigure}\hfill
  \begin{subfigure}[t]{0.48\textwidth}
    \centering
    \includegraphics[width=\linewidth]{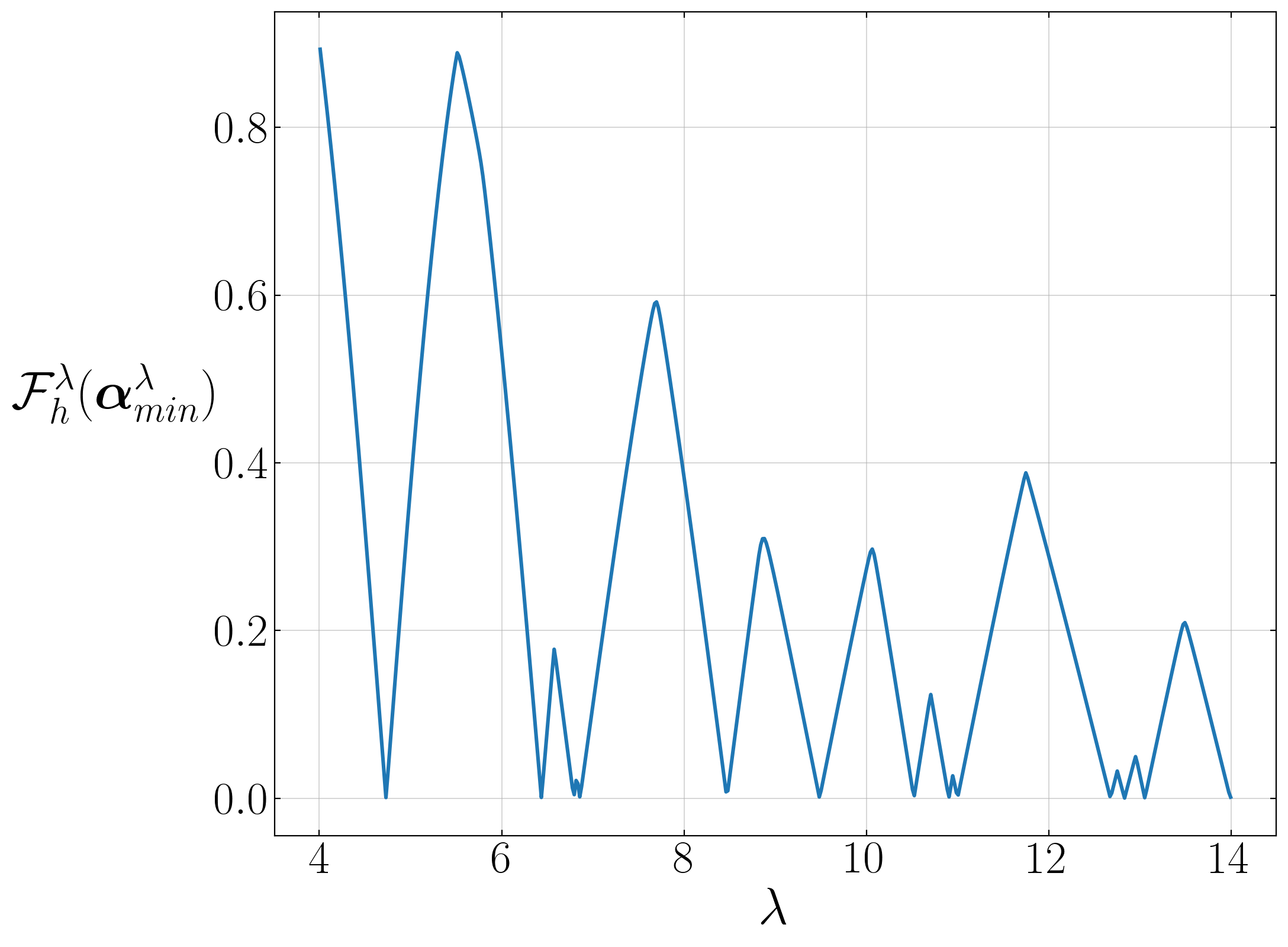}
\caption{}
    \label{fig:PlastureC1bigGraph}
  \end{subfigure}
  \caption{The function $\mathcal{F}^\lambda_h(\alpha_{min}^\lambda)$ in terms of the test value $\lambda$, for the ellipse $\Omega_E$ (A) and the star-shaped domain $\Omega_S$ (B).}
  \label{fig:bigGraph}
\end{figure}

Starting from those local minima in Figure~\ref{fig:bigGraph}, we enhance the accuracy of the eigenvalue approximation by progressively increasing the parameters $J$, $N_h$, $N_{\partial\Omega}$, and $N_{\Omega}$, while simultaneously refining the discretization of $\lambda$ (i.e., decreasing the parameter $\mu$). This iterative procedure is used to identify a new local minimum of the resulting objective function. For instance, Figure \ref{fig:smallGraph} depicts the plot of $\mathcal{F}^\lambda_h(\alpha_{min}^\lambda)$ near a local minimum for the parameters:
\begin{equation*}
J=1200;\quad N_h=3000;\quad N_{\partial\Omega}=6000;\quad  N_{\Omega}=5000; \quad \mu=5\times10^{-7}.
\end{equation*}
At this level of resolution, the minimisation is carried out without the aforementioned regularisation, as its influence is negligible in a neighbourhood of the eigenvalues; see \cite[Section 4]{Betcke2008}. In addition to simplifying the numerical computations, omitting the regularisation prevents the loss of solutions that may lie below the regularisation threshold.

\begin{figure}[htbp]
  \centering
  \begin{subfigure}[t]{0.48\textwidth}
    \centering
    \includegraphics[width=\linewidth]{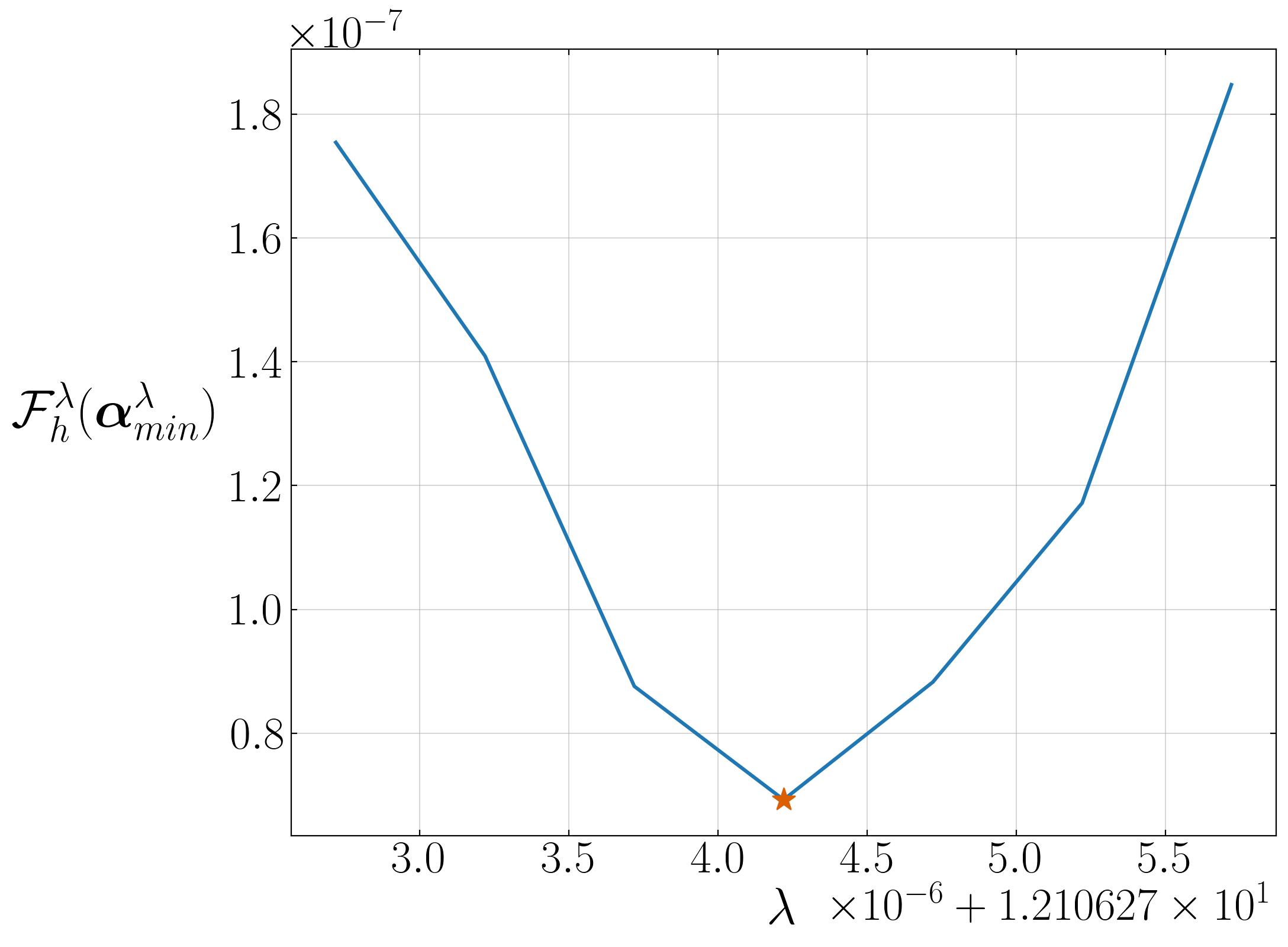}
    \caption{}
    \label{fig:plot}
  \end{subfigure}\hfill
  \begin{subfigure}[t]{0.48\textwidth}
    \centering
    \includegraphics[width=\linewidth]{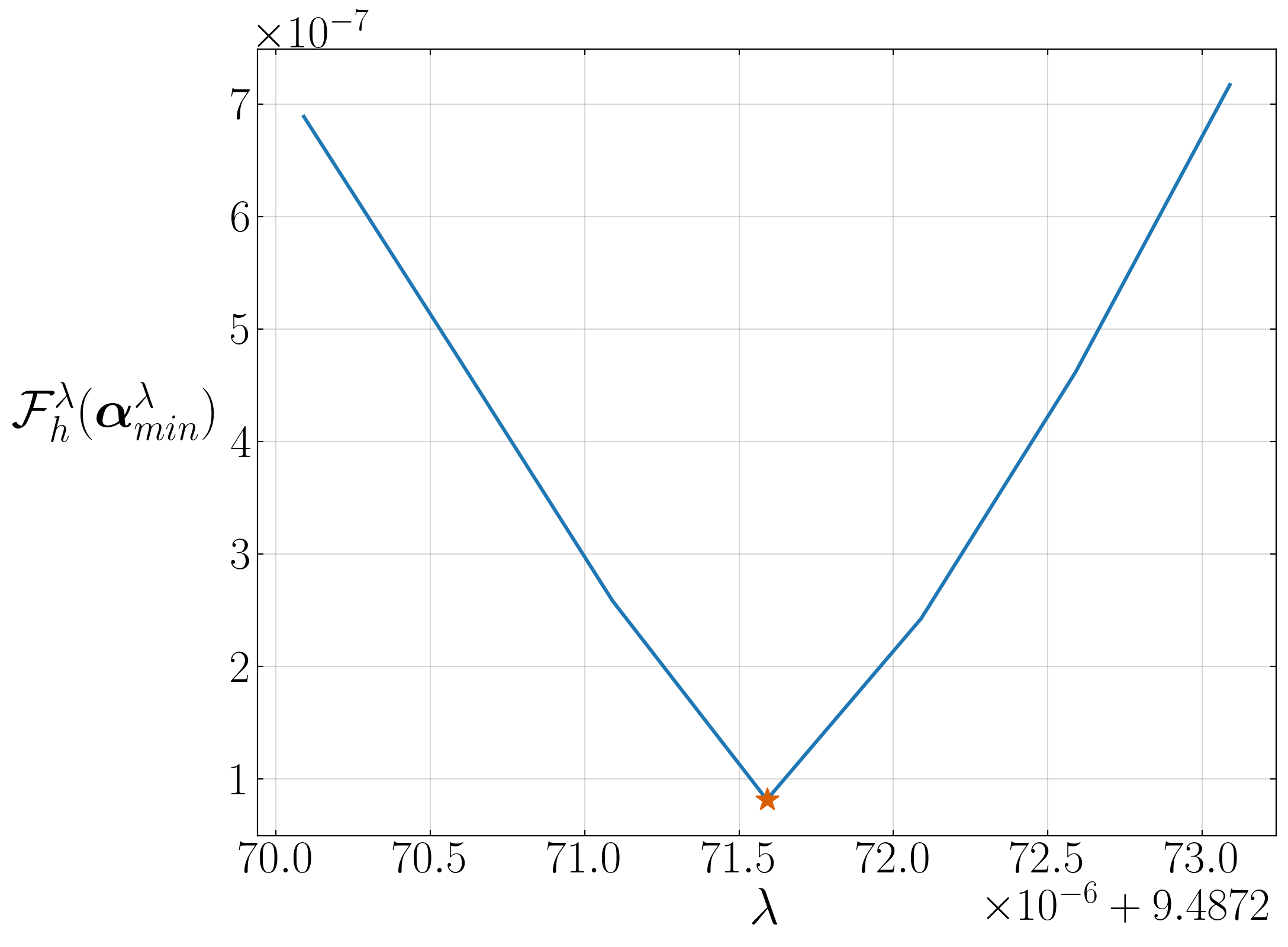}
    \caption{}
    \label{fig:PlastureC1smallGraph}
  \end{subfigure}
  \caption{Detailed view in the neighbourhood of a local minimum of the function $\mathcal{F}^\lambda_h(\alpha_{min}^\lambda)$ in terms of the test value $\lambda$, for the ellipse $\Omega_E$ (A) and the star-shaped domain $\Omega_S$ (B).}
  \label{fig:smallGraph}
\end{figure}

We compare our results with those obtained using the classical Finite Element approach. The locations of the local minima in Figure \ref{fig:bigGraph} closely match the eigenvalues computed with the \texttt{EigenValue} routine from the FreeFem++ library \cite{FreeFem}. A more detailed comparison between our MPS-based method and the Finite Element Method in the computation of the seventh eigenvalue of the problem \eqref{EVP-radial} in the case \eqref{ex:plasture} can be found in Tables \ref{table:MPS} and \ref{table:FEM}. We emphasize that our method is significantly more efficient in terms of memory usage: further mesh refinement for the Finite Element Method was not possible on a machine with 16 GB of RAM, whereas our approach achieves higher accuracy without increasing memory requirements.

\begin{table}[htbp]
\renewcommand{\arraystretch}{1.5}
\normalsize
\begin{center}
\begin{tabular}{|c|c|c|c|c|c|c|c|}
\hline  $J$ & $N_h$ & $N_{\partial\Omega}$ & $N_{\Omega}$ & $\lambda$ & $\mathcal{F}^\lambda_h(\aaa_{min}^\lambda)\leq $  & Memory usage (MB) \\
\hline
 $800$ & $1500$ & $3000$ & $2500$ & $9.48727879 \pm 10^{-8}$ & $2.075\times 10^{-7}$ & $700$ \\
\hline
 $1000$ & $2500$ & $4500$ & $4000$ & $9.48727264 \pm 10^{-8}$ & $9.495\times 10^{-8}$ & $1000$ \\
 \hline
  $1200$ & $3000$ & $6000$ & $5000$ & $9.48727158 \pm 10^{-8}$ & $8.370\times 10^{-8}$ & $1500$\\
 \hline
\end{tabular}
\caption{Upper bound for the local minimum of $\mathcal{F}_h^\lambda(\aaa_{\min}^\lambda)$ corresponding to the seventh eigenvalue for the star-shaped domain $\Omega_S$, together with the local minimizer $\lambda$. The error tolerance of $10^{-8}$ for $\lambda$ corresponds to the smallest grid spacing $\mu$ employed during the optimisation process. The results are obtained for varying values of $J$ (angular Fourier series truncation parameter), $N_h$ (number of one-dimensional FEM nodes), and $N_{\partial\Omega}$ and $N_{\Omega}$ (numbers of collocation points on the boundary and in the interior of the domain, respectively).}\label{table:MPS}
\end{center}
\end{table}

\begin{table}[htbp]
\renewcommand{\arraystretch}{1.5}
\normalsize
\begin{center}
\begin{tabular}{|c|c|c|}
\hline  $N_{\Omega}$ & $\lambda$  & Memory usage (MB) \\
\hline %1200 on boundary
 $149496$ &  $9.487256256$ & $4750$ \\
\hline %1400 on boundary
 $205003$ & $9.487259679$ & $7400$\\
 \hline %1500 on boundary
  $234293$ & $9.487260903$ & $8900$\\
 \hline
\end{tabular}
\caption{Values of the classical Finite Element approximation $\lambda$ of the seventh eigenvalue for the domain $\Omega_S$, together with the number of nodes of the triangulation and the memory usage.}\label{table:FEM}
\end{center}
\end{table}

\subsection{Numerical approximation of eigenfunctions}\label{sec:numerical-eigenfunctions} We now present numerical plots of the eigenfunctions of \eqref{EVP-radial} computed using the proposed method. The parameters used in all computations are:
\[J=1200;\quad N_h=3000;\quad N_{\partial\Omega}=6000;\quad  N_{\Omega}=5000.\]
Figure \ref{fig:ElipsaFunctie} displays the approximation of the third eigenfunction of \eqref{EVP-radial} for the elliptic domain considered in \eqref{ex:ellipsis}. An upper bound for the local minimum of the discrete quotient $\mathcal{F}_h^\lambda(\aaa_{min}^\lambda)$ and an estimation of the local minimizer $\lambda$ -- which, in turn approximates the third eigenvalue for the domain $\Omega_E$  -- are reported below:
\[
\lambda=12.10627421\pm 10^{-8}; \quad \mathcal{F}^\lambda_h(\aaa_{min}^\lambda)\leq 7.056\times 10^{-8}.
\]

\begin{figure}[htbp]
  \centering
    \includegraphics[width=0.5\linewidth]{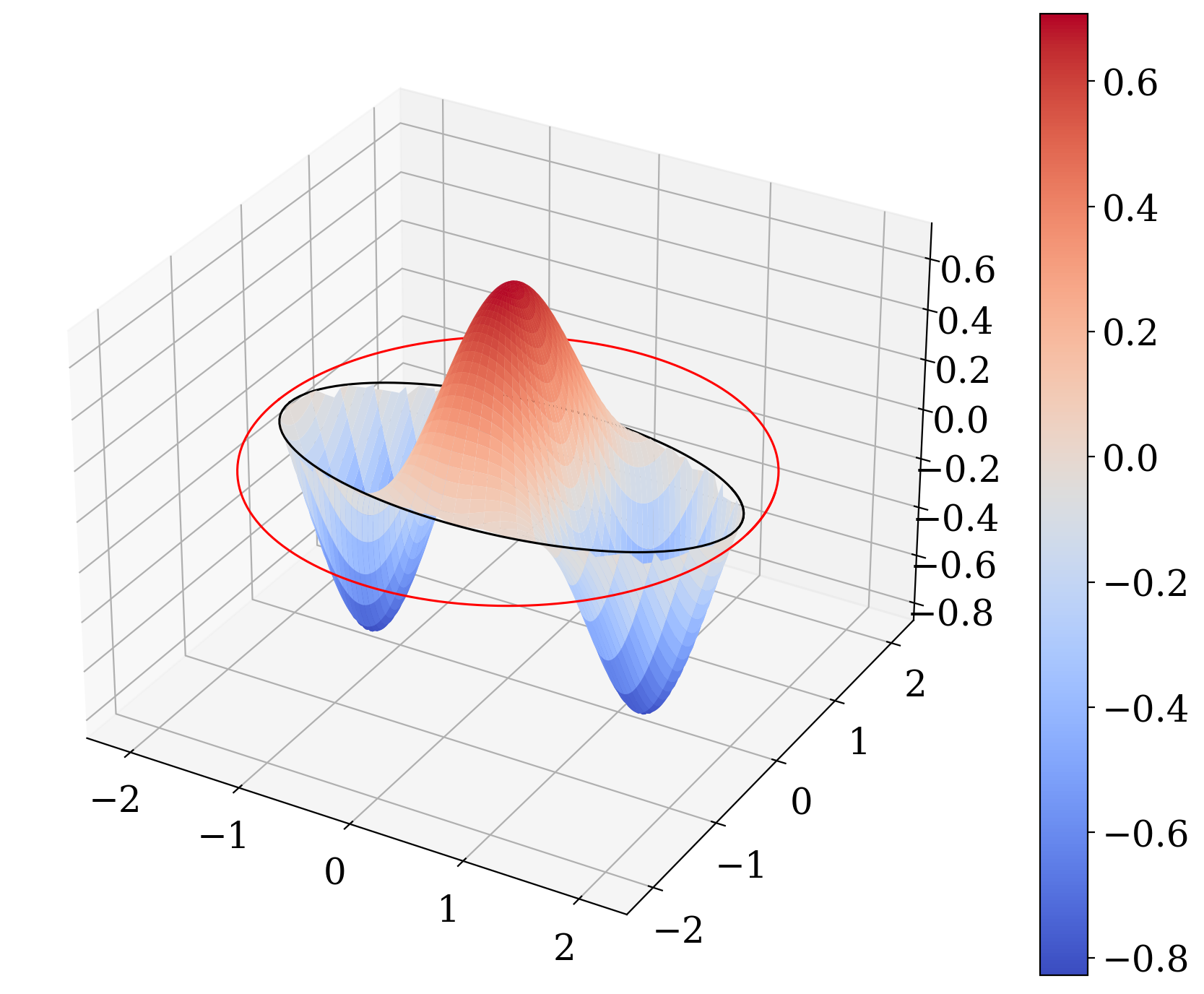}
  \caption{Approximation of the eigenfunction corresponding to the third eigenvalue in the case of the elliptic domain $\Omega_E$.}
  \label{fig:ElipsaFunctie}
\end{figure}

We next consider the star-shaped domain introduced in \eqref{ex:plasture}. In this case, near the seventh eigenvalue $\lambda^*$ which is approximated in Table \ref{table:MPS}, the three smallest singular values of the matrix $\mathcal{Q}_{\partial\Omega}$ are:
\[\sigma^\lambda_1=8.370\times 10^{-8},\quad \sigma^\lambda_2= 8.518\times 10^{-8},\quad \sigma^\lambda_3=4.009\times 10^{-2}.\]
This separation indicates that the eigenvalue $\lambda^*$ has multiplicity two. Figure \ref{fig:PlastureTurturi} illustrates numerical approximations of two orthonormal eigenfunctions spanning the corresponding eigenspace.

\begin{figure}[htbp]
  \centering
  \begin{subfigure}[t]{0.46\textwidth}
    \centering
    \includegraphics[width=\linewidth]{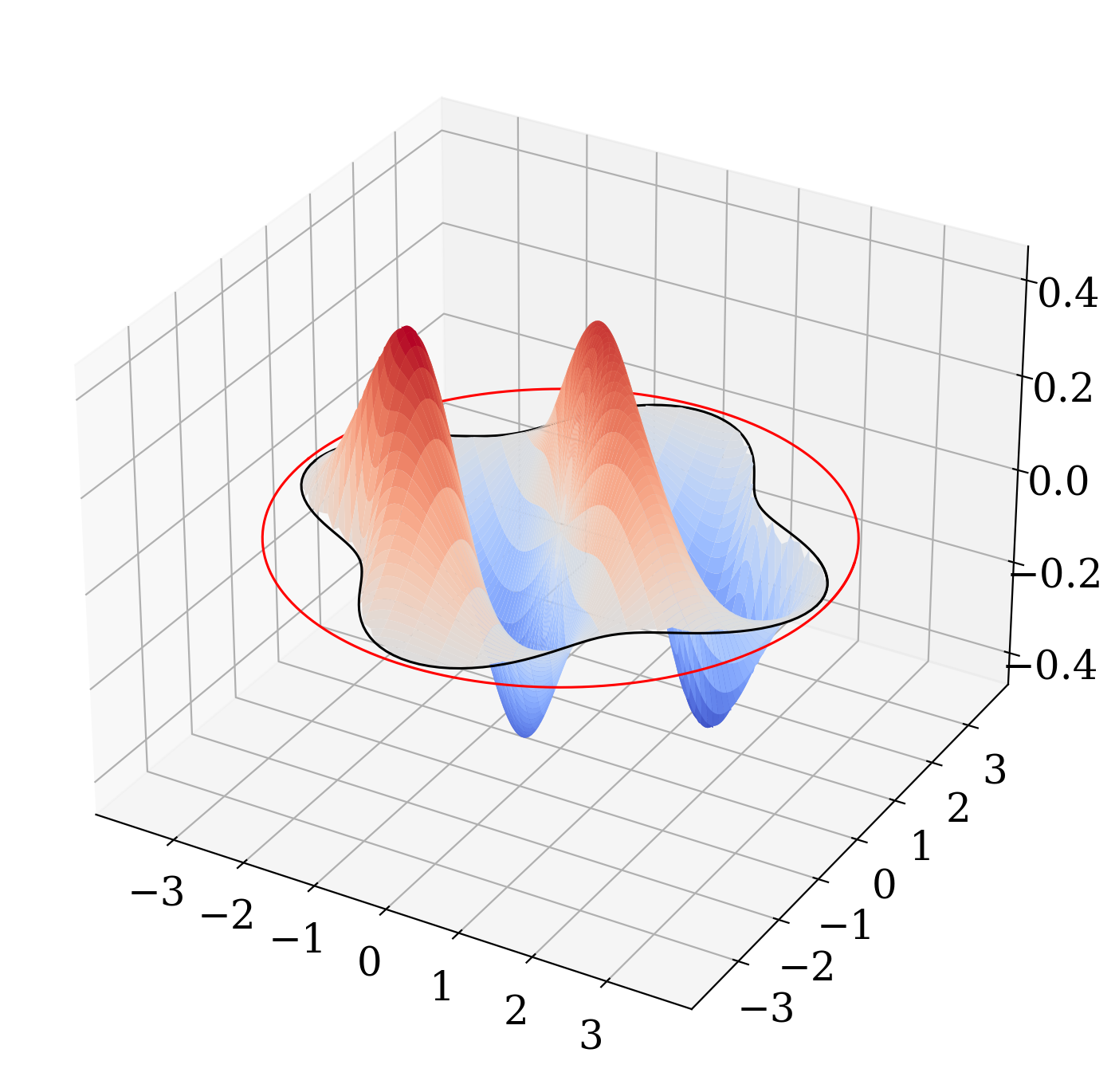}
    \label{fig:PlastureTurturi1}
  \end{subfigure}\hfill
  \begin{subfigure}[t]{0.54\textwidth}
    \centering
    \includegraphics[width=\linewidth]{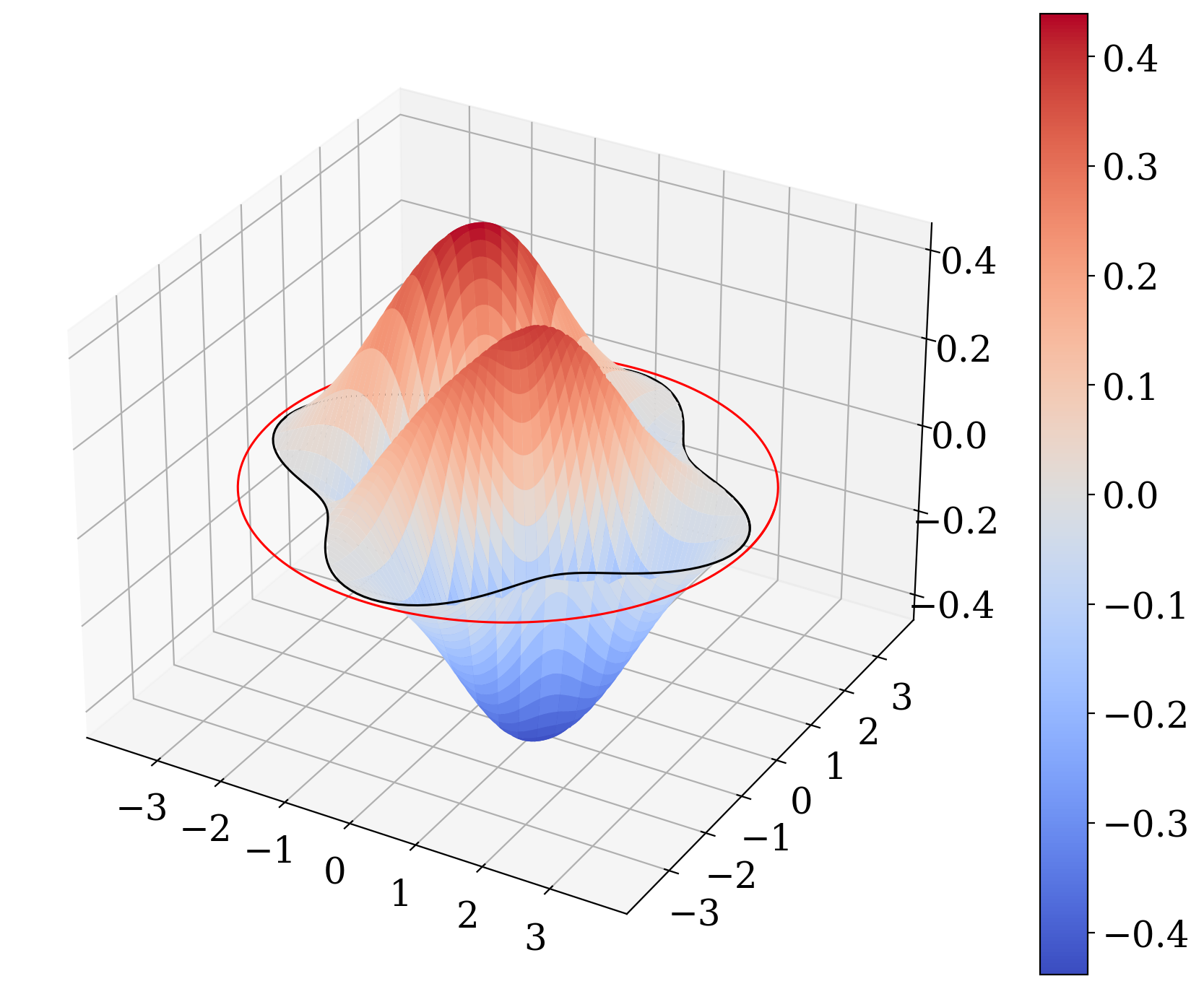}
    \label{fig:PlastureTurturi2}
  \end{subfigure}
  \caption{The approximations of two orthonormal eigenfunctions corresponding to the seventh eigenvalue in the case of the star-shaped domain $\Omega_S$.}
  \label{fig:PlastureTurturi}
\end{figure}

\section{Conclusion and further research directions}\label{sec:conclusion}
In this work, we have developed a numerical method for approximating eigenpairs of the Schrödinger operator $-\Delta + V$ with a radial potential. Compared with the classical Finite Element approach, the proposed method achieves high accuracy while remaining significantly more memory-efficient. 

The method combines techniques from theory of Partial and Ordinary Differential Equations and does not rely on tools from Complex Analysis. As a result, it applies to general $C^1$ radial potentials, not only to analytic ones. Moreover, the approach is able to detect entire eigenspaces associated with eigenvalues of higher multiplicity.\\

Several directions for further research naturally arise from this work:
\begin{itemize}
\item \textbf{Extension to three dimensions and Neumann eigenpairs.}
  We expect that the method can be extended to higher-dimensional problems, without substantial modifications of the underlying reasoning, by replacing the circular basis functions $\cos(j\theta)$ and $\sin(j\theta)$ with the relevant spherical harmonics.
  
  We also expect that our method can be adapted to approximate the Neumann solutions of the Helmholtz equation with potential.  However, the details of the proofs and the choice of parameters must be carefully adapted to this settings.
  
\item \textbf{Optimal parameter selection.}
  A systematic empirical and theoretical investigation of the optimal choice and interplay of the parameters $\mu$, $J$, $h$, $N_{\partial\Omega}$, and $N_{\Omega}$ would improve both the efficiency and robustness of the method. A key step would be establishing a quantitative Runge property in $H^1$ norms, which would allow the derivation of explicit error bounds in Proposition~$\ref{prop:runge}$.

\item \textbf{Extension to non-radial potentials.}
  Extending the approach to general potentials remains challenging. Even for simple non-radial potentials, such as $V(x)=x$, the polar decomposition leads to an infinite coupled system of ordinary differential equations rather than independent equations for each basis functions.

\item \textbf{Analysis of the regularisation procedure.}
  A deeper study of the regularisation introduced in Section \ref{sec:numerical-eigenvalues} is needed, including a rigorous proof of its effectiveness and the identification of an optimal regularisation threshold.

\end{itemize}

\section*{Acknowledgement}
The author would like to thank Dr. Mihai Bucataru from the University of Bucharest for the fruitful discussions about the subject of this paper.

\section*{Funding}
The author was partially supported for this work by the Romanian Government, through Planul Naţional de Redresare şi Rezilienţă PNRR-III-C9-2023-I8, 
CF 149/31.07.2023 \emph{Conformal Aspects of Geometry and Dynamics}.

\section*{References}
\printbibliography[heading=none]

\appendix

\section{Additional properties of functions in the space $H_{j,r}([0,R])$}

 The following proposition describing the behaviour near the origin of functions in the $H_{j,r}([0,R])$ spaces is used to construct the Finite Element basis for solving the equation \eqref{EVP-j-W}.

 \begin{proposition}\label{prop:Hjr-is-zero-in-zero}
     Let $v\in H_{j,r}([0,R])$, with $j>0$. Then, $v$ has a continuous representative on $[0,R]$ with $v(0)=0$.
 \end{proposition}
 \begin{proof}
    The fact that $v$ has a continuous representative on $(0,R]$ is a consequence of \cite[Theorem 8.2]{brezis} applied on each compact subinterval $\mathcal{I}$ of $(0,R]$, since $v$ is in $H^1(\mathcal{I})$.
    
    To prove the continuity in the origin, we first show that $|v|^2\in W^{1,1}(0,R)$, and then, by \cite[Theorem 8.2]{brezis}, $|v|^2$ will have a continuous representative on $[0,R]$. Indeed, since $v\in H_{j,r}([0,R])$ with $j>0$, the following quantities are finite:
    \[I_1\coloneqq \int_0^R |v'(r)|^2\, r\,\dd r\quad \text{and}
    \quad I_2\coloneqq \int_0^R |v(r)|^2\, \frac{1}{r}\,\dd r.\]
    The fact that $I_2<\infty$ immediately leads to $|v|^2\in L^1([0,R])$. Furthermore, H\"older's inequality implies that 
    \[\int_0^R |v'(r)v(r)|\dd r\leq I_1\cdot I_2<\infty.\]
    Therefore, since \cite[Corollary 8.10]{brezis} applied on each compact subinterval of $(0,R)$ implies that $2v'\,v$ is the weak derivative of $|w|^2$, we obtain that $|v|^2\in W^{1,1}(0,R)$, so $|v|^2$ has a continuous representative on $[0,R]$.

    Next, we assume by contradiction that $|v(0)|\neq 0$. Since $|v|^2$ is continuous, then $|v|$ will be bounded away from zero on a neighbourhood of the origin, which contradicts $I_2<\infty$. So, we obtain $v(0)=0$, which, together with $|v|^2\in C([0,R])$ and $v\in C((0,R])$, implies that $v$ is continuous (in the sense of a representative) on $[0,R]$.
 \end{proof}

 \begin{remark} The conclusion of Proposition \ref{prop:Hjr-is-zero-in-zero} is not true in the case $j=0$. A counterexample is the function $r\to \log\left(1+|\log(r)|\right)$ which belongs to $H_{0,r}([0,R])$, but it is unbounded near $r=0$.
 \end{remark}

The density result in the next lemma will be useful in some integration by parts arguments:
\begin{lemma}\label{lem:density-in Hjr}
For every $R>0$ and $j\geq 0$, the space of $C_c^\infty((0,R])$ functions is dense in $H_{j,r}([0,R])$. 
\end{lemma}
\begin{proof}
For every $v\in H_{j,r}([0,R])$, Proposition \ref{prop:spaces-equivalence} \ref{item:h1-second} implies that $v(r)\cos(j\theta)\in H^1(B_O(R))$. Therefore, for every $\eps>0$, Corollary \ref{cor:density-h1} provides us with a function $\varphi\in C_c^\infty\left(\overline{B_O(R)}\setminus\{O\}\right)$ such that 
\[\|\varphi-v(r)\cos(j\theta)\|_{H^1(B_O(R))}\leq \eps.\]
If $\varphi_j^c$ is the Fourier coefficient of $\varphi$ corresponding to $\cos(j\theta)$ as defined in \eqref{eq:def-u0-ur-vr}, then $\varphi_j^c\in C_c^\infty((0,R])$ and Proposition \ref{prop:spaces-equivalence} \ref{item:h1-first} implies that:
\[\|\varphi_j^c-v\|_{H_{j,r}([0,R])}\leq \frac{2\sqrt{2}\,\eps}{\sqrt{\pi}}. \qedhere\]
\end{proof}

The following proposition contains estimates for the first and second derivatives of functions $v\in H_{j,r}([0,R])$ that satisfy $v(r)\cos(j\theta)\in H^2(B_O(R))$, and is used to obtain error bounds between the linear interpolator $\tilde v_h$ and the function $v$ in Proposition \ref{prop:fem-interpolator}.
\begin{proposition}\label{prop:h2-estimates}
Let $j$ be a non-negative integer and the function $v\in H_{j,r}([0,R])$ such that $v(r)\cos(j\theta)\in H^2(B_O(R))$. Then, the following estimate holds true:
\begin{equation}\label{eq:h2-estimate-1}
\|v''\|_r\leq \frac{\sqrt{14}}{\sqrt{\pi}} \left\|D^2[v(r)\cos(j\theta)]\right\|_{L^2(B_O(R))}.\
\end{equation}
Moreover, if j=0, then $v$ has a continuous representative on $[0,R]$ and
\begin{equation}\label{eq:h2-estimate-2}
\int_0^R \frac{|v'(r)|^2}{r}\dd r \leq \frac{14}{\pi}\left\|D^2[v(r)\cos(j\theta)]\right\|_{L^2(B_O(R))}^2.
\end{equation}
On the other hand, if $j\geq 1$:
\begin{equation}\label{eq:h2-estimate-3}
j^2 \int_0^R \left[\frac{v'(r)}{r}-\frac{v(r)}{r^2}\right]^2 r\, \dd r\leq \frac{6}{\pi} \left\|D^2[v(r)\cos(j\theta)]\right\|_{L^2(B_O(R))}^2.
\end{equation}
\end{proposition}
\begin{proof} First of all, Propositions \ref{prop:spaces-equivalence} and \ref{prop:wcos-inH2} imply that the weak derivatives $v''$ and $v'$ make sense on $(0,R)$. To prove the inequalities, we consider a test function $w\in C_c^\infty(0,R)$ and notice that, by direct computation, for every integer $k\in \mathbb{Z}$,
\[\begin{aligned}\partial_{xx} [w(r)\cos(k\theta)]=& \frac{1}{4}\left[\frac{[r\, w(r)]''}{r}+\frac{2k-3}{r}\,w'(r)+\frac{k^2-2k}{r^2} \, w(r)\right]\cos((k-2)\theta)\\
&+ \frac{1}{2}\left[\frac{[r\,w(r)]''}{r}-\frac{1}{r}\, w'(r)-\frac{k^2}{r^2} \, w(r)\right]\cos(k\theta)\\
&+\frac{1}{4}\left[\frac{[r\, w(r)]''}{r}-\frac{2k+3}{r}\, w'(r)+\frac{k^2+2k}{r^2}\, w(r)\right]\cos((k+2)\theta)\\
\end{aligned}\]
    Testing $v(r)\cos(j\theta)$ against $\partial_{xx}[w(r)\cos(k\theta)]$ and integrating by parts, the formula above leads to:
    \[\begin{aligned}\int_{B_O(R)}& \partial_{xx}[v(r)\cos(j\theta))]w(r)\cos(k\theta)\dd\x = \\
    =&\frac{1}{4}\int_0^R v(r)\left[\frac{[r\,w(r)]''}{r}+\frac{2k-3}{r}\,w'(r)+\frac{k^2-2k}{r^2} \,w(r)\right]r\, \dd r \int_0^{2\pi}\cos((k-2)\theta)\cos(j\theta)\dd\theta\\
    &+\frac{1}{2}\int_0^R v(r)\left[\frac{[r\,w(r)]''}{r}-\frac{1}{r}\,w'(r)-\frac{k^2}{r^2} \,w(r)\right] r\, \dd r \int_0^{2\pi}\cos(k\theta)\cos(j\theta)\dd\theta\\
    &+ \frac{1}{4}\int_0^R v(r) \left[\frac{[r\,w(r)]''}{r}-\frac{2k+3}{r}\,w'(r)+\frac{k^2+2k}{r^2} \,w(r)\right] r\, \dd r \int_0^{2\pi}\cos((k+2)\theta)\cos(j\theta)\dd\theta.
    \end{aligned}\]
    An integration by parts in the variable $r$ implies
    \[\begin{aligned}\int_{B_O(R)}& \partial_{xx}[v(r)\cos(j\theta))]w(r)\cos(k\theta)\dd\x = \\
    =&\frac{1}{4}\int_0^R \left[v''(r)-\frac{2k-3}{r}\,v'(r)+\frac{k^2-2k}{r^2} \,v(r)\right]w(r)r\, \dd r \int_0^{2\pi}\cos((k-2)\theta)\cos(j\theta)\dd\theta\\
    & + \frac{1}{2} \int_0^R \left[v''(r)+\frac{1}{r}\,v'(r)-\frac{k^2}{r^2} \,v(r)\right]w(r)r\, \dd r \int_0^{2\pi}\cos(k\theta)\cos(j\theta)\dd\theta\\
        & + \frac{1}{4} \int_0^R \left[v''(r)+\frac{2k+3}{r}\,v'(r)+\frac{k^2+2k}{r^2} \,v(r)\right]w(r)r\, \dd r \int_0^{2\pi}\cos((k+2)\theta)\cos(j\theta)\dd\theta.
    \end{aligned}\]
    Choosing $k\in\{j-2,j,j+2\}$, we obtain by the Cauchy-Schwarz inequality and by duality that:
    \begin{align}
    \label{eq:h2-estimate-I1}
        \int_0^R \left[v''(r)-\frac{2j+1}{r}\,v'(r)+\frac{j^2+2j}{r^2} \,v(r)\right]^2r\, \dd r &\leq \frac{16}{\pi} \left\|D^2[v(r)\cos(j\theta)]\right\|_{L^2(B_O(R))}^2\\
        \label{eq:h2-estimate-I2}
        \int_0^R \left[v''(r)+\frac{1}{r}\,v'(r)-\frac{j^2}{r^2}\, v(r)\right]^2r\, \dd r &\leq \frac{4}{\pi} \left\|D^2[v(r)\cos(j\theta)]\right\|_{L^2(B_O(R))}^2\\
        \label{eq:h2-estimate-I3}
        \int_0^R \left[v''(r)+\frac{2j-1}{r}\, v'(r)+\frac{j^2-2j}{r^2} v(r)\right]^2r\, \dd r &\leq \frac{32}{\pi} \left\|D^2[v(r)\cos(j\theta)]\right\|_{L^2(B_O(R))}^2
    \end{align}
    We note that if $j\neq 2$, then \eqref{eq:h2-estimate-I3} also takes place with the coefficient $\frac{16}{\pi}$, instead of $\frac{32}{\pi}$.

    Next, we use the inequality $(a\pm b)^2 \leq 2 (a^2+b^2)$ and obtain from \eqref{eq:h2-estimate-I1} and \eqref{eq:h2-estimate-I3} that:
    \begin{equation}\label{eq:h2-estimate-I4}
\int_0^R \left[v''(r)-\frac{1}{r}\, v'(r)+\frac{j^2}{r^2}\, v(r)\right]^2r\, \dd r \leq \frac{24}{\pi} \left\|D^2[v(r)\cos(j\theta)]\right\|_{L^2(B_O(R))}^2       
    \end{equation} 
    and 
    \[j^2 \int_0^R \left[\frac{v'(r)}{r}-\frac{v(r)}{r^2} \right]^2r\, \dd r \leq \frac{6}{\pi} \left\|D^2[v(r)\cos(j\theta)]\right\|_{L^2(B_O(R))}^2, \]
    which is exactly \eqref{eq:h2-estimate-3}. Similarly, \eqref{eq:h2-estimate-I2} and \eqref{eq:h2-estimate-I4} lead to \eqref{eq:h2-estimate-1} and \eqref{eq:h2-estimate-2}.
    
    In order to prove that, for $j=0$, $v$ has a continuous representative on $[0,R]$, we will show that the continuous representative on $(0,R]$ provided by Proposition \ref{prop:wcos-inH2} can be extended continuously in the point $0$. Indeed, since $v\in H_{j,r}([0,R])$ and satisfies \eqref{eq:h2-estimate-2}, the Cauchy-Schwarz inequality implies that, for every $r_1,r_2\in (0,R)$, with $r_1<r_2$,
    \[\begin{aligned}\left||v(r_1)|^2-|v(r_2)|^2\right|\leq 2\int_{r_1}^{r_2}|v'(r)v(r)|\dd r&\leq \int_{r_1}^{r_2}|v(r)|^2 r\,\dd r \int_{r_1}^{r_2} \frac{|v'(r)|^2}{r}  \dd r \\
    &\leq \int_{0}^{R}|v(r)|^2 r\,\dd r \int_{0}^{R} \frac{|v'(r)|^2}{r}  \dd r <+\infty.
    \end{aligned}\]
    As a result, for every sequence $(\tau_n)_{n\geq 1}$ of positive numbers converging to zero, the sequence $|v(\tau_n)|^2$ is Cauchy, which in turn implies that $\ell\coloneqq \lim_{\tau\to 0} |v(\tau)|^2$ exists and it is finite.
    
    If $\ell=0$, this implies that $\lim_{\tau\to 0} v(\tau)=0$, which finishes the proof in this case. On the other hand, for $\ell\neq 0$, we assume by contradiction that $\lim_{\tau\to 0} v(\tau)$ does not exist. This implies that there exist two sequences $(\tau_n^1)_{n\geq 1}$ and $(\tau_n^2)_{n\geq 1}$ such that:
    \[\lim_{n\to \infty} v(\tau^1_n)=\sqrt{|\ell|} \quad \text{and}\quad \lim_{n\to \infty} v(\tau^2_n)=-\sqrt{|\ell|}.\]
    However, since $v$ is continuous on $(0,R]$, we can construct a sequence $(\tau^3_n)_{n\geq 1}$ such that $v(\tau^3_n)=0$ and, for every $n\geq 1$, $\tau^3_n$ lies between $\tau^1_n$ and $\tau^2_n$. Therefore,
    \[\lim_{n\to \infty} |v(\tau^3_n)|^2=0,\]
    which contradicts the fact that $\ell\neq 0$. Therefore, the limit $\lim_{\tau\to 0} v(\tau)$ exists and it is finite, which finishes the proof of the lemma.
\end{proof}

\section{Proof of results in Section \ref{sec:no-eigenpair-missed}}\label{acaret:proofs}
This section on the appendix contains proofs that were deferred in order to maintain the flow of the paper.
\subsection*{Proof of Proposition \ref{prop:uJ approximates-u}}
    Since $u\in H^2(B_O(R))\subset L^2(B_O(R))$, the function $\theta\to u(r,\theta)$ belongs to $L^2([0,2\pi])$, for almost every $r\in (0,R)$. Therefore, a similar reasoning to \cite[Chapter IV]{SteinWeiss2016} implies that, for a.e. $r\in (0,R)$, the series $(u^J(r,\cdot))_{J\geq 1}$ converges in $L^2([0,2\pi])$ to $u(r,\cdot)$. The error term is equal to the tail of the series, i.e.
    \begin{equation}\label{eq:err-u-uJ}
        u-u^J = \sum_{j>J}\left[u_j^c(r) \cos(j\theta)+ u_j^c(r)\sin(j\theta) \right],
    \end{equation}
    so the $L^2(B_O(R))$ truncation error is equal to:
    \begin{equation}\label{eq:err-u-uJ-L2}
    \|u-u^J\|_{L^2(B_O(R))}^2 = \pi \sum_{j>J}\int_0^R \left[|u_j^c(r)|^2+ |u_j^s(r)|^2\right]r\,\dd r.
    \end{equation}

    If $u$ were smooth, the next step in evaluating the rate of decay of \eqref{eq:err-u-uJ-L2} would be to integrate by parts in \eqref{eq:def-u0-ur-vr} with respect to $\theta$. However, we only know that $u$ is weakly differentiable, so we test $u_j^c$ against an arbitrary $C_c^\infty((0,R))$ function and then integrate by parts on $B_O(R)$ to obtain that, indeed, for almost every $r\in (0,R)$,
    \begin{equation}\label{eq:uj-1-over-j}
        u_j^c(r)= -\frac{1}{j\, \pi}\int_0^{2\pi} \partial_\theta u(r,\theta) \sin(j\theta) \dd \theta,
    \end{equation}
    where by $\partial_r$ and $\partial_\theta$ we understand, as usual, the following combinations of weak derivatives:
    \begin{equation}\label{eq:radial-derivatives-as-cartesian}
    \left\{\begin{aligned}
        \partial_r  &= \cos(\theta) \partial_x + \sin(\theta) \partial_y \\
        \partial_\theta  &= -r \sin(\theta) \partial_x + r\cos(\theta) \partial_y 
    \end{aligned}\right.\quad \text{ for } \x=(x,y)=(r\cos(\theta),r\sin(\theta))\neq (0,0).
    \end{equation}

    Reiterating the reasoning above, we obtain that:
     \begin{equation}\label{eq:uj-1-over-j2}
     u_j^c(r)= -\frac{1}{j^2\, \pi}\int_0^{2\pi} \partial_{\theta\theta} u (r,\theta)\cos(j\theta) \dd \theta,
    \end{equation}
     where by $\partial_{\theta\theta}$ we mean that the operator $\partial_\theta$ is applied twice. Therefore, \eqref{eq:radial-derivatives-as-cartesian} implies that for almost every $r\in (0,R)$, 
     \[|u_j^c(r)|\leq C(R) \frac{1}{j^2 \pi} \int_0^{2\pi} \left[\|D^2 u(r,\theta)\| + \|\nabla u(r,\theta)\|\right] \dd \theta\]
     and the same inequality is valid also for $u_j^s$. Using the Cauchy-Schwarz inequality and integrating on $[0,R]$ in the inequality above, we obtain by \eqref{eq:err-u-uJ-L2} that:
     \[\|u^J-u\|_{L^2(B_O(R))}^2\leq C(R) \|u\|^2_{H^2(B_O(R))} \sum_{j\geq J} \frac{1}{j^4},\]
     which leads to the $L^2$-norm error estimate by classical properties of the generalised harmonic series.

     To obtain the $H^1$-norm error estimate, we compute the weak gradient of $u_j^c(r)\cos(j\theta)$. We note that by \cite[Remark 6.1.18]{Heinonen2015sobolev}, $H^1(B_O(R)\setminus\{O\})=H^1(B_O(R))$, so it is enough to test against $C_c^\infty(B_O(R)\setminus\{O\})$ functions to compute the weak gradient. In this way, direct computations lead to:
     \begin{align*}
        \partial_x [u_j^c(r) \cos(j\theta)]&= \cos(\theta) \cos(j\theta) \frac{1}{\pi}\int_0^{2\pi} \left[\partial_x u(r,\sigma) \cos(\sigma) + \partial_y u (r,\sigma) \sin(\sigma)\right] \cos(j\sigma) \dd \sigma \\
        & \quad \notag
        - \frac{j}{r} \sin(\theta) \sin(j\theta) \frac{1}{\pi}\int_0^{2\pi} u(r,\sigma)\cos(j\sigma) \dd \sigma\\
         &= \notag \frac{1}{2\pi}\left[\cos((j-1)\theta)+\cos(j\theta)\right]\int_0^{2\pi} \left[\partial_x u(r,\sigma) \cos(\sigma) + \partial_y u(r,\sigma) \sin(\sigma)\right] \cos(j\sigma) \dd \sigma \\
        & \quad \notag
        - \frac{j}{2\pi r}\left[\cos((j-1)\theta)-\cos(j\theta)\right]\int_0^{2\pi} u(r,\sigma)\cos(j\sigma) \dd \sigma\\
         & \notag = -\frac{1}{2\pi j}\left[\cos((j-1)\theta)+\cos(j\theta)\right]\int_0^{2\pi} \partial_{\sigma}\left[\partial_x u(r,\sigma) \cos(\sigma) + \partial_y u (r,\sigma) \sin(\sigma)\right] \sin(j\sigma) \dd \sigma \\
        & \quad \notag
        + \frac{1}{2 \pi rj} \left[\cos((j-1)\theta)-\cos(j\theta)\right]\int_0^{2\pi} \partial_{\sigma\sigma} u(r,\sigma)\cos(j\sigma) \dd \sigma,
     \end{align*}
     where the last inequality is obtained by the same procedure as \eqref{eq:uj-1-over-j} and \eqref{eq:uj-1-over-j2}. Using the expression of the spherical derivative $\partial_\theta$ in \eqref{eq:radial-derivatives-as-cartesian} we can cancel out the $\frac{1}{r}$ in the last term above. Furthermore, the orthogonality of the Fourier basis functions $(\cos(j\theta))_{j\in \mathbb{Z}}$ implies that:
    \[\left\|\sum_{j>J}\partial_x [u_j^c(r) \cos(j\theta)] \right\|_{L^2(B_O(R))}^2\leq C(R) \|u\|_{H^2(B_O(R))}^2 \sum_{j>J} \frac{1}{j^2}\leq \frac{C(R)}{J} \|u\|_{H^2(B_O(R))}^2.
    \]
    The same inequality is obtained for $\partial_y [u_j^c(r)\cos(j\theta)]$, $\partial_x [u_j^s(r)\sin(j\theta)]$ and $\partial_y[u_j^s(r)\sin(j\theta)]$. Therefore, the series \eqref{eq:err-u-uJ} converges in $H^1(B_O(R))$ with a sum bounded from above by \linebreak $\frac{C(R)}{\sqrt{J}} \|u\|_{H^2(B_O(R))}$. This leads to the desired $H^1$-norm error estimate.\qed

\subsection*{Sketch of proof for Proposition \ref{prop:spaces-equivalence}}
Statements \ref{item:l2-first} and \ref{item:l2-second} follow immediately from the coarea formula.

The implication \ref{item:h1-first} follows from \eqref{eq:uj-1-over-j} and \eqref{eq:radial-derivatives-as-cartesian}, since one can prove that the weak derivative of $u_j^c$ is:
\[(u_j^c)'(r)=\frac{1}{\pi}\int_0^{2\pi} \partial_r u(r,\theta) \cos(j\theta) \dd \theta.\]
Indeed, by testing $u_j^c(r)$ against $w\in C_c^\infty(0,R)$ and integrating by parts on $B_O(R)$, we obtain:
\begin{equation}\label{eq:uj-weak-derivative}
    \begin{aligned}\int_0^R u_j^c(r)w'(r) \dd r &=\frac{1}{\pi}\int_0^R \int_0^{2\pi} u(r,\theta) \partial_r [w(r)\cos(j\theta)]r\,\dd \theta\,\dd r\\
&=\frac{1}{\pi}\int_{B_O(R)}u(r,\theta) \left\{\partial_x [w(r)\cos(j\theta)]\cos(\theta)+\partial_y[w(r)\cos(j\theta)]\sin(\theta)\right\}\frac{1}{r}\,\dd \x\\
&=-\frac{1}{\pi}\int_{B_O(R)}w(r)\cos(j\theta)\left\{\partial_x \left[ u(r,\theta) \cos(\theta)\frac{1}{r}\right]+\partial_y \left[u(r,\theta) \sin(\theta)\frac{1}{r}\right]\right\}r\,\dd \x\\
&=-\frac{1}{\pi} \int_0^R w(r) \int_0^{2\pi} \partial_r  u (r,\theta) \cos(j\theta) \dd \theta\, \dd r,
\end{aligned}.
\end{equation}\
where the last inequality is obtained by \eqref{eq:radial-derivatives-as-cartesian} and the coarea formula. 

The proof of the statement \ref{item:h1-second} follows a similar procedure, through which we obtain that the weak partial derivatives of $v(r)\cos(j\theta)$ have the following expressions:
\begin{equation}\label{w-cos-j-theta-weak-derivatives}
\begin{aligned}
    \partial_x [v(r)\cos(j\theta)]&=v'(r) \cos(j\theta)\cos(\theta)+\frac{j}{r}v(r)\sin(j\theta)\sin(\theta);\\
    \partial_y [v(r)\cos(j\theta)]&=v'(r) \cos(j\theta)\sin(\theta)-\frac{j}{r}v(r)\sin(j\theta)\cos(\theta).
\end{aligned}
\end{equation}
These expressions can be obtained using the definition of weak derivatives, once we note that \cite[Remark 6.1.18]{Heinonen2015sobolev} implies that the spaces $H^1(B_O(R))$ and $H^1(B_O(R)\setminus \{O\})$ coincide, so it is enough to test against $C_c^\infty(B_O(R)\setminus\{O\})$ functions in the definition of the weak partial derivatives on $B_O(R)$.

A similar result contained in Lemma \ref{lem:dense-in-h01}, namely the fact that $C_c^\infty(B_O(R)\setminus\{O\})$ functions are dense in $H_0^1(B_O(R))$, can be used to prove statement \ref{item:h01-first}. Indeed, if the function $u\in H_0^1(B_O(R))$ is approximated by the sequence $(\varphi^n)_{n\geq 1}$ of $C_c^\infty(B_O(R)\setminus\{O\})$ functions in the $\|\cdot\|_{H^1(B_O(R))}$ norm, then, by the inequality \ref{item:h1-first}, $u_j^c$ is approximated in the $\|\cdot\|_{j,r}$ norm by the sequence $((\varphi^n)_j^c)_{n\geq 1}$ defined as:
\[(\varphi^n)_j^c(r)\coloneqq\frac{1}{\pi} \int_0^{2\pi} \varphi^n(r,\theta)\cos(j\theta)\dd\theta.\]

The proof of \ref{item:h01-second} follows a similar procedure and the statements \ref{item:weak-first} and \ref{item:weak-second} follow immediately from the following identities:

\[\begin{aligned}
    \int_{B_O(R)}\nabla[v(r) \cos(j\theta)]\cdot \nabla  u(\x)\, \dd \x &= a_j\left(v, \int_0^{2\pi}  u(\cdot,\theta) \cos(j\theta) \dd\theta \right);\\
    \int_{B_O(R)} V(\|\x\|) v(r) \cos(j\theta) u(\x)\, \dd\x &=\int_0^R V(r) v(r)\int_0^{2\pi}  u(r,\theta) \cos(j\theta) \dd\theta\, r\,\dd r;\\
    \int_{B_O(R)}\,v(r) \cos(j\theta)  u(\x)\, \dd\x&= \left(v,\int_0^{2\pi} u (\cdot,\theta) \cos(j\theta) \dd\theta\right)_r\,,
    \end{aligned}\]
which are valid for any $v\in H_{j,r}([0,R])$ and $u\in H^1(B_O(R))$. The identities above can be proved by approximation with compactly supported functions, integration by parts and formulas \eqref{eq:radial-derivatives-as-cartesian}, in the spirit of \eqref{eq:uj-weak-derivative}. \qed

 \subsection*{Proof of Proposition \ref{prop:wcos-inH2}} Proposition \ref{prop:spaces-equivalence} \ref{item:h1-second} implies that $v\in H_{j,r}([0,R])$. In order to check the definition of the second weak derivative, we test $v$ against $w''$, where $w\in C_c^\infty(0,R)$:
 \begin{equation}\label{eq:weak-second-derivative-w}\begin{aligned}
 \int_0^R v(r) w''(r)\dd r &= \frac{1}{\pi}\int_0^R \int_0^{2\pi}v(r)\cos(j\theta) \partial_{rr}[w(r)\cos(j\theta)]\dd\theta\, \dd r \\
 & = \frac{1}{\pi} \int_0^R  \int_0^{2\pi}v(r)\cos(j\theta) \frac{1}{r}\partial_{rr}[w(r)\cos(j\theta)]\dd\theta\, r\,\dd r.
 \end{aligned}
 \end{equation}
Direct calculations involving the formula of the Laplace operator in polar coordinates lead to:
\[\frac{1}{r}\partial_{rr} [w(r)\cos(j\theta)]=\Delta \left[\frac{1}{r}w(r) \cos(j\theta)\right] +\frac{1}{r^2} w'(r) \cos(j\theta) + \frac{j^2-1}{r^3} w(r) \cos(j\theta).
 \]
 We insert this formula into \eqref{eq:weak-second-derivative-w} and then perform an integration by parts on $B_O(R)$ to obtain:
\[ \begin{aligned} & \int_0^R v(r) w''(r)\dd r \\
&= \frac{1}{\pi}\int_0^R \int_0^{2\pi}\Delta\left[v(r)\cos(j\theta)\right]\frac{1}{r} w(r)\cos(j\theta)\dd\theta\, r\,\dd r + \int_0^R v(r)\frac{1}{r}w'(r)\dd r + (j^2-1)\int_0^R v(r) \frac{1}{r^2}w(r) \dd r. \\
\end{aligned}
\]
Eventually, an integration by parts in the second integral of the right-hand side above leads to the conclusion. \qed

\subsection{Proof of Proposition \ref{prop:operator-outside-of-spectrum}} The symmetry and coercivity of $b$, together with the fact that $H$ is compactly embedded in $L_r([0,R])$, implies the existence of a a solution operator $B^{-1}:L_r([0,R])\to L_r([0,R])$ which is self-adjoint and compact and its range is contained in $H$. More precisely, for every $g\in L_r([0,R])$,
\begin{equation}\label{eq:def-B-1}
    b(B^{-1}g, w)=(g,w)_r, \quad \forall w\in H.
\end{equation}
Furthermore, the inequality \eqref{eq:b-coercive} implies that the operator norm $\|B^{-1}\|_{L_r([0,R])\to H}$ is less than or equal to $M$. 

Next, \eqref{eq:b-coercive} and \eqref{eq:generic-elliptic-lambda*} imply that $\ker(B^{-1})=H^\perp$, i.e. the orthogonal space to $H$ with respect to the $L_r([0,R])$ scalar product. Moreover, $B^{-1}$ is self-adjoint and compact on $L_r([0,R])$, so \cite[Theorem 6.11]{brezis} implies that there exists an orthonormal Hilbert basis $(\xi_n^*)_{n=1}^{{\rm dim}(H)}$  of the closure $\bar H$ of $H$ with respect to the $L_r([0,R])$ scalar product, such that, for every $n\in \overline{1,{\rm dim}(H)}$, $B^{-1}\xi_n^* =\frac{1}{\lambda_n^*} \xi_n^*$ and $(\lambda_n^*)_{n=1}^{{\rm dim}(H)}$ is a non-decreasing sequence of eigenvalues of $b$.
    
    Next, since $(\xi_n^*)_{n=1}^{{\rm dim}(H)}$ is an orthonormal Hilbert basis of $\bar H$ and $v\in H$, we can write:
    \begin{equation}\label{eq:norm-w-in-basis}
    \|v\|_r^2=\sum_{n=1}^{{\rm dim}(H)} (v,\xi_n^*)_r^2
    \end{equation}
    Since $v$ satisfies \eqref{eq:generic-elliptic-lambda*}, and $\xi_n^*\in H$  is an eigenfunction of $b$, we obtain:
    \[\lambda(v,\xi_n^*)_r= b(v,\xi_n^*)-(f,\xi_n^*)_r=\lambda_n^*(v,\xi_n^*)_r-(f,\xi_n^*)_r,\]
    which further implies that:
    \[(\lambda_n^*-\lambda)(v,\xi_n^*)_r=(f,\xi_n^*)_r.\]
    Inserting this into \eqref{eq:norm-w-in-basis} leads to the first inequality in the conclusion.

    To obtain the second inequality, we use that the norm of the operator $B^{-1}:L_r([0,R])\to H$ is at most $M$ and $v=B^{-1}(f+\lambda v)$. \qed

\section{Proofs of results in Section \ref{sec:FEM}}
\label{acaret:proofs-2}
\subsection*{Proof of Proposition \ref{prop:fem-interpolator}}
Proposition \ref{prop:wcos-inH2} implies that $v$ has a continuous representative on every compact subinterval of $(0,R]$. This fact, combined with Proposition \ref{prop:Hjr-is-zero-in-zero} in the case $j\geq 1$ and Proposition \ref{prop:h2-estimates} for $j=0$ implies that the values
\[v(r^h_0),v(r^h_1),v(r^h_2),\ldots,v(r^h_{N_h})\]
are well-defined. Therefore, we can define the linear interpolator $\tilde v_h \in W_{j,h}$:
\[\tilde v_h \coloneqq \sum_{i=0}^{N_h} v(r^h_i)\phi_i\]
and the error function $\tilde e_h\coloneqq v-\tilde v_h\in H_{j,r}([0,R])$. We note that $\tilde e_h(r^h_i)=0$, for every $i=\overline{0,N_h}$.
The fact that $\tilde v_h \in W_{j,h}^0$ if $v\in H_{j,r}^0([0,R])$ is guaranteed by Proposition \ref{prop:spaces-equivalence}.

In order to prove the $L_r([0,R])$ and $H_{j,r}([0,R])$ norm estimates for $\tilde e_h$, we write:
\begin{align}\label{eq:interpolation-lr-sum}\|\tilde e_h\|_{r}^2&=\sum_{i=0}^{N_h-1} \int_{r^h_{i}}^{r^h_{i+1}} |\tilde e_h(r)|^2\,r\, \dd r;\\
\label{eq:interpolation-hjr-sum}
a_j(\tilde e_h,\tilde e_h)&=\sum_{i=0}^{N_h-1} \int_{r^h_{i}}^{r^h_{i+1}} |(\tilde e_h)'(r)|^2\,r\, \dd r+j^2 \sum_{i=0}^{N_h-1} \int_{r^h_{i}}^{r^h_{i+1}} |\tilde e_h(r)|^2\frac{1}{r}\, \dd r.
\end{align}
Our aim is to bound every term above by one of the integrals in Proposition \ref{prop:h2-estimates}. We split our analysis in three cases:
\begin{enumerate}[label={\Roman*.}]
    \item \label{item:case1} $i\neq 0$;
    \item \label{item:case2} $i=0, j\geq 1$;
    \item \label{item:case3} $i=j=0$.
\end{enumerate}

\emph{Case \ref{item:case1}} We estimate the terms in \eqref{eq:interpolation-lr-sum} and \eqref{eq:interpolation-hjr-sum} corresponding to $i\neq 0$. Since, by Proposition \ref{prop:wcos-inH2}, $v\in H^2([r^h_i,r^h_{i+1}])$ and, by construction $\tilde v_h$ is an affine function on the interval $[r^h_i,r^h_{i+1}]$, we obtain that $\tilde e_h\in H^2([r^h_i,r^h_{i+1}])$. As a result, $\tilde e_h$ has a $C^1([r^h_i,r^h_{i+1}])$ representative and, since $\tilde e_h(r^h_i)=\tilde e_h(r^h_{i+1})=0$, there exists a point $\tilde r_i$ such that $(\tilde e_h)'(\tilde r_i)=0$. Therefore, we can write:
\begin{equation}\label{eq:interpolation-deriv-formula-outside-0}
(\tilde e_h)'(r)=\int_{\tilde r_i}^{r} (\tilde e_h)''(\tau)\dd\tau=\int_{\tilde r_i}^{r} v''(\tau)\dd\tau,\quad \forall r\in [r^h_i,r^h_{i+1}], \end{equation}
where the last equality is true, because $\tilde v_h$ is affine on $[r^h_i,r^h_{i+1}]$ and so it has null second derivative on this interval.
Using the Cauchy-Schwarz inequality in \eqref{eq:interpolation-deriv-formula-outside-0}, we obtain:
\begin{equation}\label{eq:interpolation-deriv-ineq-outside-0}
|(\tilde e_h)'(r)|^2\leq h \int_{r^h_i}^{r^h_{i+1}} |v''(\tau)|^2 \dd\tau
\end{equation}
It follows that:
\begin{equation}\label{eq:interpolation-hrj-first-part-outside-0}
\begin{aligned}\int_{r^h_i}^{r^h_{i+1}} |(\tilde e_h)'(r)|^2\, r\,\dd r\leq  \frac{r^h_{i+1}}{r^h_i} h^2\int_{r^h_i}^{r^h_{i+1}} |v''(\tau)|^2\tau\,\dd\tau&= \,\frac{i+1}{i}  h^2\int_{r^h_i}^{r^h_{i+1}} |v''(\tau)|^2\tau\,\dd\tau\\
&\leq 2 h^2 \int_{r^h_i}^{r^h_{i+1}} |v''(\tau)|^2\tau\,\dd\tau.
\end{aligned}
\end{equation}
The inequality \eqref{eq:interpolation-deriv-ineq-outside-0}, together with the fact that $\tilde e_h(r^h_i)=0$, implies that, for $r\in [r^h_i,r^h_{i+1}]$,
\begin{equation}\label{eq:interpolation-error-formula-outside-0}
    |\tilde e_h(r)|^2\leq \left|\int_{r^h_i}^r (\tilde e_h)'(\tau)\dd\tau\right|^2 \leq h^3 \int_{r^h_i}^{r^h_{i+1}} |v''(\tau)|^2\,\dd\tau.
\end{equation}
Similarly to \eqref{eq:interpolation-hrj-first-part-outside-0}, we obtain:
\begin{equation}\label{eq:interpolation-lr-outside-0}
    \int_{r^h_i}^{r^h_{i+1}} |\tilde e_h(r)|^2\, r\, \dd r\leq 2 h^4 \int_{r^h_i}^{r^h_{i+1}} |v''(\tau)|^2\tau\,\dd\tau.
\end{equation}

In order to estimate the terms $\int_{r^h_i}^{r^h_{i+1}} |\tilde e_h(r)|^2\,r \,\dd r$, we note that, by definition:
\[\tilde v_h(r)=\frac{r^h_{i+1}-r}{h}v(r^h_i)+ \frac{r-r^h_i}{h}v(r^h_{i+1}), \quad \forall r\in [r^h_i,r^h_{i+1}]\]
and, therefore,
\[\begin{aligned}\tilde e_h(r)&=v(r)-\frac{r^h_{i+1}-r}{h} r^h_i \frac{v(r^h_i)}{r^h_i}- \frac{r-r^h_i}{h}r^h_{i+1}\frac{v(r^h_{i+1})}{r^h_{i+1}}\\
&=\frac{r^h_{i+1}-r}{h} r^h_i\left[\frac{v(r)}{r}-\frac{v(r^h_i)}{r^h_i}\right]-\frac{r-r^h_i}{h}r^h_{i+1}\left[\frac{v(r^h_{i+1})}{r^h_{i+1}}-\frac{v(r)}{r}\right].
\end{aligned}
\]
This inequality implies that, for every $r\in [r^h_k,r^h_{k+1}]$,
\[|\tilde e_h(r)|\leq r^h_i \int_{r^h_i}^{r}\left|\left(\frac{v(\tau)}{\tau}\right)'\right|\dd\tau+ r^h_{i+1} \int_r^{r^h_{i+1}}\left|\left(\frac{v(\tau)}{\tau}\right)'\right|\dd\tau.
\]
Next, the Cauchy-Schwarz inequality leads to
\[\begin{aligned}\frac{|\tilde e_h(r)|^2}{r}&\leq 2\frac{|r^h_i|^2}{r}h\int_{r^h_i}^{r}\left|\left(\frac{v(\tau)}{\tau}\right)'\right|^2\dd\tau+2\frac{|r^h_{i+1}|^2}{r}h\int_r^{r^h_{i+1}}\left|\left(\frac{v(\tau)}{\tau}\right)'\right|^2\dd\tau\\
&\leq 2h\int_{r^h_i}^{r}\left|\left(\frac{v(\tau)}{\tau}\right)'\right|^2\tau\,\dd\tau+2\frac{|r^h_{i+1}|^2}{|r^h_i|^2}h\int_r^{r^h_{i+1}}\left|\left(\frac{v(\tau)}{\tau}\right)'\right|^2\tau \,\dd\tau\\
&= 2h\int_{r^h_i}^{r}\left[\frac{v'(\tau)}{\tau}-\frac{v(\tau)}{\tau^2}\right]^2\tau\,\dd\tau+2\frac{(i+1)^2}{i^2}h\int_r^{r^h_{i+1}}\left[\frac{v'(\tau)}{\tau}-\frac{v(\tau)}{\tau^2}\right]^2\tau \,\dd\tau \\
&\leq 8h\int_{r^h_i}^{r^h_{i+1}}\left[\frac{v'(\tau)}{\tau}-\frac{v(\tau)}{\tau^2}\right]^2\tau\,\dd\tau.
\end{aligned}\]
Integrating this inequality on $[r^h_i,r^h_{i+1}]$, we obtain:
\begin{equation}\label{eq:interpolation-hrj-second-part-outside-0}
    \int_{r^h_i}^{r^h_{i+1}} |\tilde e_h(r)|^2\frac{1}{r}\, \dd r \leq 8h^2\int_{r^h_i}^{r^h_{i+1}}\left[\frac{v'(\tau)}{\tau}-\frac{v(\tau)}{\tau^2}\right]^2\tau\,\dd\tau
\end{equation}
which concludes the estimation of the terms in \eqref{eq:interpolation-lr-sum} and \eqref{eq:interpolation-hjr-sum} corresponding to $i\geq 1$.

\emph{Case \ref{item:case2}} For the terms of \eqref{eq:interpolation-lr-sum} and \eqref{eq:interpolation-hjr-sum} in which $i=0$ (i.e., the integrals on $[r^h_0,r^h_1]=[0,h]$), we distinguish two cases, depending on whether $j$ is null or non-zero. If $j\geq 1$, for $r\in [r^h_0,r^h_1]=[0,h]$ we have $\tilde v_h(r)=\frac{r}{h}v(h)$, so:
\[\tilde e_h(r)=v(r)-\frac{r}{h}v(h)\quad \text{ and, thus, }\quad (\tilde e_h)'(r)=v'(r)-\frac{1}{h}v(h).\]
As a result, for $r\in [0,h]$,
\[\frac{\tilde e_h(r)}{\sqrt{r}}= \sqrt{r} \left(\frac{v(r)}{r}- \frac{v(h)}{h}\right)= -\sqrt{r}\int_r^h \left[\frac{v(\tau)}{\tau}\right]'\dd\tau = -\sqrt{r}\int_r^h \left[\frac{v'(\tau)}{\tau}-\frac{v(\tau)}{\tau^2}\right]\dd\tau. \]
Furthermore, the Cauchy-Schwarz inequality implies that:
\[\frac{|\tilde e_h(r)|^2}{r} \leq  (h-r) r\int_r^h \left[\frac{v'(\tau)}{\tau}-\frac{v(\tau)}{\tau^2}\right]^2\dd\tau\leq h \int_0^h\left[\frac{v'(\tau)}{\tau}-\frac{v(\tau)}{\tau^2}\right]^2\tau\,\dd\tau.\]
Integrating this inequality on $[0,h]$, one obtains:
\begin{equation}\label{eq:interpolation-hjr-second-term-near0-jgeq1}
    \int_0^h \frac{|\tilde e_h(r)|^2}{r}\dd r\leq h^2 \int_0^h\left[\frac{v'(\tau)}{\tau}-\frac{v(\tau)}{\tau^2}\right]^2\tau\,\dd\tau.
\end{equation}
which implies immediately that:
\begin{equation}\label{eq:interpolation-lr-near0-jgeq1}
    \int_0^h |\tilde e_h(r)|^2\, r\,\dd r\leq h^4 \int_0^h\left[\frac{v'(\tau)}{\tau}-\frac{v(\tau)}{\tau^2}\right]^2\tau\,\dd\tau.
\end{equation}
In order to estimate the $L_r([0,h])$ norm of the derivative of $\tilde e_h$, we note that:
\[(\tilde e_h)'(r)\sqrt{r}-\frac{\tilde e_h(r)}{\sqrt{r}}=v'(r)\sqrt{r}-\frac{v(r)}{\sqrt{r}}= r\left[\frac{v'(r)}{r}-\frac{v(r)}{r^2}\right]\sqrt{r}, \]
which implies that:
\[\int_0^h\left[ (\tilde e_h)'(r)\sqrt{r}-\frac{\tilde e_h(r)}{\sqrt{r}}\right]^2\dd r\leq h^2 \int_0^h \left[\frac{v'(r)}{r}-\frac{v(r)}{r^2}\right]^2r\, \dd r. \]
This estimate, together with the inequality \eqref{eq:interpolation-hjr-second-term-near0-jgeq1}, implies that:
\begin{equation}\label{eq:interpolation-hjr-first-term-near0-jgeq1}
    \int_0^h |(\tilde e_h)'(r)|^2 r\dd r\leq 4 h^2 \int_0^h\left[\frac{v'(\tau)}{\tau}-\frac{v(\tau)}{\tau^2}\right]^2\tau\,\dd\tau.
\end{equation}
This finishes the proof of the estimates on $[0,h]$, in the case $j\geq 1$. 

\emph{Case \ref{item:case3}} In the situation when $j=0$ and $r\in [r^h_0,r^h_1]=[0,h]$, we have:
\[\tilde e_h(r)=v(r)-\frac{h-r}{h}v(0)-\frac{r}{h} v(h)\quad \text{and, thus,}\quad (\tilde e_h)'(r)=v'(r)-\frac{1}{h}[v(h)-v(0)].\]

Therefore, we can write:
\[\tilde e_h(r)=\frac{h-r}{h}[v(r)-v(0)]+\frac{r}{h}[v(r)-v(h)]=\frac{h-r}{h}\int_0^r v'(\tau)\dd\tau-\frac{r}{h}\int_r^h v'(\tau)\dd\tau.\]
It follows by the Cauchy-Schwarz inequality that:
\[|\tilde e_h(r)|^2\leq \left|\int_0^h |v'(\tau)|\dd \tau\right|^2\leq \int_0^h \tau \, \dd \tau\int_0^h\frac{|v'(\tau)|^2}{\tau}\dd\tau= \frac{h^2}{2}\int_0^h\frac{|v'(\tau)|^2}{\tau}\dd\tau,\]
which leads to:
\begin{equation}
\label{eq:interpolation-lr-near-0-j0}
\int_0^h |\tilde e_h(r)|^2\, r\,\dd r\leq\frac{h^4}{4} \int_0^h\frac{|v'(\tau)|^2}{\tau}\dd\tau.
\end{equation}
We are left to analyse the $L_r([0,h])$ norm of the derivative of the error. For this purpose, we write:
\[\begin{aligned}
(\tilde e_h)'(r)=v'(r)- \int_0^h v'(\tau)\dd \tau &= v'(r)\frac{r}{h}-\frac{1}{h}\int_0^r v'(\tau)\dd\tau+v'(r)\frac{h-r}{h}-\frac{1}{h}\int_r^h v'(\tau)\dd\tau\\
&= \frac{v'(r)}{r}\frac{r^2}{h}-\frac{\sqrt{r}}{h}\int_0^r \frac{v'(\tau)}{\sqrt{r}}\dd\tau+\frac{1}{h}\int_r^h [v'(r)-v'(\tau)]\dd\tau\\
&=\frac{v'(r)}{r}\frac{r^2}{h}-\frac{\sqrt{r}}{h}\int_0^r \frac{v'(\tau)}{\sqrt{r}}\dd\tau-\frac{1}{h}\int_r^h\int_r^\tau v''(\sigma)\dd\sigma\,\dd\tau,
\end{aligned}
\]
which implies that:
\[|(\tilde e_h)'(r)\sqrt{r}|\leq h\frac{|v'(r)|}{\sqrt{r}}+\int_0^h \frac{|v'(\tau)|}{\sqrt{\tau}}\dd\tau+\frac{1}{h}\int_r^h \int_r^\tau |v''(\sigma)|\sqrt{\sigma}\,\dd \sigma\,\dd\tau.\]
The Cauchy-Schwarz inequality leads to:
\[|(\tilde e_h)'(r)|^2\,r\leq 3h^2\frac{|v'(r)|^2}{r}+3h\int_0^h \frac{|v'(\tau)|^2}{\tau}\dd\tau+3h \int_0^h |v''(\sigma)|^2\sigma\,\dd \sigma,\]
which, by integration on $[0,h]$ implies that:
\begin{equation}\label{eq:interpolation-hjr-near-0-j0}
    \int_0^h|(\tilde e_h)'(r)|^2\,r\leq 6h^2\int_0^h\frac{|v'(r)|^2}{r}\dd r+3h^2 \int_0^h |v''(\sigma)|^2\sigma\,\dd \sigma.
\end{equation}
The conclusion follows by inserting \eqref{eq:interpolation-hrj-first-part-outside-0}-\eqref{eq:interpolation-hjr-near-0-j0} into \eqref{eq:interpolation-lr-sum} and \eqref{eq:interpolation-hjr-sum} and then using Proposition \ref{prop:h2-estimates}.\qed

\subsection*{Proof of Lemma \ref{lem:projection-error}}
    We will prove the inequalities for the Dirichlet projection -- i.e. statement \ref{item:projection-D} --  and then we will outline the differences in reasoning for the Neumann projection \ref{item:projection-N} and the Dirichlet projection of $v-v(R)\phi_N$ in statement \ref{item:projection-D-modified}.

    Indeed, for $v\in H_{j,r}^0([0,R])$, elementary properties of the projection operator onto a closed subspace of a Hilbert space lead to:
    \begin{equation}\label{eq:projectior-minimizes}
    a_{j,V_h}(v-\Pi_h^0 v,v-\Pi_h^0 v)\leq a_{j,V_h}(v-w^h,v-w^h), \quad \forall w^h\in W_{j,h}^0.\end{equation}
    In particular, if $v(r)\cos(j\theta)\in H^2(B_O(R))$, we denote by
    \[e_h\coloneqq v-\Pi_h^0 v\in H_{j,r}^0([0,R])\]
    the projection error and we test \eqref{eq:projectior-minimizes} against the function $\tilde{v}_h$ provided by Proposition \ref{prop:fem-interpolator} to obtain:
    \begin{equation}\label{eq:projection-error-ajVh}
    a_{j,V_h}(e_h,e_h)\leq a_{j,V_h}(v-\tilde v_h,v-\tilde v_h)\leq h^2C(V_h) \|v(r)\cos(j\theta)\|_{H^2(B_O(R))}^2,
    \end{equation}
    where $C(V_h)>0$ is a constant that only depends on the suppremum norm of the discretised potential  $V_h$. Since, by \eqref{eq:error-V-Vh-C1}, $V+\frac{1}{2}\geq V_h\geq \frac{1}{2}$, it follows that:
    \begin{equation}\label{eq:projection-error-Hjr}\|v-\Pi_h^0 v\|_{j,r}\leq h\, C(V)\|v(r)\cos(j\theta)\|_{H^2(B_O(R))}.
    \end{equation}

    In order to obtain the $h^2$ estimate for the $L_r([0,R])$ norm of the projection error $e_h$, we apply a technique similar to Nitsche's trick in \cite[Theorem 4.9]{larson}: with the notation in Section \ref{sec:weaks-solutions-basis-functions}, let 
    \[\psi \coloneqq A_{j,V_h}^{-1}\, e_h,\]
    which means that $\psi\in H_{j,r}^0([0,R])$ satisfies:
    \begin{equation}\label{eq:nitsche-weak}
    a_{j,V_h}(\psi,w)=(e_h,w)_r, \quad \forall w\in H_{j,r}^0([0,R]).\end{equation}
    Similar to Proposition \ref{prop:spaces-equivalence} \ref{item:weak-second}, one has that the function $\psi(r)\cos(j\theta)$ satisfies in the weak sense the following equation with homogeneous Dirichlet boundary conditions:
    \[-\Delta[\psi(r)\cos(j\theta)]+V_h(\|\x\|)\psi(r)\cos(j\theta) = e_h(r)\cos(j\theta), \quad \x=(r\cos(\theta),r\sin(\theta))\in B_O(R).\]
    The regularity theory for elliptic equations with Dirichlet boundary conditions \cite[Theorem 8.12]{GilbargTrudinger2001} implies that there exists a constant $C(R,V)>0$ such that:
    \begin{equation}\label{eq:projection-regularity-D}
    \|\psi(r)\cos(j\theta)\|_{H^2(B_O(R))}\leq C(R,V) \|e_h(r)\cos(j\theta)\|_{L^2(B_O(R))}= C(R,V)\|e_h\|_r,
    \end{equation}
    where the last equality is due to Proposition \ref{prop:spaces-equivalence} \ref{item:l2-second}. Therefore, the same reasoning by which we obtained \eqref{eq:projection-error-ajVh} leads to:
    \begin{equation}\label{eq:projection-error-psi}
    a_{j,V_h}(\psi-\Pi_h^0 \psi,\psi-\Pi_h^0 \psi)\leq h^2 C(R,V)\|e_h\|_r^2.
    \end{equation}
    Next, testing \eqref{eq:nitsche-weak} against $e_h$ and using the definition of the projection operator $\Pi_h^0$ together with the Cauchy-Schwarz inequality, we obtain:
    \[\begin{aligned}\|e_h\|_r^2=a_{j,V_h}(\psi,e_h)=a_{j,V_h}(\psi-\Pi_h^0 \psi,e_h)&\leq \sqrt{a_{j,V_h}(\psi-\Pi_h^0 \psi,\psi-\Pi_h^0 \psi)\,a_{j,V_h}(e_h,e_h)}\\
    &\leq h^2C(V,R)\, \|e_h\|_r \,\|v(r)\cos(j\theta)\|,
    \end{aligned}\]
    where the last inequality follows by \eqref{eq:projection-error-ajVh} and \eqref{eq:projection-error-psi}. This leads to the  conclusion of the statement \ref{item:projection-D}.

    The proof of the estimates concerning the Neumann projection -- i.e. statement \ref{item:projection-N} -- are obtained similarly, with the following modifications: for the error function $e_h\coloneqq v-\Pi_hv$, with the notation in Section \ref{sec:stability-lambda} we define $\psi\coloneqq B_{j,V_h}^{-1} e_h \in H_{j,r}([0,R])$, which is equivalent to:
    \begin{equation}\label{eq:nitsche-weak-N}
    a_{j,V_h}(\psi,w)=(e_h,w)_r, \quad \forall w\in H_{j,r}([0,R]).\end{equation}
    As in the proof of Proposition \ref{prop:spaces-equivalence} \ref{item:weak-second}, the function $\psi(r)\cos(j\theta)$ satisfies the weak formulation of the homogeneous Neumann problem associated to the equation:
    \[-\Delta[\psi(r)\cos(j\theta)]+V_h(\|\x\|)\psi(r)\cos(j\theta) = e_h(r)\cos(j\theta), \quad \x=(r\cos(\theta),r\sin(\theta))\in B_O(R).\]
   Testing the weak formulation against the function $\psi(r)\cos(j\theta)$ itself and taking into account that $V_h\geq \frac 1 2$, we obtain that:
    \[\|\psi(r)\cos(j\theta)\|_{H^1(B_O(R))}\leq C(R,V) \|e_h\cos(j\theta)\|_{L^2(B_O(R))}\]
    Next, in order to obtain the estimate \eqref{eq:projection-regularity-D}, we combine the $H^1$ estimate above with the regularity theory for elliptic Neumann problems \cite[Theorem 4, p.~217]{mikhailov}. The rest of the proof of statement \ref{item:projection-N} is unchanged.

    To prove the estimate concerning the projection of $v-v(R)\phi_{N_h}$, that is statement \ref{item:projection-D-modified}, we denote the projection error \[e_h\coloneqq [v-v(R)\phi_{N_h}]-\Pi_h^0 [v-v(R)\phi_{N_h}]\] 
    and test the inequality \eqref{eq:projectior-minimizes} against $w^h\coloneqq \sum_{i=0}^{N_h-1} v(r^h_i)\phi_i\in W_{j,h}^0$ to obtain:
    \[\begin{aligned}
    a_{j,V_h} (e_h,e_h)&\leq a_{j,V_h}\left(v-v(R)\phi_{N_h} - \sum_{i=0}^{N_h-1} v(r^h_i)\phi_i\,,\,v-v(R)\phi_{N_h} - \sum_{i=0}^{N_h-1} v(r^h_i)\phi_i\right)\\
    &=a_{j,V_h}(v-\tilde v_h,v-\tilde v_h)\leq h^2C(V) \|v(r)\cos(j\theta)\|_{H^2(B_O(R))}^2,
    \end{aligned}\]
    by Proposition \ref{prop:fem-interpolator}. The rest of the proof of the statement \ref{item:projection-D-modified} uses the same reasoning (Nitsche's trick) as the proof of the statement \ref{item:projection-D}. \qed

\subsection*{Proof of Proposition \ref{prop:approximation-of-spectra}}
    This proof is inspired from \cite[Sections 7-8]{boffi2010}.
    We will outline the proof in the Dirichlet case, since the proof of the estimates for Neumann eigenvalues works exactly the same.

    First, Proposition \ref{prop:spaces-equivalence} implies that for every Dirichlet eigenpair $(\lambda^j_n,\xi^j_n)$ of $a_{j,V}$, the pair $(\lambda^j_n, \xi^j_n(r)\cos(j\theta))$ is a Dirichlet eigenpair of \eqref{eq:PDE-lambda} on $B_O(R)$. Therefore, the number $Q_{j,K}$ is smaller than or equal to $\tilde Q_K$, which is defined as the number of Dirichlet eigenvalues of \eqref{eq:PDE-lambda} on $B_O(R)$ that are less than or equal to $K$.
    
    Next, as in Lemma \ref{lem:projection-error}, we take $h_0< \frac{1}{2\|V'\|_{L^\infty([0,R])}}$ and further decrease this parameter $h_0$ such that, for every $h\in (0,h_0)$, the dimension of $W_{j,h}^0$ is greater than $\tilde Q_K$. Therefore, for every $n\in \overline{1,Q_{j,K}}$, the min-max characterisation for both continuous and discrete eigenvalues \cite[Proposition 7.2]{boffi2010} holds true:
    \begin{align}\label{eq:minmax-continuous}
    \lambda^j_n&=\min_{\substack{E\text{ subspace of }H_{j,r}^0([0,R])\\{\rm dim}(E)=n }}\quad\max_{w\in E\setminus\{0\}}\frac{a_{j,V}(w,w)}{\|w\|_r^2};\\
    \label{eq:minmax-discrete}
    \lambda^j_{n,h}&=\min_{\substack{E^h\text{ subspace of }W_{j,h}^0\\{\rm dim}(E^h)=n }}\quad\max_{w^h\in E^h\setminus\{0\}}\frac{a_{j,V_h}(w^h,w^h)}{\|w^h\|_r^2}.
    \end{align}
    An immediate consequence of these characterisations is that, for every $n\in \overline{1,Q_{j,K}}$,
    \[\lambda^j_n\leq \lambda^{j}_{n,h}.\]
    Furthermore, the proof of \cite[Proposition 7.2]{boffi2010} implies that the minimum in \eqref{eq:minmax-continuous} is attained for $E=E_n^j$, which is the span of the eigenvectors corresponding to the first $n$ eigenvalues of $a_{j,V}$ on $H_{j,r}^0([0,R])$. In order to prove an upper bound for the difference $\lambda^j_{n,h}-\lambda^j_n$, we aim to use $E^h=E^j_{n,h}\coloneqq \Pi_h^0 E^j_n$ as a test subspace in \eqref{eq:minmax-discrete}. However, we need to make sure that ${\rm dim}(E^j_{n,h})=n$, which is equivalent to $\Pi_h^0$ being injective on $E^j_n$.

    Indeed, let $(\xi^j_m)_{m\geq 1}$ an orthonormal Hilbert basis of $L_r([0,R])$ where $\xi^j_m$ is an eigenvector associated to $\lambda^j_m$. Then, Lemma \ref{lem:projection-error} and a regularity result similar to \eqref{eq:projection-regularity-D} imply that:
    \[\|\xi^j_m-\Pi_h^0\xi^j_m\|_r\leq h^2 C(R,V,K) \|\xi^j_m\|_r=h^2 C(R,V,K)\]
    Next, for every $w\in E^j_n$, we can write:
    \[w=\sum_{m=1}^n (w,\xi^j_m)_r\, \xi^j_m,\]
    so, the inequality between the arithmetic mean and the quadratic mean implies that:
    \[\left\|w-\Pi_h^0 w\right\|_r^2\leq n \sum_{m=1}^n (w,\xi^j_m)_r^2 \, \left\|\xi_m-\Pi_h^0 \xi^j_m\right\|_r^2 \leq h^4\, \tilde Q_K\, C(R,V,K)\sum_{m=1}^n (w,\xi^j_m)_r^2 = h^4\, C(R,V,K) \|w\|_r^2,\]from which we get:
    \[\|w-\Pi_h^0 w\|_r\leq h^2(R,V,K)\|w\|_r.\]
    As a result, for $h_0$ smaller than $\frac{1}{2 C(R,V,K)}$, the triangle inequality leads to:
    \begin{equation}\label{eq:projection-injective}
    \|\Pi_h^0 w\|_r\geq \|w\|_r-\|w-\Pi_h^0 w\|_r\geq [1-h^2\, C(R,V,K)]\|w\|_r\geq \frac{1}{2}\|w\|_r,\end{equation}
    which implies that $\Pi_h^0:E^j_n\to E^j_{n,h}$ is injective.

    Therefore, the estimate \eqref{eq:minmax-discrete} implies that:
    \begin{equation}\label{eq:lambdajih-upper-bound-1}
    \lambda^j_{n,h}\leq \max_{w\in E^j_n\setminus\{0\}} \frac{a_{j,V_h}(\Pi_h^0 w, \Pi_h^0 w)}{\|\Pi_h^0w\|_r^2}.\end{equation}
    Next, de definition of the projection operator $\Pi_h^0$, together with the Cauchy-Schwarz inequality leads to:
    \[a_{j,V_h}(\Pi_h^0 w , \Pi_h^0 w)=a_{j,V_h}(w , \Pi_h^0 w)\leq \sqrt{a_{j,V_h}(w , w)a_{j,V_h}(\Pi_h^0 w , \Pi_h^0 w)},\]
    from which we deduce that:
    \[a_{j,V_h}(\Pi_h^0 w, \Pi_h^0 w)\leq a_{j,V_h}(w,w)=a_{j,V}(w,w)+([V_h-V]\, w,w)_r.\]
    Combining this with the estimates \eqref{eq:projection-injective} and \eqref{eq:lambdajih-upper-bound-1}, we arrive at:
    \[
    \begin{aligned}
\lambda^j_{n,h}&\leq \max_{w\in E^j_n\setminus\{0\}} \frac{a_{j,V}(w,w)+\|V-V_h\|_{L^\infty([0,R])}\|w\|_r^2}{\|w\|_r^2}\frac{\|w\|_r^2}{\|\Pi_h^0 w\|_r^2}\\
&\leq \left(\frac{1}{1-h^2\, C(R,V,K)}\right)^2 \left[\|V-V_h\|_{L^\infty([0,R])}+\max_{w\in E^j_n\setminus\{0\}} \frac{a_{j,V}(w,w)}{\|w\|_r^2}\right]\\
&\leq \left(\frac{1}{1-h^2\, C(R,V,K)}\right)^2 \lambda^j_n + 4 \|V-V_h\|_{L^\infty([0,R])},
\end{aligned}    
    \]
    where the last line was obtained by recalling that the proof of \cite[Proposition 7.2]{boffi2010} implies that the minimum $\lambda^j_n$ is attained in \eqref{eq:minmax-continuous} for $E=E^j_n$. The conclusion follows by \eqref{eq:error-V-Vh-C1}-\eqref{eq:error-V-Vh-C2}. \qed

\section{Proofs of results in Section \ref{sec:numerical-integration}}
\label{acaret:proofs-3}
\subsection*{Proof of Proposition \ref{prop:boundary-int}}
Corollary \ref{cor:error-UJ-UJh} and Minkowski's inequality imply that:
\[\left|S_{N_{\partial\Omega}}(u^{J,\lambda}_{\boldsymbol{\alpha},h})-S_{N_{\partial\Omega}}(u^{J,\lambda}_{\boldsymbol{\alpha}})\right|\leq h\, C(\Omega,V,K,J) \|u^{J,\lambda}_{\boldsymbol{\alpha}}\|_{L^2(\Omega)}\]
Therefore, to obtain the conclusion, it suffices to prove the estimate:
\begin{equation}\label{eq:border-int-claim}
\left|S_{N_{\partial\Omega}}(u^{J,\lambda}_{\boldsymbol{\alpha}})-\|u^{J,\lambda}_{\boldsymbol{\alpha}}\|_{L^2(\partial\Omega)}\right|\leq \frac{C(\Omega,V,K,J)}{N_{\partial\Omega}}\,\|u^{J,\lambda}_{\boldsymbol{\alpha}}\|_{L^2(\Omega)}.
\end{equation}
Indeed, we extend $\gamma$ by periodicity on the real line and this allows us to write, for a function $u\in H^1(\partial\Omega)$:
\[S_{N_{\partial\Omega}}(u)=\sqrt{\sum_{i=1}^{N_{\partial \Omega}} \int_{t_{i-1}}^{t_i} |u(\gamma(t_i))|^2\|\gamma'(t_i)\|\dd t}\]
and, on the other hand,
\[\|u\|_{L^2(\partial\Omega)}=\sqrt{\sum_{i=1}^{N_{\partial \Omega}} \int_{t_{i-1}}^{t_i}|u(\gamma(t))|^2\|\gamma'(t)\|\dd t }.\]
Therefore, Minkovski's inequality implies that:
\begin{equation}\label{eq:border-int-minkovski}
\left|S_{N_{\partial\Omega}}(u)-\|u\|_{L^2(\partial\Omega)}\right|\leq \sqrt{\sum_{i=1}^{N_{\partial \Omega}} \int_{t_{i-1}}^{t_i}\left||u(\gamma(t_i))|\|\gamma'(t_i)\|^\frac12-|u(\gamma(t))|\|\gamma'(t)\|^\frac12 \right|^2\dd t}.\end{equation}
Next, since $\gamma$ is a non-degenerate parametrisation of $\Omega$, the triangle inequality implies that:
\[\begin{aligned}\left||u(\gamma(t_i))|\|\gamma'(t_i)\|^\frac12-|u(\gamma(t))|\|\gamma'(t)\|^\frac12 \right| &\leq C(\Omega)\left[|u(\gamma(t_i))-u(\gamma(t))|+|u(\gamma(t))|(t_i-t_{i-1})\right]\\
& \leq C(\Omega)\left[\int_{t}^{t_i}\left|\frac{d}{ds}u(\gamma(s))\right|\dd s+\frac{1}{N_{\partial\Omega}}|u(\gamma(t))|\right]\\
& \leq C(\Omega)\left[\int_{t_{i-1}}^{t_i}\left|\frac{d}{ds}u(\gamma(s))\right|\dd s+\frac{1}{N_{\partial\Omega}}|u(\gamma(t))|\right].
\end{aligned}\]
The Cauchy-Schwarz inequality and the fact that $\gamma$ is non-degenerate imply that:
\[\begin{aligned}
\left||u(\gamma(t_i))|\|\gamma'(t_i)\|^\frac12-|u(\gamma(t))|\|\gamma'(t)\|^\frac12 \right|^2 \leq C(\Omega)\left[ |t_{i}-t_{i-1}|\int_{t_{i-1}}^{t_i}\left|\frac{d}{ds}u(\gamma(s))\right|^2\dd s +\frac{1}{N_{\partial\Omega}^2}|u(\gamma(t))|^2\right] \phantom{.}\\
\leq C(\Omega)\left[\frac{1}{N_{\partial\Omega}}\int_{t_{i-1}}^{t_i}\Bigg|\underbrace{\frac{d}{ds}u(\gamma(s))\frac{1}{\|\gamma'(s)\|}}_{\text{weak derivative of $u$ on }\partial \Omega}\Bigg|^2\|\gamma'(s)\|\dd s +\frac{1}{N_{\partial\Omega}^2}|u(\gamma(t))|^2\|\gamma'(t)\|\right].
\end{aligned}\]
Inserting this into \eqref{eq:border-int-minkovski}, we obtain that:
\begin{equation}\label{eq:border-int-error-H1}
\left|S_{N_{\partial\Omega}}(u)-\|u\|_{L^2(\partial\Omega)}\right|\leq \frac{C(\Omega)}{{N_{\partial\Omega}}}\|u\|_{H^1(\partial\Omega)}.\end{equation}
Next, the theory of traces (see \cite[p.~316]{brezis}) and Proposition \ref{prop:H2-stability-J} implies that:
\[\|u^{J,\lambda}_{\boldsymbol{\alpha}}\|_{H^1(\partial\Omega)}\leq \|u^{J,\lambda}_{\boldsymbol{\alpha}}\|_{H^\frac32(\partial\Omega)}\leq C(\Omega) \|u^{J,\lambda}_{\boldsymbol{\alpha}}\|_{H^2(\Omega)}\leq C(\Omega,V,K,J) \|u^{J,\lambda}_{\boldsymbol{\alpha}}\|_{L^2(B_O(\tilde R))}\]
Eventually, using \eqref{eq:border-int-error-H1} for $u=u^{J,\lambda}_{\boldsymbol{\alpha}}$, we obtain \eqref{eq:border-int-claim}, which leads to the conclusion.\qed

\subsection*{Proof of Proposition \ref{prop:Monte-Carlo}}
    For every integers $j,k\in \overline{0,J}$, we define the Monte-Carlo approximation of the $L^2(\Omega)$ scalar product of $\ujh^{\lambda}(r)\cos(j\theta)$ and $\mathbf{u}_{k,h}^{\lambda}(r)\cos(k\theta)$:
    \[
    \mathcal{S}_{N_\Omega} \left(\ujh^{\lambda}(r)\cos(j\theta),\mathbf{u}_{k,h}^{\lambda}(r)\cos(k\theta)\right)\coloneqq \sum_{i=1}^N \ujh^{\lambda}(\boldsymbol{r}_i)\cos(j\boldsymbol{\theta}_i)\,\mathbf{u}_{k,h}^{\lambda}(\boldsymbol{r}_i)\cos(k\boldsymbol{\theta}_i),
    \]
    where $\x_i=(\boldsymbol{r}_i\cos(\boldsymbol{\theta}_i),\boldsymbol{r}_i\sin(\boldsymbol{\theta}_i))$, $i=\overline{1,N_\Omega}$. Then we define the event
    \[
    \mathcal{E}_{j,k,N_\Omega,\eta} \coloneqq \text{"}\left|\mathcal{S}_{N_\Omega} \left(\ujh^{\lambda}(r)\cos(j\theta),\mathbf{u}_{k,h}^{\lambda}(r)\cos(k\theta)\right) - \left(\ujh^{\lambda}(r)\cos(j\theta),\mathbf{u}_{k,h}^{\lambda}(r)\cos(k\theta)\right)_{L^2(\Omega)}\right|<\eta\text{ "}\]
    Since $\x_i$ are i.i.d. uniformly in $\Omega$, it means that:
    \[\mathbb{E}\left[\ujh^{\lambda}(\boldsymbol{r}_1)\cos(j\boldsymbol{\theta}_1)\,\mathbf{u}_{k,h}^{\lambda}(\boldsymbol{r}_1)\cos(k\boldsymbol{\theta}_1)\right]=\frac{1}{|\Omega|}\left(\ujh^{\lambda}(r)\cos(j\theta),\mathbf{u}_{k,h}^{\lambda}(r)\cos(k\theta)\right)_{L^2(\Omega)}\]

    The Law of Large Numbers \cite[Theorem 2.1]{ross2019probability} implies that:
    \begin{equation}\label{eq:MC-lln}
    \mathbb{P}( \mathcal{E}_{j,k,N_\Omega,\eta})\geq 1-\frac{|\Omega|^2\,{\rm Var}\left(\ujh^{\lambda}(\boldsymbol{r}_1)\cos(j\boldsymbol{\theta}_1)\,\mathbf{u}_{k,h}^{\lambda}(\boldsymbol{r}_1)\cos(k\boldsymbol{\theta}_1)\right)}{\eta^2\, N_\Omega}.\end{equation}
    Then, we apply the Cauchy-Schwarz inequality to obtain:
    \[\begin{aligned}{\rm Var}\left(\ujh^{\lambda}(\boldsymbol{r}_1)\cos(j\boldsymbol{\theta}_1)\,\mathbf{u}_{k,h}^{\lambda}(\boldsymbol{r}_1)\cos(k\boldsymbol{\theta}_1)\right)&\leq \mathbb{E}\left[\left|\ujh^{\lambda}(\boldsymbol{r}_1)\cos(j\boldsymbol{\theta}_1)\,\mathbf{u}_{k,h}^{\lambda}(\boldsymbol{r}_1)\cos(k\boldsymbol{\theta}_1)\right|^2\right]\\
    &=\frac{1}{|\Omega|}\left\|\ujh^{\lambda}(r)\cos(j\theta)\,\mathbf{u}_{k,h}^{\lambda}(r)\cos(k\theta)\right\|_{L^2(\Omega)}^2\\
    &\leq \frac{1}{|\Omega|}\|\ujh^{\lambda}(r)\cos(j\theta)\|_{L^4(\Omega)}^2\|\mathbf{u}_{k,h}^{\lambda}(r)\cos(k\theta)\|_{L^4(\Omega)}^2
    \end{aligned}\]
    The Sobolev Embedding inequality \cite[Theorem 4.12]{adams2003sobolev}, Theorem \ref{thm:ujh-uj-convergence} and Proposition \ref{prop:H2-stability-J} further imply that:
    \[\begin{aligned}\|\ujh^{\lambda}(r)\cos(j\theta)\|_{L^4(\Omega)}
    \leq\|\ujh^{\lambda}(r)\cos(j\theta)\|_{L^4(B_O(R))}&\leq C(R)\|\ujh^{\lambda}(r)\cos(j\theta)\|_{H^1(B_O(R))}\\
    &\leq C(R,\tilde R,V,K,J) \|\uj^{\lambda}(r)\cos(j\theta)\|_{L^2(B_O(R))}\\
    &= C(R,\tilde R,V,K,J),
    \end{aligned}\]
    where the last equality follows since $\|\uj^{\lambda}\|=1$. 
    Therefore, we obtain that:
    \[{\rm Var}\left(\ujh^{\lambda}(\boldsymbol{r}_1)\cos(j\boldsymbol{\theta}_1)\,\mathbf{u}_{k,h}^{\lambda}(\boldsymbol{r}_1)\cos(k\boldsymbol{\theta}_1)\right)\leq C(\Omega,V,K,J),\]
    which by \eqref{eq:MC-lln} implies that:
    \begin{equation}\label{eq:MC-probability-1}\mathbb{P}(\mathcal E_{j,k,N_\Omega,\eta})\geq 1-\frac{C(\Omega, V,K,J)}{\eta^2\,N_{\Omega}},\quad \forall j,k\in \overline{0,J}.\end{equation}
    We emphasize that the above estimates remain valid if the cosine function is replaced by the sine function in either one or both of the terms $\cos(j\theta)$ and $\cos(k\theta)$.

    Next, let us denote $\tilde{\mathcal{E}}_{N_\Omega,\eta}$ the intersection of all the events  $\mathcal E_{j,k,N_\Omega,\eta}$, $j,k\in \overline{0,J}$, together with all the equivalent ones with ``$\cos$'' replaced by ``$\sin$''. Then, \eqref{eq:MC-probability-1} leads to:
    \begin{equation}\label{eq:MC-the-great-event}
    \mathbb{P}(\tilde{\mathcal{E}}_{N_\Omega,\eta})\geq 1-(2J+1)^2\frac{C(\Omega,V,K,J)}{\eta^2\,N_{\Omega}}.
    \end{equation}
    We will show that, for a suitable choice of constants, the event $\tilde{\mathcal{E}}_{N_\Omega,\eta}$ implies $\mathcal{E}_{N_\Omega,\eta}$. Indeed, \eqref{eq:defUJh} leads to:
     \[\begin{aligned}\mathcal{S}_{N_\Omega} (u^{J,\lambda}_{\boldsymbol{\alpha},h})^2 &- \| u^{J,\lambda}_{\boldsymbol{\alpha},h}\|_{L^2(\Omega)}^2\\
     =&\sum_{j,k=0}^{J}\alpha_j^c\alpha_k^c\left[\sum_{i=1}^{N_\Omega} \ujh^{\lambda}(\boldsymbol{r}_i)\cos(j\boldsymbol{\theta}_i)\,\mathbf{u}_{k,h}^{\lambda}(\boldsymbol{r}_i)\cos(k\boldsymbol{\theta}_i) -\left(\ujh^{\lambda}(r)\cos(j\theta),\mathbf{u}_{k,h}^{\lambda}(r)\cos(k\theta)\right)_{L^2(\Omega)}\right]
     \\&+\sum_{j,k=1}^{J}\alpha_j^s\alpha_k^s\left[\sum_{i=1}^{N_\Omega} \ujh^{\lambda}(\boldsymbol{r}_i)\sin(j\boldsymbol{\theta}_i)\,\mathbf{u}_{k,h}^{\lambda}(\boldsymbol{r}_i)\sin(k\boldsymbol{\theta}_i) -\left(\ujh^{\lambda}(r)\sin(j\theta),\mathbf{u}_{k,h}^{\lambda}(r)\sin(k\theta)\right)_{L^2(\Omega)}\right]\\
     &+2\sum_{\substack{j=\overline{0,J}\\k=\overline{1,J}}}\alpha_j^c\alpha_k^s\left[\sum_{i=1}^{N_\Omega} \ujh^{\lambda}(\boldsymbol{r}_i)\cos(j\boldsymbol{\theta}_i)\,\mathbf{u}_{k,h}^{\lambda}(\boldsymbol{r}_i)\sin(k\boldsymbol{\theta}_i) -\left(\ujh^{\lambda}(r)\cos(j\theta),\mathbf{u}_{k,h}^{\lambda}(r)\sin(k\theta)\right)_{L^2(\Omega)}\right].
     \end{aligned}
     \]
     If we assume that $\tilde{\mathcal{E}}_{N_\Omega,\eta}$ happens, then the absolute value  of all the terms in square brackets above is smaller than $\eta$. As a result, the inequality between the arithmetic mean and the quadratic mean leads to:
     \[\left|\mathcal{S}_{N_\Omega} ( u^{J,\lambda}_{\boldsymbol{\alpha},h})^2 - \|u^{J,\lambda}_{\boldsymbol{\alpha},h}\|_{L^2(\Omega)}^2\right|\leq \eta \left(|\alpha_0^c|+\sum_{j=1}^J(|\alpha_j^c|+|\alpha_j^s|)\right)^2\leq \eta\, (2J+1)\,\|u^{J,\lambda}_{\boldsymbol{\alpha}}\|_{L^2(B_O(R))}^2,\]
     which, by Proposition \ref{prop:H2-stability-J} implies that:
     \[\left|\mathcal{S}_{N_\Omega} (u^{J,\lambda}_{\boldsymbol{\alpha},h})^2 - \| u^{J,\lambda}_{\boldsymbol{\alpha},h}\|_{L^2(\Omega)}^2\right|\leq\eta\, C(R,\tilde R,V,K,J)\,\|u^{J,\lambda}_{\boldsymbol{\alpha}}\|_{L^2(\Omega)}^2 \]
     Moreover, Corollary \ref{cor:error-UJ-UJh} leads to:
     \[\left|\|u^{J,\lambda}_{\boldsymbol{\alpha},h}\|_{L^2(\Omega)}^2-\|u^{J,\lambda}_{\boldsymbol{\alpha}}\|_{L^2(\Omega)}^2\right|\leq h C(R,\tilde R,V,K,J)\|u^{J,\lambda}_{\boldsymbol{\alpha}}\|_{L^2(\Omega)}^2.\]
     Eventually, the last two inequalities imply that:
     \[\left|\mathcal{S}_{N_\Omega} (u^{J,\lambda}_{\boldsymbol{\alpha},h}) - \|u^{J,\lambda}_{\boldsymbol{\alpha},h}\|_{L^2(\Omega)}\right|=
    \frac{\left|\mathcal{S}_{N_\Omega} (u^{J,\lambda}_{\boldsymbol{\alpha},h})^2 - \| u^{J,\lambda}_{\boldsymbol{\alpha},h}\|_{L^2(\Omega)}^2\right|}{\mathcal{S}_{N_\Omega} (u^{J,\lambda}_{\boldsymbol{\alpha},h}) + \|u^{J,\lambda}_{\boldsymbol{\alpha},h}\|_{L^2(\Omega)}}\leq (\eta+h)\,C(R,\tilde R,V,K,J)\| u^{J,\lambda}_{\boldsymbol{\alpha}}\|_{L^2(\Omega)},\]
    which, together with \eqref{eq:MC-the-great-event}, leads to the conclusion. \qed

\section{Technical lemmata}

\begin{lemma}
\label{lem:dense-in-h01} Let $R>0$ and $u\in H_0^1(B_O(R))$. Then for every $\eps>0$, there exists a function $\varphi^\eps\in C_c^{\infty}(B_O(R)\setminus \{O\})$ such that $\|\varphi^\eps-u\|_{H^1(B_O(R))}\leq \eps$.
\begin{proof}[Sketch of proof]
The proof is based on a cut-off idea found in \cite{CazacuFlynnLam}.
Since $u\in H_0^1(B_O(R))$, for every $\eps>0$, there exists $\tilde{\varphi}^\eps\in C_c^\infty(B_O(R))$ such that $\|\tilde \varphi^\eps-u\|_{H^1(B_O(R))}\leq \frac \eps 3$. For a parameter $\delta\in(0,1)$ which we will fix later, let $\varphi^{\eps,\delta}\in H_0^1(B_O(R)\setminus B_O(\delta^2))$,
\[\varphi^{\eps,\delta}(\x)\coloneqq  \tilde \varphi^\eps(\x)\, \psi_\delta(\|\x\|),\]
where $\psi_\delta:[0,R]\to \RR$,
\[\psi_\delta(r)\coloneqq\begin{cases}
    0, &r<\delta^2;\\
    \frac{\log r-2\log \delta}{-\log \delta}, & r\in [\delta^2,\delta];\\
    1, &r>\delta.
\end{cases}\]

With this notation, 
\[\begin{aligned}\| \varphi^{\eps,\delta}-\tilde \varphi^\eps\|_{H^1(B_O(R))}^2&\leq \int_{B_O(R)} |\tilde \varphi^\eps(\x)|^2 [1-\psi_\delta(\|\x\|)]^2\dd \x+ 2 \int_{B_O(R)} |\nabla\tilde \varphi^\eps(\x)|^2 [1-\psi_\delta(\|\x\|)]^2\dd \x\\
&\quad + 2 \int_{B_O(R)} |\tilde \varphi^\eps(\x)|^2 [\nabla \psi_\delta(\|\x\|)]^2\dd \x.
\end{aligned}\]
The first two terms converge to zero as $\delta\to 0$ by dominated convergence. For the last term, we use the definition of $\psi_\delta$ and the fact that $\tilde \varphi^\eps \in C_c^\infty(B_O(R))$ to write:
\[\int_{B_O(R)} |\tilde \varphi^\eps(\x)|^2 [\nabla \psi_\delta(\|\x\|)]^2\dd \x\leq \|\tilde \varphi^\eps\|_{L^\infty(B_O(R))}\int_{\delta^2}^\delta \left|\frac{1}{r\log\delta}\right|^2\,r\, \dd r= \frac{\|\tilde \varphi^\eps\|_{L^\infty(B_O(R))}}{-\log \delta}\xrightarrow{\delta\to 0}0.\]
Therefore, we can choose $\delta\coloneqq\delta_\eps$ such that \[\|\varphi^{\eps,\delta_\eps}-\tilde \varphi^\eps\|_{H^1(B_O(R))}\leq \frac \eps 3.\]
Eventually, since $\varphi^{\eps,\delta_\eps}\in H_0^1(B_O(R)\setminus B_O(\delta_\eps^2))$, there exists a function $\varphi^\eps\in C_c^{\infty}(B_O(R)\setminus \{O\})$ such that 
    \[\|\varphi^\eps-\varphi^{\eps,\delta_\eps}\|_{H^1(B_O(R))}\leq \frac \eps 3\]
    and, therefore, 
    \[\|\varphi^\eps-u\|_{H^1(B_O(R))}\leq \eps,\]
    which finishes the proof.
\end{proof}
\end{lemma}
The following corollary can be easily obtained from the proof of Lemma \ref{lem:dense-in-h01}, considering $\tilde \varphi^\eps\in C^\infty\left(\overline{B_O(R)}\right)$, where $\overline{B_O(R)}$ is the closure of the ball of radius $R$. The existence of such an $\tilde \varphi^\eps$ is provided by \cite[Corollary 9.8]{brezis}.
\begin{corollary}\label{cor:density-h1}
    Let $R>0$ and $u\in H^1(B_O(R))$. Then for every $\eps>0$, there exists a function $\varphi^\eps\in C_c^{\infty}\left(\overline{B_O(R)}\setminus \{O\}\right)$ such that $\|\varphi^\eps-u\|_{H^1(B_O(R))}\leq \eps$.
\end{corollary}

\begin{lemma}\label{lem:fractional-norms}
    Let $R>0$ and $w\in H_{j,r}([0,R])\cap H^2(\mathcal{I})$ for every  compact subinterval $\mathcal{I}$ of $(0,R]$. Then, for every integer $j\geq 0$ and every $r>0$, 
    \begin{equation}\label{eq:fractional-spaces}
    [w(r)\cos(j\theta)]\Big\vert_{\partial B_O(r)}\in H^\frac{3}{2}(\partial B_O(r))\quad\text{ and }\quad\partial_\nu [w(r)\cos(j\theta)]\Big\vert_{\partial B_O(r)}\in H^\frac{1}{2}(\partial B_O(r)), \end{equation}
    where by $\partial_\nu$ we understand the exterior normal derivative with respect to $\partial B_O(r)$. Moreover, there exist two universal constants $0<c<C$ such that:
    \begin{equation}\label{eq:fractional-norm-H32} c |w(r)|\left(\sqrt{r}+\frac{j\sqrt{j}}{r}\right) \|w(r)\cos(j\theta)\|_{H^\frac{3}{2}(\partial B_O(r))}\leq C  |w(r)|\left(\sqrt{r}+\frac{j\sqrt{j}}{r}\right)  \end{equation}
    and 
    \begin{equation} \label{eq:fractional-norm-H12}
    c |w'(r)| \left(\sqrt{r}+ \sqrt{j}\right)\leq \|\partial_\nu w(r)\cos(j\theta)\|_{H^\frac{3}{2}(\partial B_O(r))}\leq C |w'(r)| \left(\sqrt{r}+ \sqrt{j}\right),\end{equation}
    where $w(r)$ and $w'(r)$ are understood in the sense of the $C^1$ representative of a $H^2(\mathcal{I})$ function.

    The statements \eqref{eq:fractional-spaces}-\eqref{eq:fractional-norm-H12} hold true if we replace $\cos(j\theta)$ with $\sin(j\theta)$, with $j>1$.
    
    \emph{Note:} In this paper, we follow \cite[p.~314]{brezis} for the definition of fractional Sobolev spaces.
\end{lemma}
\begin{proof}
Since $w$ and $w'$ have continuous representatives on $(0,R]$, then $[w(r)\cos(j\theta)]\Big\vert_{\partial B_O(r)}$ and $\partial_\nu [w(r)\cos(j\theta)]\Big\vert_{\partial B_O(r)}$ can be understood in the classical sense. Furthermore, the normal derivative on $\partial B_O(r)$ coincides with the radial derivative, thus:
\[\partial_\nu [w(r)\cos(j\theta)]\Big\vert_{\partial B_O(r)}=w'(r)\cos(j\theta).\]
Therefore, \eqref{eq:fractional-spaces} follows since $\cos(j\theta)$ is a smooth function on $\partial B_O(r)$. Moreover,
\[ \|w(r)\cos(j\theta)\|_{H^\frac{3}{2}(\partial B_O(r))}=|w(r)|\,\|\cos(j\theta)\|_{H^\frac{3}{2}(\partial B_O(r))}\]and
\[\|\partial_\nu w(r)\cos(j\theta)\|_{H^\frac{1}{2}(\partial B_O(r))}=|w'(r)|\,\|\cos(j\theta)\|_{H^\frac{1}{2}(\partial B_O(r))}.\]

It remains to estimate the fractional Sobolev norms of $\cos(j\theta)$ on $\partial B_O(r)$.  First, for $j=0$, the inequalities \eqref{eq:fractional-norm-H32} and \eqref{eq:fractional-norm-H12} are trivially satisfied.
For $j\geq 1$, we consider the parametrisation of the boundary of $B_O(r)$, $\gamma:[0,2\pi]\to \partial B_O(r)$,
\[\gamma(\theta)=(r\cos(\theta),r\sin(\theta)).\]
Following the definition of the fractional Sobolev norm in \cite[p.~314]{brezis}, we write:
\[\|\cos(j\theta)\|_{H^\frac{1}{2}(\partial B_O(r))}^2=\|\cos(j\theta)\|_{L^2(\partial B_O(r))}^2+\|\cos(j\theta)\|_{\dot H^{\frac{1}{2}}(\partial B_O(r))}^2,\]
where
\[\|\cos(j\theta)\|_{L^2(\partial B_O(r))}^2=r \int_0^{2\pi} |\cos(j\theta)|^2\dd \theta=\pi r\]
and \[\|\cos(j\theta)\|_{\dot H^{\frac{1}{2}}(\partial B_O(r))}^2=
 r^2\int_0^{2\pi}\int_0^{2\pi} \frac{|\cos(j\theta)-\cos(j\sigma)|^2}{|\gamma(\theta)-\gamma(\sigma)|^2} \,\dd \theta\, \dd\sigma.\]
 By direct computation, $|\gamma(\theta)-\gamma(\sigma)|^2=2r^2[1-\cos(\theta-\sigma)]$, so the change of variables $\tau=\theta-\sigma$ and the sum to product formula  lead to:
 \[\begin{aligned}
     \|\cos(j\theta)\|_{\dot H^{\frac{1}{2}}(\partial B_O(r))}^2 & =  \frac{1}{4}\int_0^{2\pi}\int_0^{2\pi} \frac{|\cos(j(\sigma+\tau))-\cos(j\sigma)|^2}{1-\cos(\tau)} \,\dd \tau\, \dd\sigma\\
     & =  \int_0^{2\pi}\int_{0}^{2\pi} \frac{\left|\sin\left(\frac{j\tau}{2}\right)\right|^2}{1-\cos(\tau)}\left|\sin\left(j\sigma+\frac{j\tau}{2}\right)\right|^2 \,\dd \tau\, \dd\sigma \\
     &=\int_0^{2\pi} \frac{\left|\sin\left(\frac{j\tau}{2}\right)\right|^2}{1-\cos(\tau)}  \int_0^{2\pi} \left|\sin\left(j\sigma+\frac{j\tau}{2}\right)\right|^2 \dd\sigma \dd\tau.
 \end{aligned}\]
Using the $2\pi$-periodicity of the function $\sigma \to \sin(j\sigma)$,we obtain: 
 \[\begin{aligned}
     \|\cos(j\theta)\|_{\dot H^{\frac{1}{2}}(\partial B_O(r))}^2 & =  
    \pi\int_0^{2\pi}  \frac{\left|\sin\left(\frac{j\tau}{2}\right)\right|^2}{1-\cos(\tau)} \dd\tau=2\pi \int_0^{\pi}  \frac{\left|\sin\left(\frac{j\tau}{2}\right)\right|^2}{1-\cos(\tau)} \dd\tau,
 \end{aligned}\]
where the last equality was obtained via the change of variables $\tau\leftarrow 2\pi-\tau$. Next, Since $\frac{1-\cos(\tau)}{\tau^2}\in \left[\frac{2}{\pi^2},\frac{1}{2}\right]$, for every $\tau\in [0,\pi]$, we obtain:
 \begin{equation*}
 4\pi \int_0^\pi\frac{\left|\sin\left(\frac{j\tau}{2}\right)\right|^2}{\tau^2}   \dd \tau\leq  \|\cos(j\theta)\|_{\dot H^{\frac{1}{2}}(\partial B_O(r))}^2 \leq \pi^3 \int_0^\pi\frac{\left|\sin\left(\frac{j\tau}{2}\right)\right|^2}{\tau^2}   \dd \tau.\end{equation*}
Next, the change of variables $z=\frac{j\tau}{2}$ leads to:
 \begin{equation}\label{eq:H12-norm-between-integrals-preliminary}
 2\pi j \int_0^{\frac{\pi j}{2}}\frac{\left|\sin\left(z\right)\right|^2}{z^2}   \dd z\leq  \|\cos(j\theta)\|_{\dot H^{\frac{1}{2}}(\partial B_O(r))}^2 \leq \frac{\pi^3 j}{2} \int_0^{\frac{\pi j}{2}}\frac{\left|\sin\left(z\right)\right|^2}{z^2}   \dd z  \dd z,\end{equation}
Consequently, we obtain:
  \begin{equation}\label{eq:H12-norm-between-integrals}
 0< 2\pi j \int_0^{\frac \pi 2}\frac{\left|\sin\left(z\right)\right|^2}{z^2}   \dd z \leq  \|\cos(j\theta)\|_{\dot H^{\frac{1}{2}}(\partial B_O(r))}^2  \leq \frac{\pi^3j}{2}\int_0^{\infty}\frac{\left|\sin\left(z\right)\right|^2}{z^2}   \dd z<+\infty,\end{equation}
 which concludes the proof of \eqref{eq:fractional-norm-H12}.
 
 In order to prove \eqref{eq:fractional-norm-H32}, we take into account that the derivative of the function $\cos(j\theta)$ with respect to the arc length $|\gamma'|=r$ on $\partial B_O(r)$ equals to $-\frac{j}{r}\sin(j\theta)$ (refer also to the formula of the Riemannian gradient in coordinates, e.g. \cite[Section 3.3]{grigoryan2009heat}). Therefore, by the definition of the fractional Sobolev spaces in \cite[p.~314]{brezis},  
 \[\begin{aligned}\|\cos(j\theta)\|_{H^\frac{3}{2}(\partial B_O(r))}^2&=\|\cos(j\theta)\|_{L^2(\partial B_O(r))}^2+\frac{j^2}{r^2}\left\|\sin(j\theta)\right\|_{L^2(\partial B_O(r))}^2+\frac{j^2}{r^2}\left\|\sin(j\theta)\right\|_{\dot H^{\frac{1}{2}}(\partial B_O(r))}^2\\
 &= \pi r + \frac{j^2\pi}{r} + \frac{j^2}{r^2} \left\|\sin(j\theta)\right\|_{\dot H^{\frac{1}{2}}(\partial B_O(r))}^2.
 \end{aligned}\]
 The conclusion follows, since the change of variables $\theta \leftarrow \frac{\pi}{2j}-\theta$ implies  that $\left\|\sin(j\theta)\right\|_{\dot H^{\frac{1}{2}}(\partial B_O(r))}^2=\left\|\cos(j\theta)\right\|_{\dot H^{\frac{1}{2}}(\partial B_O(r))}^2$, which satisfies \eqref{eq:H12-norm-between-integrals}.
\end{proof}

\end{document}